\documentclass[12pt]{amsart}

\newcommand{\journal}[1]{}
\newcommand{\WithTorelli}[1]{#1}
\newcommand{\WithoutTorelli}[1]{}
\newcommand{\hide}[1]{}

\usepackage{amssymb}
\usepackage{amsbsy}
\usepackage{amscd}
\usepackage{amsmath}
\usepackage{amsthm}
\usepackage{geometry}

\numberwithin{equation}{section}

\theoremstyle{plain}
\newtheorem{thm}{Theorem}[section]
\newtheorem{prop}[thm]{Proposition}
\newtheorem{claim}[thm]{Claim}

\newtheorem{conj}[thm]{Conjecture}
\newtheorem{cor}[thm]{Corollary}
\newtheorem{lem}[thm]{Lemma}

\newtheorem{observation}[thm]{Observation}

\theoremstyle{definition}
\newtheorem{defi}[thm]{Definition}

\newtheorem{question}[thm]{Question}

\newtheorem{example}[thm]{Example}
\newtheorem{rem}[thm]{Remark}

\newcommand{\C}{{\mathcal C}}

\newcommand{\E}{{\mathcal E}}
\newcommand{\F}{{\mathcal F}}
\newcommand{\G}{{\mathcal G}}
\newcommand{\IC}{{\mathcal I}{\mathcal C}}

\newcommand{\K}{{\mathcal K}}
\newcommand{\BK}{{\mathcal B}{\mathcal K}}

\newcommand{\LB}{{\mathcal L}}
\newcommand{\M}{{\mathcal M}}

\newcommand{\PP}{{\mathbb P}}

\newcommand{\X}{{\mathcal X}}
\newcommand{\Y}{{\mathcal Y}}
\newcommand{\Z}{{\mathcal Z}}
\newcommand{\RealNumbers}{{\mathbb R}}
\newcommand{\Integers}{{\mathbb Z}}
\newcommand{\ComplexNumbers}{{\mathbb C}}
\newcommand{\RationalNumbers}{{\mathbb Q}}

\newcommand{\linsys}[1]{{\mid}#1{\mid}}

\newcommand{\RightArrowOf}[1]{\stackrel{#1}{\rightarrow}}

\newcommand{\LongRightArrowOf}[1]{\stackrel{#1}{\longrightarrow}}

\newcommand{\StructureSheaf}[1]{{\mathcal O}_{#1}}
\newcommand{\EndProof}{\hfill  $\Box$}
\newcommand{\restricted}[2]{#1_{\mid_{#2}}}

\newcommand{\rank}{{\rm rank}}

\newcommand{\Pic}{{\rm Pic}}
\newcommand{\Pef}{{\rm Pef}}
\newcommand{\Gram}[2]{
\left(\begin{array}{cc}
(#1,#1)&(#1,#2)
\\
(#1,#2)&(#2,#2)
\end{array}\right)
}
\renewcommand{\div}{{\rm div}}

\newcommand{\Ext}{{\rm Ext}}

\newcommand{\Hom}{{\rm Hom}}

\newcommand{\End}{{\rm End}}

\newcommand{\MV}{{\mathcal M}{\mathcal V}}
\newcommand{\Abs}[1]{|\!#1\!|}

\newcommand{\Wedge}[1]{\stackrel{#1}{\wedge}}

\newcommand{\Choose}[2]{\left(\!\!\begin{array}{c}#1\\#2\end{array}\!\!\right)}

\begin{document}
\title[Prime exceptional divisors]
{Prime exceptional divisors on holomorphic symplectic varieties
and monodromy-reflections}
\author{Eyal Markman}

\address{Department of Mathematics and Statistics, 
University of Massachusetts, Amherst, MA 01003}
\email{markman@math.umass.edu}

\subjclass[2010]{Primary 14D05, 14J60; Secondary 14J28, 53C26, 14C20}

\begin{abstract}
Let $X$ be a projective irreducible holomorphic symplectic
manifold. The second integral cohomology of $X$ is a lattice
with respect to the Beauville-Bogomolov pairing. 
A divisor $E$ on $X$ is called a {\em prime exceptional divisor}, 
if $E$ is reduced and irreducible and of negative Beauville-Bogomolov degree.

Let $E$ be a prime exceptional divisor on $X$. 
We first observe that associated to $E$ is 
a monodromy involution of the integral cohomology 
$H^*(X,\Integers)$, which acts on the second cohomology lattice
as the reflection by the cohomology class $[E]$ of $E$ 
(Theorem \ref{cor-introduction}). 

We then specialize to the case that
$X$ is deformation equivalent to the Hilbert scheme of length $n$ 
zero-dimensional subschemes of a $K3$ surface, $n\geq 2$.
We determine the set of classes of exceptional divisors on $X$
(Theorem  
\ref{conj-exceptional-line-bundles}).
This leads to a determination of the closure of the movable cone of $X$.
%
\end{abstract}

\maketitle

\centerline{\sf In memoriam Professor Masaki Maruyama}


\tableofcontents 
%
\section{Introduction}
An {\em irreducible holomorphic symplectic manifold} is a simply connected
compact K\"{a}hler manifold $X$, such that $H^0(X,\Omega^2_X)$ is generated
by an everywhere non-degenerate holomorphic two-form 
\cite{beauville,huybrects-basic-results}. 
The dimension of $X$ is even, say $2n$. 
The second integral cohomology of $X$ is a lattice
with respect to the Beauville-Bogomolov pairing \cite{beauville}. 
A divisor $E$ on $X$ is called a {\em prime exceptional divisor}
if $E$ is reduced and irreducible and of negative Beauville-Bogomolov degree
\cite{boucksom}.

%
\subsection{A prime exceptional divisor is monodromy-reflective}
When $\dim(X)=2$, then $X$ is a K\"{a}hler $K3$ surface. 
Let $E$ be a prime divisor of negative degree on $X$. Then $E$ is necessarily
a smooth rational curve. Its degree is thus $-2$. $E$ may be
contracted, resulting in a surface $Y$ with an ordinary double point
(\cite{BHPV}, Ch. III).
A class $x\in H^2(X,\Integers)$ is {\em primitive} if it is not  
a multiple of another integral class by an integer larger than $1$.
Let $c\in H^2(X,\Integers)$ be a primitive class of negative degree.
Then $c$ has degree $-2$ if and only if the reflection
$R_c:H^2(X,\RationalNumbers)\rightarrow H^2(X,\RationalNumbers)$, 
given by
\begin{equation}
\label{eq-reflection-R-c}
R_c(x) \ \ = \ \ x - \frac{2(x,c)}{(c,c)}c,
\end{equation}
has integral values, since the lattice $H^2(X,\Integers)$ is even and
unimodular.

Druel recently established the birational-contractibility of a 
prime exceptional divisor $E$ on a projective irreducible holomorphic
symplectic manifold $X$ of arbitrary dimension $2n$. 
There exists a sequence of flops of $X$, resulting in a 
projective irreducible holomorphic
symplectic manifold $X'$, and a projective birational morphism
$\pi:X'\rightarrow Y$ onto a normal projective variety $Y$,
such that the exceptional divisor $E'\subset X'$ of $\pi$ 
is the strict transform of $E$
(see \cite{druel} and Proposition \ref{prop-druel} below).
The result relies on the work of several authors, in particular,  
on Boucksom's work on the divisorial Zariski decomposition and on
recent results in the minimal model program \cite{boucksom,BCHM}. 

Let $E$ be a prime exceptional divisor on a projective irreducible holomorphic 
symplectic manifold $X$. 
Let $c$ be the class of $E$ in $H^2(X,\Integers)$ and 
consider the reflection 
$R_c:H^2(X,\RationalNumbers)\rightarrow H^2(X,\RationalNumbers)$, 
given by (\ref{eq-reflection-R-c}).
Building on Druel's result, we prove the following statement.

\begin{thm}
\label{cor-introduction}
The reflection $R_c$ is a monodromy operator.
In particular, $R_c$ is an integral isometry.
Furthermore, $c$ is either a primitive class or two times a 
primitive class.
\end{thm}

See Corollary \ref{cor-1} for a more detailed statement and a proof. 
The reflection $R_c$ arises as a monodromy operator as follows.
Let $Def(Y)$ be the Kuranishi deformation space of $Y$ and 
$\bar{\psi}:\Y\rightarrow Def(Y)$ the semi-universal family.
Then $Def(Y)$ is smooth, as is the fiber $Y_t$ of $\bar{\psi}$, 
over a generic point $t\in Def(Y)$,  
and the smooth fiber $Y_t$ is deformation equivalent to $X$ \cite{namikawa}. 
Let $U\subset Def(Y)$ be the complement of
the discriminant locus. $R_c$ is exhibited as a monodromy operator of
the local system $R^2_{\bar{\psi}_*}\Integers$ over $U$.

%
\subsection{Prime-exceptional divisors in the $K3^{[n]}$-type case}
Theorem \ref{cor-introduction} imposes 
a rather strong numerical condition on a class $c$ to be 
the class of a prime exceptional divisor.
We get, for example, the following Theorem.
Assume that $X$ is deformation equivalent to the
Hilbert scheme $S^{[n]}$ of length $n$ subschemes of a $K3$ surface $S$.
We will abbreviate this statement by saying that $X$ is of
{\em $K3^{[n]}$-type.} Assume that $n \geq 2$ and $X$ is projective.

\begin{thm}\label{thm-2}
Let $e$ be a primitive class in $H^2(X,\Integers)$, with negative
Beauville-Bogomolov degree $(e,e)<0$,
such that some integer multiple of $e$
is the class of an irreducible divisor $E$. Then
$(e,e)= -2$ or $(e,e)=2-2n$. 
If $(e,e)=2-2n$, then the class
$(e,\bullet)$ in $H^2(X,\Integers)^*$ is divisible by $n-1$.
\end{thm}

The Theorem is related to Theorem 22 in the beautiful paper 
\cite{hassett-tschinkel}. In particular, the case of fourfolds
is settled in that paper. 
The hypothesis that the divisor $E$ is irreducible is necessary.
There exist examples of pairs $(X,e)$, with $X$ of $K3^{[n]}$-type,
$e\in H^2(X,\Integers)$, such that $2e$ is effective, 
the reflection by $e$ is an integral reflection of $H^2(X,\Integers)$, 
but $2-2n<(e,e)<-2$
 (\cite{markman-constraints}, Example 4.8). 

\begin{defi}
\label{def-Mon-2}
An isometry $g$ of 
$H^2(X,\Integers)$ is called a {\em monodromy operator}, if there exists a 
family $\X \rightarrow T$ (which may depend on $g$) 
of irreducible holomorphic symplectic manifolds, having $X$ as a fiber
over a point $t_0\in T$, 
and such that $g$ belongs to the image of $\pi_1(T,t_0)$ under
the monodromy representation. 
The {\em monodromy group} $Mon^2(X)$ of $X$ is the subgroup 
of $O[H^2(X,\Integers)]$ generated by all the monodromy operators. 
\end{defi}

\begin{defi}
\label{def-monodromy-reflective}
\begin{enumerate}
\item
A class $e\in H^2(X,\Integers)$ is said to be {\em monodromy-reflective}, if
$e$ is primitive, and the reflection $R_e(x):=x-\frac{2(x,e)}{(e,e)}e$, 
with respect to the class $e$, 
belongs to $Mon^2(X)$.
\item
A line bundle $L$ is said to be {\em monodromy-reflective}, if 
the class $c_1(L)$ is.
\end{enumerate}
\end{defi}

Theorem \ref{thm-2}
is an immediate consequence of Theorem \ref{cor-introduction} and 
the following characterization of 
monodromy-reflective line bundles on $X$ of $K3^{[n]}$-type.

\begin{prop}
\label{prop-reflection-by-a-numerically-prime-exceptional-is-in-Mon}
Let $e\in H^2(X,\Integers)$ be a primitive class of
negative degree $(e,e)$. Then 
the reflection $R_e$ belongs to $Mon^2(X)$, if and only if 
$e$ has one of the following two properties.
\begin{enumerate}
\item 
$(e,e)=-2$, or
\item
$(e,e)=2-2n$, and $n-1$ divides the class 
$(e,\bullet)\in H^2(X,\Integers)^*.$
\end{enumerate}
\end{prop}

The proposition is proven in section \ref{sec-degrees-2-and-2-2n}.
A class $e\in H^{1,1}(X,\Integers)$ is said to be 
{\em $\RationalNumbers$-effective}, if some non-zero integer multiple of
$e$ is the class of an effective divisor. 
Examples of monodromy-reflective line bundles, 
which are not $\RationalNumbers$-effective, 
are exhibited in section \ref{sec-non-effective}. 

%
\subsection{A classification of monodromy-reflective line bundles}
\hspace{1ex}\\
When $X$ is a $K3$ surface, a monodromy-reflective line bundle 
$L$ has degree $-2$,
and precisely one of $L$ or $L^{-1}$ is isomorphic to 
$\StructureSheaf{X}(E)$, where $E$ is an effective divisor
(see \cite{BHPV}, chapter VIII, Proposition 3.6). 
If $\Pic(X)$ is cyclic
then $E$ is necessarily a smooth rational curve. 
We may deform to this case upon deforming 
the pair $(X,L)$ to a nearby deformation equivalent pair.

Monodromy-reflective line bundles of degree $2-2n$, 
over $X$ of $K3^{[n]}$-type, need 
not be $\RationalNumbers$-effective if $n> 1$. 
Whether the line bundle is or is not $\RationalNumbers$-effective
depends on a monodromy-invariant defined in Proposition
\ref{prop-definition-of-rs} below.
The definition depends on the following Theorem.

The topological $K$-group $K(S)$ of a $K3$ surface $S$, 
endowed with the {\em Mukai pairing} $(v,w):=-\chi(v^\vee\otimes w)$,
is called the {\em Mukai lattice}. $K(S)$ is a rank $24$ 
even unimodular lattice isometric to the orthogonal direct sum
$\widetilde{\Lambda}:=E_8(-1)^{\oplus 2}\oplus U^{\oplus 4}$,
where $E_8(-1)$ is the negative definite $E_8$ lattice and
$U$ is the rank $2$ lattice with Gram matrix
{\scriptsize
$\left(\begin{array}{cc}
0 & 1\\
1 & 0
\end{array}
\right)$}.

Let $\Lambda:=E_8(-1)^{\oplus 2}\oplus U^{\oplus 3}\oplus\Integers\delta$,
with $(\delta,\delta)=2-2n$. 
Then $H^2(X,\Integers)$ is isometric to $\Lambda$, 
for any $X$ of $K3^{[n]}$-type, $n>1$ \cite{beauville}.
Let $O(\Lambda,\widetilde{\Lambda})$ be the set of primitive isometric
embeddings $\iota:\Lambda\hookrightarrow\widetilde{\Lambda}$.
$O(\widetilde{\Lambda})$ acts on $O(\Lambda,\widetilde{\Lambda})$ by
compositions. If $n-1$ is a prime power, then 
$O(\Lambda,\widetilde{\Lambda})$ consists of a single 
$O(\widetilde{\Lambda})$-orbit. 
The Euler number $\eta:=\eta(n-1)$ is the number of distinct primes
$p_1, \dots, p_\eta$ in the prime factorization 
$n-1=p_1^{e_1}\cdots p_\eta^{e_\eta}$, with positive integers $e_i$.
For $n>2$, there are $2^{\eta-1}$ distinct 
$O(\widetilde{\Lambda})$-orbits in $O(\Lambda,\widetilde{\Lambda})$ 
(see \cite{oguiso} or \cite{markman-constraints}, Lemma 4.3).

\begin{thm}
\label{thm-a-natural-orbit-of-embeddings-of-H-2-in-Mukai-lattice}
(\cite{markman-constraints}, Theorem 1.10).
An irreducible holomorphic symplectic manifold $X$ of $K3^{[n]}$-type, 
$n\geq 2$, comes with a natural choice of an $O(\widetilde{\Lambda})$-orbit
of primitive isometric embeddings of $H^2(X,\Integers)$ 
in $\widetilde{\Lambda}$. 
This orbit is monodromy-invariant, i.e., 
$\iota:H^2(X,\Integers)\hookrightarrow \widetilde{\Lambda}$ belongs to
this orbit, if and only if $\iota\circ g$ does, for all $g\in Mon^2(X)$. 
\end{thm}

Let $S$ be a $K3$ surface, $H$ an ample line bundle on $S$,
and $v\in K(S)$ a primitive class satisfying
$(v,v)=2n-2$, $n\geq 2$.
Assume that $X:=M_H(v)$ is a smooth and compact moduli space of 
$H$-stable sheaves of class $v$. Then $X$ is of $K3^{[n]}$-type
and the orbit in the Theorem is that of Mukai's isometry
$\iota:H^2(M_H(v),\Integers)\rightarrow v^\perp$, where
$v^\perp\subset K(S)$ is the sub-lattice orthogonal to $v$
(see \cite{markman-constraints}, Theorem 1.14 or Theorem 
\ref{thm-item-orbit-of-inverse-of-Mukai-isom-is-natural} below). 
The monodromy-invariance, of the $O(K(S))$-orbit of Mukai's isometry, 
uniquely characterizes the orbit in the above Theorem,
for every $X$ of $K3^{[n]}$-type.

Let $X$ be of $K3^{[n]}$-type, $n>1$.
Let $I''_n(X)\subset H^2(X,\Integers)$ be the subset of 
monodromy-reflective classes of degree $2-2n$,
and $e$ a class in  $I''_n(X)$.
Choose a primitive isometric embedding 
$\iota:H^2(X,\Integers)\hookrightarrow \widetilde{\Lambda}$,
in the natural orbit of Theorem 
\ref{thm-a-natural-orbit-of-embeddings-of-H-2-in-Mukai-lattice}.
Choose a generator $v$ of the rank $1$ 
sub-lattice of $\widetilde{\Lambda}$ orthogonal to the image of
$\iota$. Then $(v,v)=2n-2$. Indeed, $(v,v)>0$,
since the signature of $\widetilde{\Lambda}$ is $(4,20)$, and 
$(v,v)$ is equal to the order of 
$(\Integers v)^*/\Integers v$, which is equal to the order of 
$H^2(X,\Integers)^*/H^2(X,\Integers)$, which is $2n-2$. 
Let $\rho$ be the positive integer, such that $(e+v)/\rho$
is an integral and primitive class in $\widetilde{\Lambda}$. 
Define the integer $\sigma$ similarly using $e-v$. 
Let $\div(e,\bullet)$ be the integer in $\{(e,e)/2,(e,e)\}$, 
such that the class $(e,\bullet)/\div(e,\bullet)$ 
is an integral and primitive class in $H^2(X,\Integers)^*$.
Given a rational number $m$, 
let $\F(m)$ be  the set of unordered pairs $\{r,s\}$ of
positive integers, such that $rs=m$ and $\gcd(r,s)=1$. 
If $m$ is not a positive integer, then $\F(m)$ is empty.
Set 
\[
\Sigma''_n \ \ := \ \ \F(n-1) \ \cup \ \F([n-1]/2) \ \cup \ \F([n-1]/4).
\]  
Note that $\Sigma''_n$ is a singleton if and only if $n=2$ or 
$n-1$ is an odd prime power. 

\begin{prop}
\label{prop-definition-of-rs}
If $\div(e,\bullet)=n-1$ and $n$ is even, set $\{r,s\}:=\{\rho,\sigma\}$.
Otherwise, set $\{r,s\}:=\{\frac{\rho}{2},\frac{\sigma}{2}\}$.
Then $\{r,s\}$ is a pair of relatively prime integers in $\Sigma''_n$,
and the function
\[
rs \ : \ I''_n(X) \ \ \ \longrightarrow \ \ \ \Sigma''_n,
\]
sending the class $e$ to the unordered pair $\{r,s\}$, is
monodromy-invariant. 
The function $rs$ is surjective, if $n\equiv 1$ modulo $8$, and its image is
$\F(n-1)\cup\F([n-1]/2)$ otherwise. 
\end{prop}

The proposition is proven in 
Lemmas \ref{lem-faithful-Mon-invariant-in-case-divisibility-2n-2} 
and \ref{lem-non-unimodular-rank-two-lattice}. 
A more conceptual definition of the monodromy-invariant
$rs$ is provided in the statements of these Lemmas.
The proof relies on the classification of the 
isometry classes of all possible pairs 
$(\widetilde{L},e)$, where $\widetilde{L}$ is the saturation in
$\widetilde{\Lambda}$ of the rank $2$ sub-lattice ${\rm span}\{e,v\}$.
The classification is summarized in the table following 
Lemma \ref{lemma-isometry-orbits-in-rank-2}.
We finally arrive at the classification of monodromy-reflective
line bundles.

\begin{prop}
\label{prop-introduction-Mon-2-orbit-is-determined-by-three-invariants}
Let $X$ be of $K3^{[n]}$-type and 
$L$ a monodromy-reflective line bundle. Set $e:=c_1(L)$. 
\begin{enumerate}
\item
\label{prop-item-Mon-orbit-in-degree-2-minus-2n}
If $(e,e)=2-2n$, then the $Mon^2(X)$-orbit of the class $e$ is 
determined by $\div(e,\bullet)$ and the value $rs(e)$.
\item
\label{prop-item-Mon-orbit-in-degree-minus-2}
If $(e,e)=-2$, then the $Mon^2(X)$-orbit of the class $e$ is 
determined by $\div(e,\bullet)$.
\end{enumerate}
\end{prop}

Part \ref{prop-item-Mon-orbit-in-degree-2-minus-2n} is proven in
Lemmas \ref{lem-faithful-Mon-invariant-in-case-divisibility-2n-2} 
and \ref{lem-non-unimodular-rank-two-lattice}. 
Part \ref{prop-item-Mon-orbit-in-degree-minus-2} is Lemma 8.9
in \cite{markman-monodromy-I}.  

%
\subsection{A numerical characterization of exceptional 
\WithoutTorelli{classes}
\WithTorelli{classes}}
\hspace{1ex}\\
Let $(X_1,L_1)$ and $(X_2,L_2)$ be two pairs as in
Proposition \ref{prop-introduction-Mon-2-orbit-is-determined-by-three-invariants}.
Set $e_i:=c_1(L_i)$. 
If $(e_i,e_i)=2-2n$, $\div(e_1,\bullet)=\div(e_2,\bullet)$, and
$rs(e_1)=rs(e_2)$, then 
\WithoutTorelli{the pairs $(X_1,e_1)$ and  $(X_2,e_2)$ are deformation 
equivalent (allowing $e_i$ to loose it Hodge type along the deformation),
by Proposition 
\ref{prop-introduction-Mon-2-orbit-is-determined-by-three-invariants} and 
Lemma \ref{lem-monodromy-invariants-and-deformation-equivalence}.
Furthermore, a deformation relating the two 
pairs and preserving the Hodge type exists as well, assuming
a version of the Torelli Theorem holds (see section 
\ref{sec-deformation-equivalence-and-torelli}). Finally,}
\WithTorelli{the pairs $(X_1,L_1)$ and $(X_2,L_2^\epsilon)$ are deformation 
equivalent (requiring $e_i$ to preserve its Hodge type along the deformation), 
for $\epsilon=1$ or $\epsilon=-1$, by Proposition
\ref{prop-introduction-Mon-2-orbit-is-determined-by-three-invariants} and 
Lemma \ref{lem-monodromy-invariants-and-deformation-equivalence} 
(the proof of the latter depends on the Torelli Theorem 
\cite{verbitsky}).
Furthermore,}
if $X_1$ is projective and 
$L_1^{\otimes k}\cong\StructureSheaf{X_1}(E_1)$, for some $k>0$ and 
a prime exceptional divisor $E_1$, 
then a generic small deformation $(X,L)$ of $(X_2,L_2)$
consists of $L$ satisfying 
$L^{\otimes d}\cong\StructureSheaf{X}(E)$, for a prime exceptional divisor $E$
and for $d=k$ or $d=-k$
(Proposition \ref{prop-main-question-on-deformation-equivalence}).
This leads us to the
\WithoutTorelli{conjectural} 
numerical characterization of exceptional line bundles described in this section. 

Let $X$ be an irreducible holomorphic symplectic manifold of $K3^{[n]}$-type, 
$n\geq 2$.
Let $L$ be a monodromy-reflective line bundle on $X$, $e:=c_1(L)$, 
and $R_e$ the reflection by $e$. 
$R_e$ preserves the Hodge structure, and so acts on 
$H^{1,1}(X)\cong H^1(X,T_X)$. 
The Kuranishi deformation space $Def(X)$ is an open neighborhood of $0$ in
$H^1(X,T_X)$, which may be chosen to be $R_e$ invariant. 
Hence, $R_e$ acts on $Def(X)$. 
The Local Kuranishi deformation space $Def(X,L)$, of the pair $(X,L)$,
is the smooth divisor $D_e\subset Def(X)$ of fixed points of $R_e$. 

\begin{defi}
\label{def-numerically-prime-exceptional}
Let $h\in H^2(X,\RealNumbers)$ be a K\"{a}hler class.
A line bundle $L\in \Pic(X)$ is called {\em numerically exceptional},
if its first Chern class $e:=c_1(L)$ is a primitive class 
in $H^2(X,\Integers)$, satisfying $(h,e)>0$ and the following properties.
The Beauville-Bogomolov degree is either $(e,e)=-2$, or $(e,e)=2-2n$
and $n:=\dim_\ComplexNumbers(X)/2>2$. 
In the latter case one of the following properties holds: 
\begin{enumerate}
\item
$\div(e,\bullet)=2n-2$ and 
$rs(e)=\{1,n-1\}$.
\item
$\div(e,\bullet)=2n-2$ and $rs(e)=\{2,(n-1)/2\}$. We must have
$n\equiv 3$ (modulo $4$) for the pair $rs(e)$ to be relatively prime.
\item
$\div(e,\bullet)=n-1$, $n$ is even, and $rs(e)=\{1,n-1\}$.
\item
$\div(e,\bullet)=n-1$, $n$ is odd, and $rs(e)=\{1,(n-1)/2\}$.
\end{enumerate}
A cohomology class $e\in H^{1,1}(X,\Integers)$ is 
{\em numerically exceptional}, if $e=c_1(L)$, for 
a numerically exceptional line bundle $L$.
\end{defi}

\begin{defi}
\label{def-stably-prime-exceptional}
\begin{enumerate}
\item
A line bundle $L\in \Pic(X)$ is called {\em stably-prime-exceptional},
if there exists a closed complex analytic subset $Z\subset D_e$, 
of codimension $\geq 1$, such that the linear system $\linsys{L_t}$
consists of a prime-exceptional divisor $E_t$, for all $t\in [D_e\setminus Z]$.
\item
$L$ is said to be {\em stably-$\RationalNumbers$-effective},
if there exists a non-zero integer $k$, 
such that the linear system $\linsys{L_t^k}$
is non-empty, for all $t\in D_e$.
\end{enumerate} 
\end{defi}

If $E$ is a prime exceptional divisor on a projective irreducible
holomorphic symplectic manifold $X$, then $\StructureSheaf{X}(E)$
is stably-prime-exceptional, by Proposition 
\ref{prop-generic-prime-exceptional}.

Let $L$ be a line bundle on an irreducible
holomorphic symplectic manifold $X$ of $K3^{[n]}$-type
with a primitive first Chern class.
Recall that a necessary condition for the linear system 
$\linsys{L^k}$ to consist of an exceptional divisor $E$,
is that $L$ is monodromy-reflective (Definition \ref{def-monodromy-reflective}), 
by Theorem \ref{cor-introduction}.
Assume that $L$ is monodromy-reflective.
Set $e:=c_1(L)$. 
Let $D_e\subset Def(X)$ be the divisor fixed by the reflection $R_e$. 

\WithoutTorelli{\begin{conj}}
\WithTorelli{\begin{thm}}
\label{conj-exceptional-line-bundles}
\begin{enumerate}
\item
\label{conj-item-effective}
Assume that $L$ is numerically exceptional. 
Then $L^k$ is stably-prime-exceptional, where $k$ is determined as follows.
If the degree of $L$ is $2-2n$, then
\[
k = \left\{
\begin{array}{ccl}
2, & \mbox{if} & \div(e,\bullet)=2n-2 \ \mbox{and} \ rs(e)=\{1,n-1\},
\\
1, & \mbox{if} & \div(e,\bullet)=2n-2 \ \mbox{and} \  rs(e)=\{2,(n-1)/2\},
\\
1, & \mbox{if} & \div(e,\bullet)=n-1.
\end{array}
\right.
\]
If the degree of $L$ is $-2$, then
\[
k = \left\{
\begin{array}{ccl}
2, & \mbox{if} & \div(e,\bullet)=2 \ \mbox{and} \ n=2,
\\
1, & \mbox{if} & \div(e,\bullet)=2 \ \mbox{and} \  n>2,
\\
1, & \mbox{if} & \div(e,\bullet)=1.
\end{array}
\right.
\]
\item
\label{conj-item-vanishing}
If $L$ is not numerically exceptional, then $L$ is not 
stably-$\RationalNumbers$-effective. I.e., 
for every non-zero integer $k$, there exists a dense open subset
$U^k$ of $D_e$, such that 
$H^0(X_t,L_t^k)$ vanishes, for all $t\in U^k$.
\end{enumerate}
\WithoutTorelli{\end{conj}}
\WithTorelli{\end{thm}}

See Remark \ref{rem-Euler-characteristic} for the
Euler characteristic $\chi(L^k)$. 
Note that in part \ref{conj-item-effective} above $L^k$
is effective as well, for the specified integer $k$, by the semi-continuity 
theorem. 

\WithoutTorelli{
\begin{thm}
\label{thm-main-conjecture-follows-from-torelli}
Let $X$ be an irreducible holomorphic symplectic manifold of 
$K3^{[n]}$-type and $L$ a monodromy-reflective line bundle on $X$.
Assume an affirmative answer to the Torelli
Question \ref{thm-torelli}
[or the weaker Question \ref{question-connectedness} for the pair $(X,L)$]. 
Then Conjecture \ref{conj-exceptional-line-bundles}
holds for $(X,L)$.
\end{thm}
Verbitsky  recently posted a proof of an affirmative answer to 
Question \ref{thm-torelli} \cite{verbitsky}.
Theorem \ref{thm-main-conjecture-follows-from-torelli}
is proven in section
\ref{sec-numerical-characterization-via-torelli}.}
\WithTorelli{Theorem
\ref{conj-exceptional-line-bundles}
is proven in section \ref{sec-numerical-characterization-via-torelli}.}
The proof relies both on the Torelli Theorem \cite{verbitsky}
and  the examples worked out in sections 
\ref{sec-examples} and \ref{sec-non-effective}. 
We exhibit an example of a pair $(X,L)$, for each 
possible value of the monodromy invariants $(e,e)$, $\div(e,\bullet)$, and $rs(e)$, and verify
\WithTorelli{Theorem}\WithoutTorelli{Conjecture}
\ref{conj-exceptional-line-bundles}
for $(X,L)$. 
All values of the monodromy invariants are realized by 
examples where $X$ is a smooth and projective moduli space of sheaves 
on a $K3$ surface. See the table in section 
\ref{sec-numerical-characterization-via-torelli} 
for a reference to an example, for each value of the 
monodromy-invariants.

The vanishing in part \ref{conj-item-vanishing} 
of \WithTorelli{Theorem}\WithoutTorelli{Conjecture}
\ref{conj-exceptional-line-bundles}
is verified in the examples as follows. 
In all the examples of monodromy-reflective but
non-numerically-exceptional
line bundles considered in section \ref{sec-non-effective}, 
$X$ admits a birational involution $\iota: X\rightarrow X$,
inducing the reflection $R_{e}$. 

The following simple observation is proven in section \ref{sec-non-effective}.
\begin{observation}
\label{observation-not-Q-effective}
If $L$ is a monodromy reflective line bundle on $X$, and
there exists a bimeromorphic involution $\iota:X\rightarrow X$
inducing the reflection $R_{e}$, $e=c_1(L)$, then 
the line bundle $L$ is not $\RationalNumbers$-effective.
\end{observation}

\subsection{Cones}
Let $X$ be a projective irreducible holomorphic symplectic manifold.
Set $N^1(X):=H^{1,1}(X,\Integers)\otimes_\Integers\RealNumbers$ and let $\C_X^{1,1}$ 
be the connected component of the cone $\{\lambda\in N^1(X) \  : \ (\lambda,\lambda)>0\}$,
which contains the ample cone. Denote by $\overline{\C}_X^{1,1}$ its closure. 
A divisor $D$ on $X$ is called {\em movable}, if the base locus of the linear system $\linsys{D}$
has codimension $\geq 2$ in $X$. Denote by
$\MV_X$ the convex cone in $N^1(X)$ generated by classes of movable divisors. 
Let $\overline{\MV}_X$ be its closure in $N^1(X)$. Then $\overline{\MV}_X$
is equal to the sub-cone of $\overline{\C}_X^{1,1}$, 
consisting of classes $\lambda$, such that $(\lambda,[E])\geq 0$, for every 
prime exceptional divisor $E$ \cite[Lemma 6.22]{boucksom,markman-torelli}. 

The closure of the movable cone can be described also in terms of the set of stably-prime-exceptional divisors.
$\overline{\MV}_X$  is the sub-cone of $\overline{\C}_X^{1,1}$, 
consisting of classes $\lambda$, such that $(\lambda,e)\geq 0$, for every 
stably-prime-exceptional class $e$ \cite[Theorem 6.17 and Lemma 6.22]{markman-torelli}. 
Hence, Theorem \ref{conj-exceptional-line-bundles} above determines the closure of the movable cone. 
Furthermore, a stably-prime-exceptional class $e$ is prime exceptional, if and only if 
the hyperplane orthogonal to $e$ intersects $\overline{\MV}_X$ along a face 
of codimension one of the latter \cite[Lemma 6.20]{markman-torelli}. 
In this sense Theorem \ref{conj-exceptional-line-bundles} determines the set of 
classes of prime exceptional divisors. 

%
\subsection{The structure of the paper}
The paper is organized as follows. 
In section \ref{sec-easy-examples} we provide a sequence of easy examples 
of monodromy-reflective line bundles on moduli spaces of sheaves on $K3$ 
surfaces. We calculate their invariants, and determine whether or not they are
effective, illustrating 
\WithTorelli{Theorem}
\WithoutTorelli{Conjecture}
\ref{conj-exceptional-line-bundles}.

In section \ref{sec-monodromy-reflection} we prove 
Theorem \ref{cor-introduction} stating 
that associated to a prime exceptional divisor $E$ is 
a monodromy involution of the integral cohomology 
$H^*(X,\Integers)$, which acts on the second cohomology lattice
as the reflection by the cohomology class $[E]$ of $E$ 
(Corollary \ref{cor-1}).
In section \ref{sec-degrees-2-and-2-2n}
we specialize to the $K3^{[n]}$-type case, $n\geq 2$,
and prove Theorem \ref{thm-2} about the possible degrees of prime exceptional
divisors.

Let $(X_i,E_i)$, $i=1,2$, be two pairs, each consisting of an 
irreducible holomorphic symplectic manifold $X_i$, 
and a prime exceptional divisor $E_i$.
Let $e_i\in H^{2}(X,\Integers)$ be the class of $E_i$.
In section \ref{sec-deformation-equivalence}
we define two notions of deformation equivalence:

(1) Deformation equivalence of the two pairs $(X_i,E_i)$, $i=1,2$.

(2) Deformation equivalence of the two pairs $(X_i,e_i)$, $i=1,2$.

\noindent
We relate these two notions via Torelli.

In section \ref{sec-Mukai} we return to the case where $X$ is of 
$K3^{[n]}$-type, $n\geq 2$.
We associate, to each monodromy-reflective 
class $e\in H^2(X,\Integers)$ of degree $2-2n$, 
an isometry class of a pair $(\widetilde{L},\tilde{e})$, consisting of 
a rank $2$ integral lattice $\widetilde{L}$ of signature $(1,1)$ and a primitive class
$\tilde{e}\in \widetilde{L}$, with $(\tilde{e},\tilde{e})=2-2n$.
The isometry class of the pair $(\widetilde{L},\tilde{e})$ is a 
monodromy-invariant, denoted by $f(X,e)$.
In section \ref{sec-invariant-rs} we calculate the monodromy-invariant
$f(X,e)$ explicitly as the function $rs$ in Proposition
\ref{prop-introduction-Mon-2-orbit-is-determined-by-three-invariants}.
We then prove Proposition 
\ref{prop-introduction-Mon-2-orbit-is-determined-by-three-invariants}.

In section \ref{sec-numerical-characterization-via-torelli}
we prove 
\WithTorelli{Theorem \ref{conj-exceptional-line-bundles}, which
provides a numerical characterization of exceptional classes.}
\WithoutTorelli{Theorem \ref{thm-main-conjecture-follows-from-torelli}
stating that the numerical characterization of exceptional classes,  
suggested in Conjecture 
\ref{conj-exceptional-line-bundles}, 
follows from a version of Torelli.}

Sections \ref{sec-conditions-for-existence-ofslope-stable-vector-bundles},
\ref{sec-examples}, and \ref{sec-non-effective} are devoted to
examples of monodromy-reflective line bundles over moduli
spaces of stable sheaves.
In section \ref{sec-conditions-for-existence-ofslope-stable-vector-bundles} 
we study the exceptional locus of 
Jun Li's morphism from certain  moduli spaces, of Gieseker-Maruyama 
$H$-stable sheaves on a $K3$-surface, to 
the Uhlenbeck-Yau compactifications of the moduli spaces of
$H$-slope stable vector-bundles.

In section \ref{sec-examples} we exhibit an example of
a prime exceptional divisor, for each value of the 
invariants of a monodromy-reflective line bundle $L$, 
for which $L$ is stated to be stably-$\RationalNumbers$-effective
in \WithTorelli{Theorem}\WithoutTorelli{Conjecture}
\ref{conj-exceptional-line-bundles}.

In section \ref{sec-non-effective} we exhibit an example of
a monodromy-reflective line bundle $L$, which is
not $\RationalNumbers$-effective, for each value of the 
invariants for which $L$ is stated not
to be $\RationalNumbers$-effective 
in \WithTorelli{Theorem}\WithoutTorelli{Conjecture}
\ref{conj-exceptional-line-bundles}.

\hide{
In section \ref{sec-zariski-decomposition} we consider pairs $(X,L)$,
with $X$ of $K3^{[n]}$-type and $L$ a line bundle
of negative Beauville-Bogomolov degree, which is {\em not}
monodromy-reflective. We review first the 
Zariski decomposition of effective divisors, 
due to Boucksom \cite{boucksom}. Theorem \ref{thm-2},
and the existence of a divisorial Zariski decomposition, 
imply that $L$ is not $\RationalNumbers$-effective, for
a generic such pair $(X,L)$  (Lemma \ref{lemma-generic-vanishing}).
}

%
\section{Easy examples of monodromy-reflective line bundles}
\label{sec-easy-examples}
In section \ref{sec-Mukai-notation} 
we review basic facts about moduli spaces of coherent sheaves
on $K3$ surfaces.
In section \ref{sec-sequence-of-examples} 
we briefly describe a sequence of examples of pairs $(X,e)$,
with $X$ of $K3^{[n]}$-type, $e$ a monodromy-reflective class of
degree $2-2n$ with $\div(e,\bullet)=2n-2$, for each $n\geq 2$,
and for each value of the invariant $rs$. For details, 
references, as well as for examples of degree $-2$, or with 
$\div(e,\bullet)=n-1$, see sections
\ref{sec-examples} and \ref{sec-non-effective}.

%
\subsection{The Mukai isomorphism}
\label{sec-Mukai-notation}
The group $K(S)$, endowed with the {\em Mukai pairing}
\[
(v,w) \ \ := \ \ -\chi(v^\vee\otimes w), 
\]
is called the 
{\em Mukai lattice}. Let us recall Mukai's notation for elements of $K(S)$.
Identify the group $K(S)$ with $H^*(S,\Integers)$, via the 
isomorphism  sending a class $F$ to its {\em Mukai vector} 
$ch(F)\sqrt{td_S}$. Using the grading of $H^*(S,\Integers)$, 
the Mukai vector is 
\begin{equation}
\label{eq-Mukai-vector}
(\rank(F),c_1(F),\chi(F)-\rank(F)),
\end{equation} 
where the rank is considered in
$H^0$ and $\chi(F)-\rank(F)$ in $H^4$ via multiplication by the 
orientation class of $S$. The homomorphism 
$ch(\bullet)\sqrt{td_S}:K(S)\rightarrow H^*(S,\Integers)$ 
is an isometry with respect 
to the Mukai pairing on $K(S)$ and the pairing 
\[
\left((r',c',s'),(r'',c'',s'')\right) \ \ = \ \ 
\int_{S}c'\cup c'' -r'\cup s''-s'\cup r''
\]
on $H^*(S,\Integers)$ (by the Hirzebruch-Riemann-Roch Theorem). 
For example, $(1,0,1-n)$ is the Mukai vector in $H^*(S,\Integers)$, of the
ideal sheaf of a length $n$ subscheme. 
Mukai defines a weight $2$ Hodge structure on the Mukai lattice 
$H^*(S,\Integers)$, 
and hence on $K(S)$, by extending that of $H^2(S,\Integers)$, 
so that the direct summands $H^0(S,\Integers)$ and $H^4(S,\Integers)$
are of type $(1,1)$.

Let $v\in K(S)$ be a primitive class with $c_1(v)$ of Hodge-type $(1,1)$.
There is a system of hyperplanes in the ample cone of $S$, called $v$-walls,
that is countable but locally finite \cite{huybrechts-lehn-book}, Ch. 4C.
An ample class is called {\em $v$-generic}, if it does not
belong to any $v$-wall. Choose a $v$-generic ample class $H$. 
Let $M_H(v)$ be the moduli space of $H$-stable  
sheaves on the $K3$ surface $S$ with class $v$.
When non-empty, the moduli space 
$M_H(v)$ is a smooth projective irreducible holomorphic symplectic variety
of $K3^{[n]}$ type, with $n=\frac{(v,v)+2}{2}$.
This result is due to several people, including 
Huybrechts, Mukai, O'Grady, and Yoshioka. It can be found in its final form in
\cite{yoshioka-abelian-surface}.

Over $S\times M_H(v)$ there exists a universal sheaf
$\F$, possibly twisted with respect to a non-trivial 
Brauer class pulled-back from $M_H(v)$.
Associated to $\F$ is a class $[\F]$ in $K(S\times M_H(v))$
(\cite{markman-integral-generators}, Definition 26).
Let $\pi_i$ be the projection from $S\times M_H(v)$ onto the $i$-th factor. 
Assume that $(v,v)>0$. 
The second integral cohomology $H^2(M_H(v),\Integers)$, its Hodge structure, 
and its Beauville-Bogomolov pairing, are all described by 
Mukai's Hodge-isometry
\begin{equation}
\label{eq-Mukai-isomorphism}
\theta \ : \ v^\perp \ \ \  \longrightarrow \ \ \  H^2(M_H(v),\Integers),
\end{equation}
given by $\theta(x):=c_1\left(\pi_{2_!}\{\pi_1^!(x^\vee)\otimes [\F]\}\right)$
(see \cite{yoshioka-abelian-surface}). 

Let $\widetilde{\Lambda}$ be the unimodular lattice 
$E_8(-1)^{\oplus 2}\oplus U^{\oplus 4}$, where $U$ is the rank two unimodular 
hyperbolic lattice. 
$\widetilde{\Lambda}$ is isometric 
to the Mukai lattice of a $K3$ surface. 
Let $X$ be an irreducible holomorphic symplectic 
manifold of $K3^{[n]}$-type, $n\geq 2$.
Recall that $X$ comes with a natural choice of an 
$O(\widetilde{\Lambda})$-orbit
of primitive isometric embeddings of $H^2(X,\Integers)$ 
in $\widetilde{\Lambda}$, by
Theorem \ref{thm-a-natural-orbit-of-embeddings-of-H-2-in-Mukai-lattice}.

\begin{thm}
\label{thm-item-orbit-of-inverse-of-Mukai-isom-is-natural}
(\cite{markman-constraints}, Theorem 1.14).
When $X$ is isomorphic to the moduli space 
$M_H(v)$, of $H$-stable sheaves on a $K3$ surface of class $v\in K(S)$,
then the above mentioned $O(\widetilde{\Lambda})$-orbit is 
that of the composition
\begin{equation}
\label{eq-iota-for-a-moduli-space}
H^2(M_H(v),\Integers) \LongRightArrowOf{\theta^{-1}}v^\perp
\subset K(S) \cong \widetilde{\Lambda},
\end{equation}
where $\theta^{-1}$ is the inverse of the Mukai isometry given in
(\ref{eq-Mukai-isomorphism}).
\end{thm}

The combination of Theorems 
\ref{thm-a-natural-orbit-of-embeddings-of-H-2-in-Mukai-lattice} and 
\ref{thm-item-orbit-of-inverse-of-Mukai-isom-is-natural}
is an example of the following meta-principle guiding our study of the 
monodromy of holomorphic symplectic varieties of $K3^{[n]}$-type.

\smallskip
{\em Any topological construction, 
which can be performed uniformly and naturally 
for all smooth and compact moduli spaces of sheaves on any $K3$ surface $S$,
and which is invariant under symmetries induced by equivalences of 
derived categories of $K3$-surfaces,
is monodromy-invariant.
}

%
\subsection{A representative sequence of examples}
\label{sec-sequence-of-examples}
Let $S$ be a projective $K3$ surface with a cyclic Picard group generated by an
ample line bundle $H$.
Fix integers $r$ and $s$ satisfying $s\geq r\geq 1$, and $\gcd(r,s)=1$.
Let $X$ be the moduli space $M_H(r,0,-s)$.
Then $X$ is a projective irreducible 
holomorphic symplectic manifold of $K3^{[n]}$-type with $n=1+\nolinebreak rs$
\cite{yoshioka-abelian-surface}. 
Set $e:=\theta(r,0,s)$.
The weight two integral Hodge structure
$H^2(M_H(r,0,-s),\Integers)$ is Hodge-isometric
to the orthogonal direct sum $H^2(S,\Integers)\oplus \Integers e$,
and the class 
$e$ is monodromy-reflective of Hodge-type $(1,1)$, $(e,e)=2-2n$, 
$\div(e,\bullet)=2n-2$, and $rs(e)=\{r,s\}$. 

When $r=1$, then $X=S^{[n]}$ is the Hilbert scheme. Let $E\subset S^{[n]}$
be the big diagonal. Then $E$ is a prime divisor, which is the
exceptional locus of the Hilbert-Chow morphism $\pi:S^{[n]}\rightarrow S^{(n)}$
onto the $n$-th symmetric product of $S$. The equality $e=\frac{1}{2}[E]$
was proven in \cite{beauville}.

When $r=2$, let $E\subset M_H(2,0,-s)$ be the locus
of $H$-stable sheaves which are not locally free.
Then $E$ is a prime divisor, which is the exceptional locus of 
Jun Li's morphism from $M_H(2,0,-s)$ onto the Uhlenbeck-Yau compactification
of the moduli space of $H$-slope-stable vector bundles of that class.
The equality $e=[E]$ holds, by Lemma \ref{lemma-class-of-exceptional-locus}.

When $r\geq 3$, let $Exc\subset M_H(r,0,-s)$ be the locus of 
$H$-stable sheaves, which are not locally free or not
$H$-slope-stable. Then $Exc$ is a closed algebraic subset of 
$M_H(r,0,-s)$ of codimension $\geq 2$, by Lemma
\ref{lem-codimension-of-Exc}. Jun Li's morphism is thus not a 
divisorial contraction. Set $U:=X\setminus Exc$. 
Let $\iota:U\rightarrow U$ be the regular involution, which sends
a locally free $H$-slope-stable sheaf $F$ to the dual sheaf $F^*$.
The restriction homomorphism from $H^2(X,\Integers)$
to $H^2(U,\Integers)$ is an isomorphism, and the induced involution
$\iota^*$ of $H^2(X,\Integers)$ is the reflection by the 
class $e$, by Proposition \ref{prop-vanishing-in-divisibility-2n-2}.
The class $e$ is thus not $\RationalNumbers$-effective, by Observation
\ref{observation-not-Q-effective}.

%
\section{The monodromy reflection of a prime exceptional divisor}
\label{sec-monodromy-reflection}
Let $X$ be a projective irreducible holomorphic symplectic manifold
and $E$ a reduced and irreducible divisor with negative Beauville-Bogomolov
degree. The following result is due to S. Bouksom and S. Druel.

\begin{prop}
\label{prop-druel}
(\cite{druel}, Proposition 1.4)
There exists a sequence of flops of $X$,
resulting in a smooth birational model $X'$ of $X$,
such that the strict transform $E'$ of $E$ in $X'$ is contractible
via a projective birational 
morphism $\pi:X' \rightarrow Y$ to a normal projective variety 
$Y$. The exceptional locus of $\pi$ is equal to the support of $E'$.
\end{prop}

The divisor $E$ is assumed to be {\em exceptional}, rather than to have
negative Beauville-Bogomolov degree, in the 
statement of Proposition 1.4 in \cite{druel}. 
The technical term exceptional is in the sense of \cite{boucksom},
Definition 3.10, and is a precise measure of rigidity.
Boucksom characterized exceptional divisors on irreducible holomorphic
symplectic varieties by the following property, which we will use
as a definition (\cite{boucksom}, Theorem 4.5). 

\begin{defi}
\label{def-rational-exceptional}
A rational divisor $E\in Div(X)\otimes_\Integers\RationalNumbers$ is 
{\em exceptional}, if $E=\sum_{i=1}^k n_i E_i$, 
with positive rational coefficients $n_i$, prime divisors $E_i$, 
and a negative definite Gram-matrix $([E_i],[E_j])$. 
\end{defi}

In particular, a prime divisor is exceptional, if and only if it
has negative Beauville-Bogomolov degree. 
Hence, we may replace in the above Proposition the hypothesis that $E$ is 
exceptional by the hypothesis that $E$ has negative Beauville-Bogomolov 
degree.

\begin{defi}
\label{def-exceptional}
A primitive class $e\in H^2(X,\Integers)$ is {\em (prime) exceptional} 
if some positive multiple of $e$ is the class of a 
(prime) exceptional divisor. A line bundle $L\in Pic(X)$ is 
{\em (prime) exceptional}, if $c_1(L)$ is.
\end{defi}

Let $Def(X')$ and $Def(Y)$ be the Kuranishi deformation spaces of $X'$ and $Y$.
Denote by $\psi:\X\rightarrow Def(X')$ 
the semi-universal deformation of $X'$, by $0\in Def(X')$ the point
with fiber $X'$, and let $X_t$ be the fiber over $t\in Def(X')$. 
Let $\bar{\psi}:\Y\rightarrow Def(Y)$ 
be the semi-universal deformation of $Y$, $\bar{0}\in Def(Y)$
its special point with fiber $Y$, and 
$Y_t$ the fiber over $t\in Def(Y)$.

The variety $Y$ necessarily has rational Gorenstein singularities, by
\cite{beuville-symplectic-singularities}, Proposition 1.3.
The morphism $\pi:X'\rightarrow Y$ deforms as a morphism $\nu$ of the 
semi-universal families, which fits in a commutative diagram 
\begin{equation}
\label{diagram-f-general}
\begin{array}{ccc}
\X & \RightArrowOf{\nu} & \Y
\\
\psi \ \downarrow \hspace{1ex} & & \bar{\psi} \ \downarrow \ \hspace{1ex}
\\
Def(X') & \RightArrowOf{f} & Def(Y),
\end{array}
\end{equation}
by \cite{kollar-mori}, Proposition 11.4.
The following is a fundamental theorem of Namikawa:

\begin{thm}
\label{thm-namikawa}
(\cite{namikawa}, Theorem 2.2)
The Kuranishi deformation spaces
$Def(X')$ and $Def(Y)$ are both smooth of the same dimension. 
They can be replaced by open neighborhoods of $0\in Def(X')$
and $\bar{0}\in Def(Y)$, and denoted again by $Def(X')$ and $Def(Y)$, 
in such a way that 
there exists a natural proper surjective map
$f:Def(X')\rightarrow Def(Y)$ with finite fibers.
Moreover, for a generic point $t\in Def(X')$, 
$Y_{f(t)}$ is isomorphic to $X_t$. 
\end{thm}

The morphism $f:Def(X')\rightarrow Def(Y)$ is in fact a branched Galois 
covering, by \cite{markman-galois}, Lemma 1.2.
The Galois group $G$ is a product of Weyl groups of finite type,
by \cite{markman-galois}, Theorem 1.4 
(see also \cite{namikawa-galois}).
Furthermore, $G$ acts on $H^*(X',\Integers)$ via monodromy operators 
preserving the Hodge structure. When the exceptional locus of 
$\pi:X'\rightarrow Y$ 
contains a single irreducible component of co-dimension one in $X'$,
then the Galois group $G$ is $\Integers/2\Integers$.

Let $\Sigma\subset Y$ be the singular locus. 
The {\em dissident locus} $\Sigma_0\subset \Sigma$ is the locus along 
which the singularities of $Y$ fail to be of $ADE$ type.  
$\Sigma_0$ is a closed subvariety of $\Sigma$.

\begin{prop} 
\label{prop-dissident-locus}
(\cite{namikawa}, Propositions 1.6, \cite{wierzba})
$Y$ has only canonical singularities. The dissident locus 
$\Sigma_0$ has codimension at least $4$ in $Y$. 
The complement $\Sigma\setminus\Sigma_0$ is either empty, 
or the disjoint union of codimension $2$ smooth and symplectic subvarieties of 
$Y\setminus \Sigma_0$. 
\end{prop}

Keep the notation of Proposition \ref{prop-druel}.
Let $[E]$ be the class of $E$ in $H^2(X,\Integers)$ and $R$ the reflection of
$H^2(X,\RationalNumbers)$ given by  
\[
R(x) \ \ := \ \ x-\left(\frac{2(x,[E])}{([E],[E])}\right)[E].
\]
Consider the natural isomorphism 
$H^2(X,\Integers)^*\cong H_2(X,\Integers)$, given by the 
Universal Coefficients Theorem and the fact that $X$ is simply connected.
Denote by
\begin{equation}
\label{eq-E-vee}
[E]^\vee \ \ \ \in \ \ \ H_2(X,\RationalNumbers)
\end{equation}
the class corresponding to $\frac{-2([E],\bullet)}{([E],[E])}$,
where both pairings in the fraction are the Beauville-Bogomolov
pairing. 

We identify $H^2(X,\Integers)$ and $H^2(X',\Integers)$ 
via the graph of the birational map. This graph induces a Hodge isometry
and the isometry maps the class $[E]\in H^2(X,\Integers)$ to the class 
$[E']\in H^2(X',\Integers)$,
by \cite{ogrady-weight-two}, Proposition 1.6.2.  
We get an identification of the dual groups 
$H_2(X,\Integers)$ and $H_2(X',\Integers)$. 
The following Corollary was proven, 
in case $E$ is an irreducible component of a contractible divisor,
in \cite{markman-galois} Lemmas 4.10 and 4.23.
We are now able to extend it to the more general case of a
prime exceptional divisor $E$. 
The following is a Corollary of 
Proposition \ref{prop-druel}, 
Proposition \ref{prop-dissident-locus}, 
and \cite{markman-galois}, Lemmas 4.10 and 4.23.

\begin{cor}\label{cor-1}
\begin{enumerate}
\item
\label{item-integrality}
The class $[E]^\vee\in H_2(X,\Integers)$ corresponds to the class in 
$H_2(X',\Integers)$ of the generic fiber of the contraction $E'\rightarrow Y$
in Proposition \ref{prop-druel}.
The generic fiber is either a smooth rational curve, or the union of two 
homologous smooth rational curves meeting at one point. In particular, 
$[E]^\vee$ is an integral class in $H_2(X,\Integers)$ and 
$R$ is an integral isometry. 
\item
\label{item-monodromy-operator}
The reflection $R$ is a monodromy operator in $Mon^2(X)$ as well as 
$Mon^2(X')$. $R$ preserves the 
Hodge structure. The action of $R$ on $H^{1,1}(X')\cong H^1(X',TX')$ induces
an involution of $Def(X')$, which generates the Galois group of
the double cover of the Kuranishi deformation spaces 
$Def(X')\rightarrow Def(Y)$.
\item
\label{item-divisibility-at-most-two}
Either $[E]$ is a primitive class of $H^2(X,\Integers)$,
or $[E]$ is twice a primitive class. Similarly, 
$[E]^\vee$ is either a primitive class or  twice a primitive class.
Finally, at least one of $[E]$ or $[E]^\vee$ is a primitive class
\end{enumerate}
\end{cor}

\begin{proof} 
\ref{item-integrality})
The singular locus of $Y$ contains a unique 
irreducible component $\Sigma$ of codimension $2$, and $Y$
has singularities of type $A_1$ or $A_2$ along the Zariski dense 
open subset $\Sigma\setminus\Sigma_0$, by Proposition
\ref{prop-dissident-locus} 
(see also the classification of singularities in section 1.8 of
\cite{namikawa}). 
When $X=X'$, the class $[E']^\vee$ is the class of the fiber
of the composite morphism $E'\hookrightarrow X'\rightarrow Y$,
by Lemmas 4.10 and 4.23 in \cite{markman-galois}.

\ref{item-monodromy-operator})
$R$ is a monodromy operator in $Mon^2(X')$, 
by Lemmas 4.10 and 4.23 in \cite{markman-galois}. 
Now the isometry $Z_*:H^2(X,\Integers)\rightarrow H^2(X',\Integers)$,
induced by the graph $Z$ of the birational map, is a
parallel transport operator. This follows from 
the proof of Theorem 2.5 in \cite{huybrechts-kahler-cone}. 
In this proof Huybrechts constructs
a correspondence $\Gamma:=Z+\sum_{i}Y_i$ in $X\times X'$
with the following properties. 
$\Gamma_*:H^2(X,\Integers)\rightarrow H^2(X',\Integers)$
is a parallel transport operator, 
$Z$ is the closure of the graph of the birational map as above, 
and the image of $Y_i$ in each factor $X$ and $X'$ has codimension $\geq 2$.
It follows that the two isometries $Z_*$ and $\Gamma_*$ coincide.

\ref{item-divisibility-at-most-two})
The statements about the divisibility of $[E]$ and $[E]^\vee$
follow from the equality $\int_{[E]^\vee}[E]=-2$.
\end{proof}

We denote by 
\[
e\in H^2(X,\Integers)
\] 
the primitive class, such that 
either $[E]=e$ or $[E]=2e$.
Let $e^\vee$ be the primitive
class in $H_2(X,\Integers)$, such that $[E]^\vee=e^\vee$ or
$[E]^\vee=2e^\vee$. The {\em divisibility factor}
$\div(e,\bullet)$, of the class $(e,\bullet)\in H^2(X,\Integers)^*$, 
is the positive number satisfying the equality 
$(e,\bullet)=\div(e,\bullet)\cdot e^\vee$. 

\begin{lem}
\label{lemma-divisibility}
We have
\[
-\div(e,\bullet) \ \ = \ \ 
\left\{
\begin{array}{ccc}
(e,e)/2 & \mbox{if} & [E]=e \ \mbox{and} \ [E]^\vee=e^\vee,
\\
(e,e) & \mbox{if} & [E]=2e \ \mbox{and} \ [E]^\vee=e^\vee,
\\
(e,e) & \mbox{if} & [E]=e \ \mbox{and} \ [E]^\vee=2e^\vee.
\end{array}
\right.
\]
\end{lem}

\begin{proof}
Let $[E]=k_1e$ and $[E]^\vee=k_2e^\vee$. Then  
\[
-(e,\bullet)=
\frac{-1}{k_1}([E],\bullet)=\frac{([E],[E])}{2k_1}[E]^\vee=
\frac{k_1(e,e)}{2}[E]^\vee=\frac{k_1k_2(e,e)}{2}e^\vee.
\]
\end{proof}

\begin{rem}
\label{rem-an-exceptional-linear-system-is-a-singleton}
Let $L$ be the line bundle with $c_1(L)=e$. 
Then $\dim H^0(X,L^n)$ is either $0$ or $1$, for all $n\in\Integers$,
by \cite{boucksom}, Proposition 3.13. 
Hence, there exists at most one non-zero integer $n$, 
such that the linear system $\linsys{L^n}$ contains a prime divisor. 
In particular, for a given pair $(X,e)$, at most one of 
the equalities $([E],[E]^\vee)=(e,2e^\vee)$ or
$([E],[E]^\vee)=(2e,e^\vee)$ can hold, for some prime divisor $E$ with
$[E]\in{\rm span}_\Integers\{e\}$. The same holds for an exceptional divisor,
where the coefficients $n_i$ in Definition \ref{def-rational-exceptional}
are integral and with $\gcd\{n_i \ : \ 1\leq i \leq k\}=1$.
\end{rem}

%

%
\section{Holomorphic symplectic manifolds of $K3^{[n]}$-type}
\label{sec-degrees-2-and-2-2n}
We prove Proposition 
\ref{prop-reflection-by-a-numerically-prime-exceptional-is-in-Mon}
in this section. This completes
the proof of Theorem \ref{thm-2}.
The lattice $H^2(X,\Integers)$ has signature $(3,20)$ \cite{beauville}.
A $3$-dimensional subspace of $H^2(X,\RealNumbers)$ is said to be 
{\em positive-definite}, if the 
Beauville-Bogomolov pairing restricts to it as a positive definite pairing. 
The unit $2$-sphere, in any positive-definite $3$-dimensional subspace, 
is a deformation retract of 
the positive cone $\C_+\subset H^2(X,\RealNumbers)$, given by
$
\C_+ := \{\lambda\in  H^2(X,\RealNumbers) \ : \ 
(\lambda,\lambda)>0\}.
$
Hence, $H^2(\C_+,\Integers)$ is isomorphic to $\Integers$
and is a natural representation of the isometry group
$OH^2(X,\RealNumbers)$. 
We denote by $O_+H^2(X,\Integers)$ the index two subgroup
of $OH^2(X,\Integers)$, which acts trivially on 
$H^2(\C_+,\Integers)$.

Let $X$ be of $K3^{[n]}$-type, $n\geq 2$. 
Embed the lattice $H^2(X,\Integers)$ in its dual lattice
$H^2(X,\Integers)^*$, via the Beauville-Bogomolov form.

\begin{thm}
\label{thm-monodromy-constraints}
(\cite{markman-constraints}, Theorem 1.2 and Lemma 4.2).
$Mon^2(X)$ is equal to the subgroup of $O_+H^2(X,\Integers)$,
which acts via multiplication by $1$ or $-1$ on the quotient group
$H^2(X,\Integers)^*/H^2(X,\Integers)$.
\end{thm}

The quotient $H^2(X,\Integers)^*/H^2(X,\Integers)$ is a cyclic
group of order $2n-2$. Indeed, 
we may deform $X$ to the Hilbert scheme $S^{[n]}$ of length $n$ subschemes of a
$K3$ surface $S$. $H^2(S^{[n]},\Integers)$ is
Hodge-isometric to the orthogonal direct sum 
\begin{equation}
\label{eq-orthogonal-direct-sum}
H^2(S,\Integers)\oplus \Integers \delta, 
\end{equation}
where 
$\delta$ is half the class of the big diagonal, 
and $(\delta,\delta)=2-2n$ \cite{beauville}.
Let $\pi:S^{[n]}\rightarrow S^{(n)}$ be the Hilbert-Chow morphism onto
the symmetric product of $S$. 
The isometric embedding 
$H^2(S,\Integers)\hookrightarrow H^2(S^{[n]},\Integers)$ is 
given by the composition
$
H^2(S,\Integers)\cong H^2(S^{(n)},\Integers)\LongRightArrowOf{\pi^*}
H^2(S^{[n]},\Integers). 
$

{\bf Proof\footnote{I thank V. Gritsenko for reference \cite[Corollary 3.4]{GHS-K3}, which drastically shortens the original proof.} 
of Proposition 
\ref{prop-reflection-by-a-numerically-prime-exceptional-is-in-Mon}:}
The lattice $H^2(X,\Integers)$ is isometric to the orthogonal direct sum (\ref{eq-orthogonal-direct-sum}).
Let $e$ be a class in $H^2(X,\Integers)$
of negative Beauville-Bogomolov degree, 
and let $R_e(x):=x-2\frac{(x,e)}{(e,e)}e$ be the reflection by $e$.
Then $R_e$ is an integral isometry of $H^2(X,\Integers)$, which acts by $1$ or $-1$ on 
the quotient $H^2(X,\Integers)^*/H^2(X,\Integers)$, if and only if $e$ has one of the two properties 
in the statement of Proposition 
\ref{prop-reflection-by-a-numerically-prime-exceptional-is-in-Mon}, by 
\cite[Corollary 3.4]{GHS-K3}. The proposition now follows from Theorem \ref{thm-monodromy-constraints}.
\EndProof


\begin{rem}
\label{rem-Euler-characteristic}
Let $X$ be of $K3^{[n]}$-type, $L$ a line bundle on $X$, and set
$\alpha:=c_1(L)$.
Then the sheaf-cohomology Euler characteristic of $L$ is given by the 
binomial coefficient
$
\chi(L) = \Choose{[(\alpha,\alpha)/2]+n+1}{n},
$
by \cite{huybrechts-norway}, section 3.4, Example 7. 
We get the following equalities.\\
$
\chi(L) = \left\{
\begin{array}{ccl}
1 & \mbox{if} & (\alpha,\alpha)=-2
\\
0 & \mbox{if} & (\alpha,\alpha)=2-2n \ \mbox{and} \ n\geq 3,
\end{array}\right.
$
\\
$
\begin{array}{ccccccc}
\chi(L^2) & = & 
0 & \mbox{if} & (\alpha,\alpha)=-2 & \mbox{and} & n\geq 3,
\\
\chi(L^2) & < &
0 & \mbox{if} & (\alpha,\alpha)=2-2n & \mbox{and} & n\geq 2.
\end{array}
$

\end{rem}

%
\section{Deformation equivalence}
\label{sec-deformation-equivalence}
This section is influenced by an early draft of a paper 
of Brendan Hassett and Yuri Tschinkel, which was graciously communicated to 
the author \cite{hassett-tschinkel-monodromy}.
\WithTorelli{The results rely heavily on
Verbitsky's  Torelli Theorem (Theorem \ref{thm-torelli}).} 

%
\subsection{The prime-exceptional property of pairs $(X,L)$ is open}
Let $X$ be a projective irreducible holomorphic symplectic manifold
and $E$ a prime exceptional divisor.
Set $c:=[E]\in H^2(X,\Integers)$.
Given a point $t\in Def(X)$, let $c_t\in H^2(X_t,\Integers)$
be the class associated to $c$ via the parallel transport
isomorphism\footnote{The local system $R^2\psi_*\Integers$ over $Def(X)$
is trivial, since we may choose $Def(X)$ to be simply connected.} 
$H^2(X,\Integers)\rightarrow H^2(X_t,\Integers)$.
Denote by $R_c$ both the reflection of $H^2(X,\Integers)$ with respect to $c$,
as well as the involution of $Def(X)$. 
Let $D_c\subset Def(X)$ be the fixed locus of $R_c$.
$D_c$ is a smooth divisor in $Def(X)$, which is characterized also as the 
subset
\begin{equation}
\label{eq-wall}
D_c \ \ := \ \ \{t\in Def(X) \ : \ c_t \ \mbox{is of Hodge type} \ (1,1)\}.
\end{equation}

\begin{lem}
\label{lem-effectivity-along-a-wall}
There exists an open subset $D_c^0\subset D_c$, containing $0$, 
such that for every $t\in D_c^0$ the class $c_t$ is Poincare dual to 
a prime exceptional divisor $E_t$. 
\end{lem}

\begin{proof}
Let $X'$, $E'$, and $Y$ be as in Proposition \ref{prop-druel}.
Denote the image of $E'\rightarrow Y$ by $B$. 
The generic fiber of $E'\rightarrow B$ is either a smooth rational
curve $C$, whose normal bundle satisfies
\[
N_{C/X'} \ \ \cong \ \ 
\omega_C\oplus \left(\oplus_{i=1}^{2n-2}\StructureSheaf{C}\right),
\]
or the union of two such curves meeting non-tangentially at one point,
by Proposition \ref{prop-dissident-locus}.
Druel shows that the exceptional locus of the birational map
$X \rightarrow X'$ does not dominate $B$
(see the proof of \cite{druel}, Theorem 1.3). We conclude that a Zariski dense 
open subset of the original divisor $E$ in $X$ is covered by such rational 
curves. The proposition now follows from results of Ziv Ran
about the deformations  of such pairs $(X,C)$ (\cite{ziv-ran}, Theorem 1, 
with further comments in \cite{kawamata}).
Our argument is inspired by \cite{hassett-tschinkel-conj}, 
Theorems 4.1 and 4.3. 
Note first that the class of the curve $C$ remains of type $(n-1,n-1)$
over $D_c$, by part \ref{item-integrality} of Corollary \ref{cor-1}.
Let $\psi:\X\rightarrow D_c$ be the semi-universal family,
${\mathcal H}\rightarrow D_c$ the irreducible component
of the relative Douady space containing the point $t_0$ 
representing the pair $(X,C)$, and $\C\subset {\mathcal H}\times_{D_c}\X$ 
the universal subscheme. We get the diagram
\[
\begin{array}{ccccc}
& & \C & \LongRightArrowOf{f} & \X
\\
& & \alpha \ \downarrow \ \hspace{1ex} & & \hspace{1ex} \ \downarrow \ \psi
\\
t_0 & \in & {\mathcal H} & \LongRightArrowOf{\beta} & D_c.
\end{array}
\]
Let $\pi:H^1(C,N_{C/X})\rightarrow H^{2n}(\Omega_X^{2n-2})$
be the semi-regularity map. Then $\pi$
is an isomorphism of these one-dimensional vector spaces
(Observation (a) before Corollary 4 in \cite{ziv-ran}).
Theorem 1 of  \cite{ziv-ran} implies that the morphism $\beta$ is 
surjective, of relative dimension $2n-2$, and it is smooth at
the point $t_0$. It follows that $\C$ has relative dimension
$2n-1$ over $D_c$, and $\C$ is smooth along the rational curve $C$
over $t_0$. Furthermore, the fiber $(\beta\circ \alpha)^{-1}(0)$ 
contains a unique irreducible component $\widetilde{E}$, which 
dominates $E$, 
and $f:\widetilde{E}\rightarrow E$ has degree $1$,
by part \ref{item-integrality} of Corollary \nolinebreak\ref{cor-1}.

We claim that the differential $df:T\C\rightarrow f^*T\X$ is 
injective along $C$. 
$T\C$ comes with a natural filtration
$T_{\alpha}\subset T_{\beta\circ\alpha}\subset T\C$.
$(f^*T\X\restricted{)}{C}$ comes with the filtration
$TC\subset (TX\restricted{)}{C}\subset (T\X\restricted{)}{C}.$
The homomorphism $df$ is compatible with the filtarations and induces
the identity on the first and third graded summands $TC$ and $T_0(D_c)$.
It suffices to prove injectivity on the middle graded summand. 
The above condition on $N_{C/X}$ implies, furthermore,
that the evaluation homomorphism 
$H^0(N_{C/X})\otimes \StructureSheaf{C}\rightarrow N_{C/X}$ is injective.
It follows that the differential $df$ is injective along $C$. 
Consequently, $f(\C)$ determines a divisor $\E$ in $\X$, 
possibly after eliminating embedded components of $f(\C)$, which are
disjoint from the curve $C$. Furthermore, $\E$
intersects the fiber $X$ of $\psi$ along a divisor $E_0$ containing $E$
and $E_0$ is reduced along $E$. 

It remains to prove that $E_0$ is irreducible. Now the
fiber $X_t$ has a cyclic Picard group, for a generic $t \in D_c$.
Hence, the generic fiber $E_t$ of $\E$ is of class $kc_t$, for some positive
integer $k$. Thus $E_0$ is of class $kc$.
But the linear system $\linsys{kE}$ consists of a
single divisor $kE$, by \cite{boucksom}, Proposition 3.13. 
We get that $k=1$, since $E_0$ is reduced along $E$.
\end{proof}

Let $\pi:\X\rightarrow T$ be a smooth family of irreducible holomorphic
symplectic manifolds over a connected complex manifold $T$. 
Assume that there exists a section $e$ of $R^2_{\pi_*}\Integers$, 
everywhere of Hodge type $(1,1)$. Given a point $t\in T$, denote by
$L_t$ the line bundle on the fiber $X_t$ with class $e_t$.

\begin{prop}
\label{prop-generic-prime-exceptional}
Assume given a point $0\in T$, such that the fiber $X_0$ is projective and 
the linear system $\linsys{L_0^k}$, of the $k$-th tensor power,  
consists of a prime exceptional divisor 
$E_0\subset X_0$, for some positive integer $k$. Then $k=1$ or $2$. 
Let $Z\subset T$ be the subset of points $t\in T$, such that
$h^0(X_t,L_t^k)>1$, or there exists a non-prime divisor, which is a
member of the linear system $\linsys{L_t^k}$.
Then $Z$ is a proper and closed analytic subset of $T$. 
Furthermore, there exists 
an irreducible divisor $\E$ in $\X\setminus \pi^{-1}(Z)$, 
such that $\E$ intersects the fiber $\pi^{-1}(t)$ along 
a prime exceptional divisor $E_t$ of class 
$ke_t$, for every $t$ in the complement $T\setminus Z$.
\end{prop}

\begin{proof}
The integer $k$ is $1$ or $2$ by Corollary \ref{cor-1}.
The dimension $h^0(X_t,L_t^k)$ is an upper-semi-continuous function
on $T$, and so the locus where it is positive is a closed analytic subset of 
$T$. On the other hand, $h^0(X_t,L_t^k)$ is positive over an open subset,
by Lemma \ref{lem-effectivity-along-a-wall}. Hence, it is 
positive everywhere and $L_t^k$ is effective for all $t$.

Let $Z_1\subset T$ be the closed analytic subset, where 
$h^0(X_t,L_t^k)>1$. We know that $h(X_0,L_0^k)=1$, by \cite{boucksom},
Proposition 3.13.
Hence, we may assume, possibly after replacing $T$ by $T\setminus Z_1$, that 
$h^0(X_t,L_t^k)=1$, for all $t\in T$. 

We prove next that the section $e$ lifts to a line bundle 
$\LB\cong \StructureSheaf{\X}(\E)$, for a divisor
$\E\subset\X$, which does not contain any fiber of $\pi$. 
The following is part of the edge exact sequence of the spectral
sequence of the composite functor $\Gamma\circ\pi_*$ of push-forward
and taking global sections.
\[
H^1(T,\StructureSheaf{T}^*)\rightarrow H^1(\X,\StructureSheaf{\X}^*)
\rightarrow H^0(T,R^1_{\pi_*}\StructureSheaf{\X}^*) \rightarrow 
H^2(T,\StructureSheaf{T}^*).
\]
Let $V$ be an open subset of $T$ satisfying $H^i(V,\StructureSheaf{V}^*)=0$,
for $i=1,2$. Then the restriction of $e$ to $V$ lifts to
a line bundle $\LB_V$ over $\pi^{-1}(V)$.
Now $\pi_*\LB_V$ is a line bundle over $V$, 
which must be trivial, by the vanishing of $H^1(V,\StructureSheaf{V}^*)$.
Hence, $H^0(\pi^{-1}(V),\LB_V)\cong H^0(V,\StructureSheaf{V})$,
and there exists a unique divisor $\E_V\subset \pi^{-1}(V)$,
in the linear system $\linsys{\LB_V}$, which does not contain any 
fiber of $\pi$. If $V_1$ and $V_2$ are two such open subsets of $T$, then
the divisors $\E_{V_i}$ constructed above agree along
$\pi^{-1}(V_1\cap V_2)$, since they are canonical over any subset $V$ of
$V_1\cap V_2$, over which $H^i(V,\StructureSheaf{V}^*)=0$,
for $i=1,2$. Hence, we get a global divisor $\E\subset \X$.
Set $\LB:=\StructureSheaf{\X}(\E)$.

We prove next that $\E$ is irreducible.
Let $p:\E\rightarrow T$ be the restriction of 
$\pi$. Then $p$ is a proper morphism, which is also flat by 
\cite{matsumura} application 2 page 150. 
All fibers of $p$ are connected, since $T$ is smooth,
and in particular normal, and the fiber over $0$ is connected. 
The morphism $p$ is smooth along the smooth locus of $E_0$,
and $\E$ is a local complete intersection in the smooth complex manifold
$\X$. Hence, there exists precisely one irreducible component of $\E$
which contains $E_0$. 
Assume that there exists another irreducible component $\E'$.
Then $\E'$ maps to a proper closed subset of $T$, 
which does not contain $0$. 
But $T$ is irreducible, and $\E'$ intersects each fiber of $\pi$
along a subset, which is either empty or of codimension at least one. Hence,
the codimension of $\E'$ in $\X$ is larger than one. This contradicts
the fact that $\E$ is a divisor.
We conclude that $\E$ is irreducible. 

The subset $Z\subset T$, consisting of points
$t\in T$, where $E_t$ is reducible or non-reduced, 
is a closed analytic subset of $T$, which does not contain $0$.
\end{proof}

Proposition \ref{prop-generic-prime-exceptional} shows that the property 
that $L$ is prime exceptional is {\em open} in any smooth
and connected base $T$ of a
deformation of a pair $(X,L)$, as long as it holds for at least one projective
pair. 
One limiting case is that of a pair $(X,L)$, where 
$L$ is exceptional, in the sense of 
Definition \ref{def-rational-exceptional}, but no longer prime. 
However, the exceptional property is unfortunately not closed, 
as the following example shows.

\begin{example}
\label{ex-being-exceptional-is-not-a-closed-property}
Let $Y$ be the intersection of a quadric and a cubic in $\PP^{4}$,
which are tangent at one point $y_0$, such that $Y$ has an 
ordinary double point at $y_0$. Let $H$ be the hyperplane line bundle on $Y$.
$Y$ is a singular $K3$ surface of degree $6$. 
Let $\pi:X\rightarrow Y$ be the blow-up of $Y$ at $y_0$
and $E\subset X$ its exceptional divisor.
$X$ is a smooth $K3$ surface.
Set $L_0:=\pi^*H\otimes \StructureSheaf{X}(2E)$.
Then $L_0$ has degree $-2$, but $L_0$ is not exceptional. 
Set $c:=c_1(L_0)$, let $D_c\subset Def(X)$ be the divisor
defined in equation (\ref{eq-wall}), and denote by $L_t$ the line bundle
on $X_t$ with class $c_t$, $t\in D_c$.
Then $L_t$ has degree $-2$, and thus precisely one of $L_t$ or
$L_t^{-1}$ is effective (\cite{BHPV}, Ch. VIII Prop. 3.6). 
The semi-continuity theorem implies that $L_t$ is effective, 
since $L_0^{-1}$ isn't and $D_c$ is connected.
For a generic $t\in D_c$, the pair $(X_t,L_t)$
consists of a K\"{a}hler $K3$ surface, whose Picard group is generated by 
$L_t$. Hence, the linear system $\linsys{L_t}$ consists of a 
single smooth rational curve $E_t$. 
The analogue of Lemma \ref{lem-effectivity-along-a-wall} 
is known for such a pair $(X_t,L_t)$,
even if $X_t$ is not projective. Hence, Proposition
\ref{prop-generic-prime-exceptional} applies as well. 
Let $D_c^0\subset D_c$ be the subset
of pairs $(X_t,L_t)$, such that $L_t\cong\StructureSheaf{X_t}(E_t)$,
for a smooth connected rational curve $E_t$. 
We get that $D_c^0$ is non-empty and the complement $Z:=D_c\setminus D_c^0$
is a closed analytic subset containing $0\in D_c$.
Consequently, the property of $L_t$ being exceptional is not closed. 
%
\end{example}

%
\subsection{Deformation equivalence and  Torelli}
\label{sec-deformation-equivalence-and-torelli}
We introduce and relate three notions of deformation equivalence of pairs.

\begin{defi}
\label{def-deformation-equivalent-pairs-with-effective-divisors}
Let $(X_i,E_i)$, $i=1,2$, be two pairs of an irreducible holomorphic 
symplectic
manifold $X_i$, and an effective divisor $E_i\in Div(X_i)$. 
The two pairs are said to be {\em deformation equivalent}, 
if there exists a smooth proper family
$\pi:\X\rightarrow T$ of irreducible holomorphic symplectic
manifolds, 
over a connected analytic space $T$ with finitely many irreducible components, 
a holomorphic 
line bundle $\LB$ over $\X$, a nowhere-vanishing section $s$ of $\pi_*\LB$, 
points $t_i\in T$, and isomorphisms $f_i:\X_{t_i}\rightarrow X_i$,
such that $f_i((s_{t_i})_0)=E_i$, $i=1,2$. Above $(s_{t_i})_0$ denotes the zero
divisor of $s_{t_i}$ in $X_{t_i}$.
\end{defi}

The relation is clearly symmetric and reflexive. It is also transitive,
since we allow $T$ to be reducible.

\begin{defi}
\label{def-deformation-equivalent-pairs-with-line-bundle}
Let $(X_i,L_i)$, $i=1,2$, be two pairs of an irreducible holomorphic 
symplectic manifold $X_i$, 
and a line bundle $L_i$. 
The two pairs are said to be {\em deformation equivalent}, 
if there exists a smooth proper family
$\pi:\X\rightarrow T$ of irreducible holomorphic symplectic manifolds, 
over a connected analytic space $T$ with finitely many irreducible components, 
and a section $e$ of $R^2\pi_*\Integers$, 
which is everywhere of Hodge-type $(1,1)$, 
points $t_i\in T$, and isomorphisms 
$f_i:\X_{t_i}\rightarrow X_i$, such that $(f_i)_*(e_{t_i})=c_1(L_i)$.
\end{defi}

\begin{defi}
\label{def-deformation-equivalent-pairs-with-cohomology-class}
Let $(X_i,e_i)$, $i=1,2$, be two pairs of an irreducible holomorphic 
symplectic manifold $X_i$, 
and a class $e_i\in H^2(X_i,\Integers)$. 
The two pairs are said to be {\em deformation equivalent}, 
if there exists a smooth proper family
$\pi:\X\rightarrow T$ of irreducible holomorphic symplectic
manifolds, 
over a connected analytic space $T$ with finitely many irreducible components, 
a section $e$ of
$R^2\pi_*\Integers$, points $t_i\in T$, and isomorphisms 
$f_i:\X_{t_i}\rightarrow X_i$, such that $(f_i)_*(e_{t_i})=e_i$.
\end{defi}

The three 
Definitions, \ref{def-deformation-equivalent-pairs-with-line-bundle},
\ref{def-deformation-equivalent-pairs-with-effective-divisors}, 
and \ref{def-deformation-equivalent-pairs-with-cohomology-class},
fit in a hierarchy. If $L_i=\StructureSheaf{X_i}(E_i)$, 
and $e_i:=c_1(L_i)$, then 
\begin{eqnarray}
\nonumber
(X_1,E_1)\equiv (X_2,E_2) & \Rightarrow & (X_1,L_1)\equiv (X_2,L_2),
\\
\label{eq-equivalence-of-line-bundles-implies-that-of-classes}
(X_1,L_1)\equiv (X_2,L_2) & \Rightarrow & (X_1,e_1)\equiv (X_2,e_2).
\end{eqnarray}

Assume that the divisors $E_i$, $i=1,2$, are prime exceptional,
and $X_1$ is projective. 
Then both implications above are 
equivalences,
\WithoutTorelli{assuming an affirmative answer to a 
version of the Torelli Question,}
by Proposition \ref{equivalence-of-deformation-equivalences-relations}.
For a qualified ``converse'' to the second implication 
(\ref{eq-equivalence-of-line-bundles-implies-that-of-classes}),
without assuming that $L_1$ and $L_2$ are effective,
see Lemma \ref{lem-monodromy-invariants-and-deformation-equivalence}.

%
\subsubsection{Period maps}
A {\em marking}, for an irreducible holomorphic symplectic manifold $X$, 
is a choice of an isometry
$\eta: H^2(X,\Integers)\rightarrow \Lambda$ with a fixed lattice $\Lambda$.
The {\em period}, of the marked manifold $(X,\eta)$, is the
line $\eta[H^{2,0}(X)]$ considered as a point in the projective space
$\PP[\Lambda\otimes\ComplexNumbers]$. The period lies in the period domain 
\begin{equation}
\label{eq-period-domain}
\Omega \ := \ \{
\ell \ : \ (\ell,\ell)=0 \ \ \ \mbox{and} \ \ \ 
(\ell,\bar{\ell}) > 0
\}.
\end{equation}
$\Omega$ is an open subset, in the classical topology, of the quadric in 
$\PP[\Lambda\otimes\ComplexNumbers]$ of isotropic lines \cite{beauville}. 

There is a (non-Hausdorff) moduli space ${\mathfrak M}_\Lambda$ of marked 
irreducible holomorphic symplectic manifolds, 
with a second integral cohomology 
lattice isometric to $\Lambda$ \cite{huybrects-basic-results}. 
The period map 
\begin{eqnarray}
\label{eq-period-map}
P \ : \ {\mathfrak M}_\Lambda & \longrightarrow & \Omega,
\\
\nonumber
(X,\eta) & \mapsto & \eta[H^{2,0}(X)]
\end{eqnarray}
is a local isomorphism, by the Local Torelli Theorem \cite{beauville}.
The Surjectivity Theorem states that the restriction of the period map
to every connected component of ${\mathfrak M}_\Lambda$ is surjective 
\cite{huybrects-basic-results}. 

Two points $x$ and $y$ of a topological space are {\em inseparable}, 
if every pair of open subsets $U$, $V$, with $x\in U$ and $y\in V$, have a non-empty intersection.
\WithTorelli{
Assume given  a bimeromorphic map 
$f:X_1\rightarrow X_2$ and a marking $\eta_1$ for $X_1$. Let 
$f^* : H^2(X_2,\Integers) \rightarrow H^2(X_1,\Integers)$ 
be the isometry induced by the bimeromorphic map $f$ 
(see the proof of Corollary \ref{cor-1}).
Set $\eta_2=\eta_1\circ f^*$. Then $(X_1,\eta_1)$ and $(X_2,\eta_2)$
are inseparable points of ${\mathfrak M}_\Lambda$ 
(\cite{huybrechts-kahler-cone}, Theorem 2.5).
Conversely, Verbitsky recently proved
the following version of the  Torelli Theorem.

\begin{thm} 
\label{thm-torelli}
(\cite[Theorem 4.24]{verbitsky}, \cite{huybrechts-bourbaki}). 
Let ${\mathfrak M}^0_\Lambda$ be a connected component of 
${\mathfrak M}_\Lambda$. If $(X_1,\eta_1)$ and $(X_2,\eta_2)$
are two pairs in ${\mathfrak M}^0_\Lambda$ and
$P(X_1,\eta_1)=P(X_2,\eta_2)$, then $(X_1,\eta_1)$ and $(X_2,\eta_2)$ 
are inseparable points of ${\mathfrak M}^0_\Lambda$.
\end{thm}
}
A homomorphism $h:H^*(X_1,\Integers)\rightarrow H^*(X_2,\Integers)$ is a
{\em parallel transport operator} if there exists a smooth and proper family $f:\X\rightarrow B$, 
of irreducible holomorphic symplectic manifolds over an analytic base $B$, points $b_1$, $b_2$ in $B$, isomorphisms
$X_i\cong \X_{b_i}$, and a continuous path $\gamma$ from $b_1$ to $b_2$, such 
that parallel transport in the local system $R^*f_*\Integers$ along $\gamma$ induces $h$.
The following is a fundamental result of Huybrechts.
\begin{thm}
\label{thm-inseparable-are-birational}
(\cite{huybrects-basic-results}, Theorem 4.3).
Let $(X_1,\eta_1)$ and $(X_2,\eta_2)$ be two inseparable points of 
${\mathfrak M}_\Lambda$, with $\dim(X_i)=2n$. Then there exists an effective cycle
$\Gamma:=Z+\sum Y_j$ in $X_1\times X_2$, of pure dimension $2n$,
with the following properties.
\begin{enumerate}
\item
$Z$ is the graph of a bimeromorphic map from $X_1$ to $X_2$.
\item
The correspondence $[\Gamma]_*:H^*(X_1,\Integers)\rightarrow H^*(X_2,\Integers)$
is a parallel transport operator. Furthermore, the 
composition 
\[
\eta_2^{-1}\circ\eta_1:H^2(X_1,\Integers)\rightarrow H^2(X_2,\Integers)
\]
is equal to the restriction of $[\Gamma]_*$.
\item
The image of the projection of each $Y_j$ into each $X_i$, $i=1,2$, 
has positive codimension in $X_i$. 
\end{enumerate}

\end{thm}

Assume given two deformation equivalent pairs $(X_i,e_i)$, $i=1,2$,
in the sense of Definition
\ref{def-deformation-equivalent-pairs-with-cohomology-class}.
Then there exist isometries $\eta_i:H^2(X_i,\Integers)\rightarrow \Lambda$,
having the following two properties: 
\begin{enumerate}
\item
$\eta_1(e_1)=\eta_2(e_2)$.
\item
The marked pairs $(X_i,\eta_i)$ belong to the same connected component 
${\mathfrak M}^0_\Lambda$.
\end{enumerate}
Let $\lambda$ be the common value $\eta_i(e_i)$, $i=1,2$.
If both classes $e_i$ belong to $H^{1,1}(X_i,\Integers)$, 
then the periods $P(X_i,\eta_i)$ belong to the hyperplane
$\lambda^\perp\subset \PP[\Lambda\otimes\ComplexNumbers]$ orthogonal to 
$\lambda$.
The intersection $\lambda^\perp\cap\Omega$ is connected.

Fix a primitive non-zero class $\lambda\in \Lambda$
with $(\lambda,\lambda)<0$. Let 
\[
{\mathfrak M}^0_{\Lambda,\lambda} \ \ \ \subset \ \ \ {\mathfrak M}^0_\Lambda
\]
be the subset parametrizing marked pairs $(X,\eta)$, such that 
$\eta^{-1}(\lambda)$ is of Hodge type $(1,1)$, and 
$(\kappa,\eta^{-1}(\lambda))>0$, for some 
K\"{a}hler class $\kappa$ on $X$. 

\begin{claim}
${\mathfrak M}^0_{\Lambda,\lambda}$ is an open subset 
of $P^{-1}(\lambda^\perp\cap\Omega)$.
\end{claim}

\begin{proof}
Let ${\mathfrak M}^0_+$ be the subset of ${\mathfrak M}^0_\Lambda$,
consisting of marked pairs $(X,\eta)$, such that 
$(\kappa,\eta^{-1}(\lambda))>0$, for some K\"{a}hler class $\kappa$ on $X$. 
It suffices to prove that  ${\mathfrak M}^0_+$ is an open subset
of ${\mathfrak M}^0_\Lambda$. 
Let $(X_0,\eta_0)$ be a point of ${\mathfrak M}^0_+$ 
and $\kappa_0$ a K\"{a}hler class on $X_0$ satisfying 
$(\kappa_0,\eta_0^{-1}(\lambda))>0$. 
Let $Def(X_0)$ be the Kuranishi deformation space and
$\psi:\X\rightarrow Def(X_0)$ the semi-universal family with fiber $X_0$ over
$0\in Def(X_0)$. 
There exists an open subset $U$ of $Def(X_0)$, and a differentiable section
$\kappa$ of the real vector bundle 
$(R^2_{\psi_*}\RealNumbers\restricted{)}{U}$, 
over $U$, such that $\kappa_t$ is a K\"{a}hler class of $X_t$, for all
$t\in U$, by the proof of the Stability of
K\"{a}hler manifolds (\cite{voisin-book-vol1}, Theorem 9.3.3). 
We may identify $U$ with an open subset of ${\mathfrak M}^0_\Lambda$
containing $(X_0,\eta_0)$, by the Local Torelli Theorem.
We get the continuous function $(\kappa_t,\eta_t^{-1}(\lambda))$
over $U$, which is positive at $(X_0,\eta_0)$. 
Hence, there is an open subset $W\subset U$, containing $(X_0,\eta_0)$, 
such that $(\kappa_t,\eta_t^{-1}(\lambda))>0$, for all $t\in W$.
\end{proof}

The Local Torelli Theorem 
implies that the period map restricts to a local isomorphism 
\[
P_\lambda \ : \ {\mathfrak M}^0_{\Lambda,\lambda} \ \ \ 
\longrightarrow \ \ \ \lambda^\perp\cap \Omega.
\]
${\mathfrak M}^0_{\Lambda,\lambda}$ is thus a non-Hausdorff
smooth complex manifold of dimension $b_2(X)-3$. 
%
\subsubsection{${\mathfrak M}^0_{\Lambda,\lambda}$ is path-connected 
\WithoutTorelli{if  Torelli holds}} \hspace{1ex}\\
Given a point $t\in \Omega$, set 
$
\Lambda_t^{1,1} := \{x\in \Lambda \ : \ (x,t)=0\}.
$
\WithTorelli{
The following statement is a Corollary of Theorem \ref{thm-torelli}.
}
\WithTorelli{\begin{cor}}
\WithoutTorelli{\begin{lem}}
\label{cor-torelli-for-M-Lambda-lambda}
Let $t\in \lambda^\perp\cap\Omega$ be a point, such that 
$\Lambda_t^{1,1}={\rm span}_\Integers\{\lambda\}$. 
Then the fiber $P_\lambda^{-1}(t)$
consists of the isomorphism class of
a single marked \WithTorelli{pair.}
\WithoutTorelli{pair, assuming an affirmative answer to Question
\ref{thm-torelli}.}
\WithTorelli{\end{cor}}
\WithoutTorelli{\end{lem}}

\begin{proof}
Let $(X,\eta)$ be a marked pair in $P_\lambda^{-1}(t)$.
Set $\tilde{\lambda}:=\eta^{-1}(\lambda)$.
Then $H^{1,1}(X,\Integers)$ is spanned by $\tilde{\lambda}$, and
there exists a K\"{a}hler class $\kappa_0$, 
such that $(\kappa_0,\tilde{\lambda})>0$. 
Let us first describe the three possibilities for the K\"{a}hler
cone $\K_X$ and the 
{\em birational K\"{a}hler cone} $\BK_X$ of $X$.
Recall that $\BK_X$ is the union of the subsets 
$f^*(\K_Y)\subset H^{1,1}(X,\RealNumbers)$, as $f$ varies over 
all bimeromorphic maps $f:X\rightarrow Y$ from X to another 
irreducible holomorphic symplectic manifold $Y$.
Denote by $\tilde{\lambda}^\vee$ the primitive class in $H^2(X,\Integers)^*$,
which is a positive multiple of $(\tilde{\lambda},\bullet)$.
Let $\C_X$ be the connected component of the cone
$\{\kappa\in H^{1,1}(X,\RealNumbers) \ : \ (\kappa,\kappa)>0\}$,
which contains the K\"{a}hler cone.

\underline{Case 1:} If $d\lambda^\vee$ is not represented by a rational curve,
for any positive integer $d$, then
$
\BK_X=\K_X=\C_X,
$
by \cite{huybrechts-kahler-cone}, Corollary 3.3.

\underline{Case 2:} Assume that $\tilde{\lambda}$ is 
$\RationalNumbers$-effective.
Then $d\tilde{\lambda}$ is represented by a prime 
exceptional divisor $E\subset X$, for some positive integer $d$,
which is uniruled, by \cite{boucksom}, Proposition 4.7.
Then 
\[
\BK_X=\K_X=\{\kappa\in \C_X \ : \ (\kappa,\tilde{\lambda})>0\},
\]
by \cite{boucksom}, Theorem 4.3. 

\underline{Case 3:} Assume that 
$d\tilde{\lambda}$ is not effective, for any
positive integer $d$, but $d\tilde{\lambda}^\vee$ is represented by a rational
curve, for some positive integer $d$. Then
\begin{eqnarray*}
\K_X & = & \{\kappa\in \C_X \ : \  (\kappa,\tilde{\lambda})>0\},
\\
\BK_X & = & \K_X\cup \K'_X, \ \ \mbox{where} 
\\
\K'_X & = & \{\kappa\in \C_X \ : \ (\kappa,\tilde{\lambda})<0\},
\end{eqnarray*} 
by \cite{boucksom}, Theorem 4.3. 

Let $(X_1,\eta_1)$ and $(X_2,\eta_2)$ be two marked pairs in 
$P_\lambda^{-1}(t)$.
Then $(X_1,\eta_1)$ and $(X_2,\eta_2)$ are inseparable points, 
by \WithTorelli{Theorem} 
\WithoutTorelli{the assumed affirmative answer to Question}
\ref{thm-torelli}. Hence, there exists a cycle $\Gamma:=Z+\sum Y_j$
in $X_1\times X_2$, satisfying the properties listed in Theorem
\ref{thm-inseparable-are-birational}. Denote by $g:X_1\rightarrow X_2$
the bimeromorphic map with graph $Z$, and let
$f:H^2(X_1,\Integers)\rightarrow H^2(X_2,\Integers)$ be the 
parallel-transport operator $[\Gamma]_*$, so that 
$f=\eta_2^{-1}\circ\eta_1$, by Theorem \ref{thm-inseparable-are-birational}. 
Set $\tilde{\lambda}_i:=\eta_i^{-1}(\lambda)$.
Let $\kappa_i$ be a K\"{a}hler class on $X_i$, 
such that $(\tilde{\lambda}_i,\kappa_i)>0$.

In cases 1 and 3, $X_i$ does not contain any effective divisor, $i=1,2$.
In particular, the image of each $Y_j$ has codimension $\geq 2$ in each $X_i$, and 
$f=g_*$. 
We have 
$
(\tilde{\lambda}_1,g^*(\kappa_2))=(\eta_1^{-1}(\lambda),g^*(\kappa_2))=
(g_*\eta_1^{-1}(\lambda),\kappa_2)=(\eta_2^{-1}(\lambda),\kappa_2)=
(\tilde{\lambda}_2,\kappa_2),
$
since $g^*$ is an isometry. We conclude the inequality
\begin{equation}
\label{eq-lambda-tilde-1-paired-with-pull-back-of-kappa-2-is-positive}
(\tilde{\lambda}_1,g^*(\kappa_2)) \ \ > \ \ 0.
\end{equation}

If $g^*(\kappa_2)$ is not a K\"{a}hler class, then the birational K\"{a}hler 
cone $\BK_{X_1}$ consists of at least two connected components. Thus we 
must be in case 3. 
So $\kappa_1\in \K_{X_1}$ and $g^*(\kappa_2)$ belongs to $\K'_{X_1}$.
Hence, $(g^*(\kappa_2),\tilde{\lambda}_1)<0$, by the characterization of
$\K'_{X_1}$. This contradicts the inequality 
(\ref{eq-lambda-tilde-1-paired-with-pull-back-of-kappa-2-is-positive}).

We conclude that $g^*(\kappa_2)$ is a K\"{a}hler class. 
Thus $g$ is an isomorphism, 
by \cite{huybrechts-kahler-cone}, Proposition 2.1, and
$(X_1,\eta_1)$ and $(X_2,\eta_2)$ are isomorphic marked pairs.

It remains to treat case 2. In that case $\BK_{X_i}=\K_{X_i}$, and so $g$ 
is an isomorphism. Hence, 
$g_*(\tilde{\lambda}_1)=\tilde{\lambda}_2$, since the classes 
$\tilde{\lambda}_i$ are effective.
On the other hand,
$f(\tilde{\lambda}_1)=\eta_2^{-1}\eta_1(\eta_1^{-1}(\lambda))=\eta_2^{-1}(\lambda)=
\tilde{\lambda}_2$.
Hence, $f(\tilde{\lambda}_1)=g_*(\tilde{\lambda}_1)$.
The subspace $\tilde{\lambda}_1^\perp$, orthogonal to $\tilde{\lambda}_1$,
is necessarily in the kernel of
the correspondence $[\sum Y_j]_*:H^2(X_1,\Integers)\rightarrow H^2(X_2,\Integers)$. 
Hence, $f(\alpha)=g_*(\alpha)$, for all $\alpha\in \tilde{\lambda}_1^\perp$. 
We conclude that $f=g_*$, and the two pairs $(X_1,\eta_1)$ and $(X_2,\eta_2)$
 are isomorphic.
\end{proof}

%
\WithoutTorelli{
\begin{question}
\label{question-connectedness}
Let $(X_0,\eta_0)\in {\mathfrak M}^0_\Lambda$ be a marked pair, 
$L$ a monodromy reflective line bundle on $X$, 
and set $\lambda:=\eta_0(c_1(L))$. 
Is ${\mathfrak M}^0_{\Lambda,\lambda}$
a path-connected subset of ${\mathfrak M}^0_\Lambda$?
\end{question}

}

\WithTorelli{\begin{cor}}
\WithoutTorelli{\begin{prop}}
\label{cor-connectedness} 
${\mathfrak M}^0_{\Lambda,\lambda}$
is a path-connected subset of \WithTorelli{${\mathfrak M}^0_\Lambda$.}
\WithoutTorelli{${\mathfrak M}^0_\Lambda$, 
assuming an affirmative answer to Question 
\ref{thm-torelli}.}
\WithTorelli{\end{cor}}
\WithoutTorelli{\end{prop}}

\begin{proof}
Let $(X,\eta)$ be a marked pair in ${\mathfrak M}^0_{\Lambda,\lambda}$.
Then there exists a continuous path from $(X,\eta)$ 
to some $(X_0,\eta_0)$, where $H^{1,1}(X_0,\Integers)$ is spanned by 
$\eta_0^{-1}(\lambda)$, by the Local Torelli Theorem. 
Hence, it suffices to construct a continuous path between any two
pairs $(X_0,\eta_0)$ and $(X_1,\eta_1)$ in
${\mathfrak M}^0_{\Lambda,\lambda}$, such that
$H^{1,1}(X_i,\Integers)$ is cyclic, for $i=0,1$. 
Set $t_i:=P(X_i,\eta_i)$, $i=0,1$.
Let $I$ be the closed interval $[0,1]$. 
Let $\gamma:I\rightarrow \lambda^\perp\cap\Omega$ be a continuous path
from $t_0$ to $t_1$. Let $I_1\subset I$ be the subset of points $t$,
such that $\Lambda_{\gamma(t)}^{1,1}$ is cyclic. 
We may choose $\gamma$ so that $I_1$ is a dense subset of $I$.

For each $t\in I_1$, there exists a unique isomorphism class 
of a marked pair $(X_t,\eta_t)$ in ${\mathfrak M}^0_{\Lambda,\lambda}$ with period $\gamma(t)$,
by \WithTorelli{Corollary} \WithoutTorelli{Lemma} 
\ref{cor-torelli-for-M-Lambda-lambda}.
Choose an open path-connected 
subset $U_t\subset {\mathfrak M}^0_{\Lambda,\lambda}$, 
containing $(X_t,\eta_t)$,
such that $P_\lambda$ restricts to $U_t$ as an open embedding.
This is possible, by the Local Torelli Theorem.
We get the open covering $\gamma(I)\subset \cup_{t\in I_1} P_\lambda(U_{t})$.
Choose a finite sub-covering $\cup_{j=0}^N P_\lambda(U_{t_j})$ of 
$\gamma(I)$,
with $0=t_0<t_1< \ \cdots \ < t_N=1$. 
Choose an increasing subsequence $\tau_j:=t_{i_j}$, $0\leq j\leq k$,
such that $\tau_0=t_0$, $\tau_k=t_N$, and 
$P_\lambda(U_{\tau_j})\cap P_\lambda(U_{\tau_{j+1}})$ is non-empty.
Choose points $s_{j,j+1}$ in 
$P_\lambda(U_{\tau_j})\cap P_\lambda(U_{\tau_{j+1}})$, such that 
$\Lambda^{1,1}_{s_{j,j+1}}$ is cyclic, and let 
$\tilde{s}_{j,j+1}$ be the unique point of 
${\mathfrak M}^0_{\Lambda,\lambda}$ over $s_{j,j+1}$.
Then $\tilde{s}_{j,j+1}$ belongs to $U_{\tau_j}\cap U_{\tau_{j+1}}$.
Choose continuous paths $\alpha_j$ in $U_{\tau_j}$ 
from $\tilde{s}_{j-1,j}$ to 
$(X_{\tau_j},\eta_{\tau_j})$, if $j>0$, 
and $\beta_j$  in $U_{\tau_j}$ from $(X_{\tau_j},\eta_{\tau_j})$
to $\tilde{s}_{j,j+1}$, if $j<k$. Then the concatenated path 
$\beta_0\alpha_1\beta_1 \cdots \alpha_{k-1}\beta_{k-1}\alpha_k$
is a continuous path from the isomorphism class
of $(X_0,\eta_0)$ to that of  $(X_1,\eta_1)$. 
\end{proof}

%
\subsubsection{Deformation equivalent monodromy-reflective line bundles
are simultaneously stably-$\RationalNumbers$-effective or 
not stably-$\RationalNumbers$-effective}

\begin{prop}
\label{equivalence-of-deformation-equivalences-relations}
Let $(X_1,E_1)$, $(X_2,E_2)$ be two pairs, of irreducible holomorphic
symplectic manifolds $X_i$ and prime exceptional divisors $E_i\subset X_i$. 
Assume that $X_1$ is projective and 
$(X_1,[E_1])$ is deformation equivalent to $(X_2,[E_2])$, in the
sense of Definition
\ref{def-deformation-equivalent-pairs-with-cohomology-class}.
\WithoutTorelli{
Assume that for some, and hence any, marking $\eta_1$ of $X_1$, 
Question \ref{question-connectedness} 
has an affirmative answer for the connected component 
${\mathfrak M}^0_\Lambda$ of $(X_1,\eta_1)$ and for $\lambda=\eta_1([E_1])$.
}
Then $(X_1,E_1)$ and $(X_2,E_2)$ are deformation equivalent 
in the sense of Definition 
\ref{def-deformation-equivalent-pairs-with-effective-divisors}.
\end{prop}

\begin{proof}
We assume, for simplicity of notation, that the class $[E_i]$
is primitive. The generalization of the proof to the case $[E_i]=2e_i$ is 
straightforward.
As noted above, we can choose a marking $\eta_2$ of $X_2$, such that 
$(X_2,\eta_2)$ belongs to ${\mathfrak M}^0_\Lambda$ and 
$\eta_2([E_2])=\lambda$.
Any K\"{a}hler class $\kappa$ on $X_2$ satisfies
$(\kappa,[E_2])>0$, since $E_2$ is effective \cite{huybrects-basic-results}. 
Hence, $(X_2,\eta_2)$ belongs to 
${\mathfrak M}^0_{\Lambda,\lambda}$.

Choose a continuous path 
$\gamma:[0,1]\rightarrow {\mathfrak M}^0_{\Lambda,\lambda}$
from $(X_1,\eta_1)$ to $(X_2,\eta_2)$.
Further choose a sufficiently fine partition of the unit interval 
\[
0=t_0<t_1 < \ \cdots \ < t_N=1
\]
and open connected subsets
$U_i\subset {\mathfrak M}^0_{\Lambda,\lambda}$, $1\leq i \leq N$, 
such that $\gamma([t_{i-1},t_i])$ is contained in $U_i$,
and the restriction of $P$ to $U_i$ is an open embedding 
$P_i:U_i\hookrightarrow \lambda^\perp\cap\Omega$.
This is possible by the Local Torelli Theorem.

\begin{claim}
For each $1\leq i \leq N-1$, there exists a marked 
pair $(Y_i,\varphi_i)$ in $U_i\cap U_{i+1}$, such that $Y_i$ is 
projective, and $\varphi_i^{-1}(\lambda)$ is the class of a 
prime exceptional divisor on $Y_i$.
\end{claim}

\begin{proof}
Following is an iterative process of constructing the pairs 
$(Y_i,\varphi_i)$.
Set $(Y_0,\varphi_0)=(X_1,\eta_1)$. 
Assume that $i=1$, or that $1<i\leq N-1$ and 
$(Y_j,\varphi_j)$ exists for all $1\leq j<i$. 
The pair $(Y_{i-1},\varphi_{i-1})$ belongs to $U_i$. 
Proposition \ref{prop-generic-prime-exceptional}
implies that there exists a closed analytic subvariety $Z_i\subset U_i$,
not containing $(Y_{i-1},\varphi_{i-1})$, 
such that for every $(X,\eta)$ in $U_i\setminus Z_i$,
$\eta^{-1}(\lambda)$ is the class of a prime exceptional divisor $E\subset X$.
The locus of projective marked pairs is dense in $U_i\cap U_{i+1}$,
by  \cite{huybrechts-norway}, Proposition 21. 
Hence, there exists a projective
pair $(Y_i,\varphi_i)$ in $[U_i\setminus Z_i]\cap U_{i+1}$.
\end{proof}

Set $(Y_N,\eta_N):=(X_2,\eta_2)$.
Let $D_i\subset Y_i$ 
be the prime exceptional divisor with $[D_i]=\eta_i^{-1}(\lambda)$.
It remains to prove that $(Y_{i-1},D_{i-1})$ 
is deformation equivalent to $(Y_i,D_i)$, for $1\leq i \leq N$. 
Both pairs $(Y_{i-1},\varphi_{i-1})$ and $(Y_i,\varphi_i)$
belong to $U_i\setminus Z_i$, by construction, for $i<N$, 
and by the characterization
of $Z_N$ in Proposition \ref{prop-generic-prime-exceptional}, for $i=N$.
Proposition \ref{prop-generic-prime-exceptional}
exhibits a divisor $\E_i$ in the restriction of $\X$ to
$U_i\setminus Z_i$, whose fiber over the pair $(Y_i,\varphi_i)$ is $D_i$,
and whose fiber over the  pair $(Y_{i-1},\varphi_{i-1})$ is $D_{i-1}$. 
This completes the proof of Proposition
\ref{equivalence-of-deformation-equivalences-relations}.
\end{proof}

The following variant of Proposition
\ref{equivalence-of-deformation-equivalences-relations}
will be used in the derivation of
\WithTorelli{Theorem}\WithoutTorelli{Conjecture}
\ref{conj-exceptional-line-bundles}
from Torelli.

\begin{prop}
\label{prop-main-question-on-deformation-equivalence}
Let $X$ and $Y$ be two irreducible holomorphic symplectic manifolds, 
with $X$ projective, 
$E\subset X$ a prime exceptional divisor, and
$L$ a line bundle on $Y$. Set $c:=c_1(L)$. Assume that 
$(X,[E])$ and $(Y,c)$ are deformation equivalent in the sense of Definition
\ref{def-deformation-equivalent-pairs-with-cohomology-class}.
Assume further that there exists a K\"{a}hler class $\kappa$ on $Y$,
such that $(\kappa,c)>0$.
\WithoutTorelli{
Finally assume an affirmative answer to Question \ref{question-connectedness}
with a marking $\eta$ for $X$ and with $\lambda=\eta([E])$.
}
Then $L$ is stably-prime-exceptional (in the sense of 
\WithTorelli{Theorem}\WithoutTorelli{Conjecture}
\ref{conj-exceptional-line-bundles}.)
\end{prop}

The above Proposition was proven in the course of proving
Proposition
\ref{equivalence-of-deformation-equivalences-relations}.

\begin{cor}
\label{cor-if-L-1-is-prime-exceptional-L-2-not-Q-effective-then-not-def-equiv}
Let $(X_1,L_1)$ and $(X_2,L_2)$ be two pairs, each of an irreducible 
holomorphic symplectic manifold $X_i$, and a monodromy-reflective
line bundle $L_i$. Set $e_i:=c_1(L_i)$.
Assume that $X_1$ is projective, $ke_1$ is the class of a prime exceptional
divisor $E_1$, for some non-zero integer $k$, and 
$H^0(X_2,L_2^d)$ vanishes, for all non-zero integers $d$. 
\WithoutTorelli{
Finally assume  an affirmative answer to Question \ref{question-connectedness}
with a marking $\eta$ for $X_1$ and with $\lambda=\eta(e_1)$.
}
Then the pairs $(X_1,e_1)$ and $(X_2,e_2)$ are not deformation equivalent,
in the sense of Definition
\ref{def-deformation-equivalent-pairs-with-cohomology-class}.
\end{cor}

\begin{proof}
If $(X_1,e_1)$ and $(X_2,e_2)$ were deformation equivalent,
in the sense of Definition
\ref{def-deformation-equivalent-pairs-with-cohomology-class},
then $H^0(X_2,L_2^d)$ would not vanish for $d=k$ or $d=-k$,  
by Proposition 
\ref{prop-main-question-on-deformation-equivalence} and
the semi-continuity theorem.
\end{proof}

%
\subsection{Deformation equivalence and monodromy-invariants}
\label{sec-deformation-equivalence-and-monodromy-invariants}
Let $Mon^2(X)$ be the monodromy group, introduced in Definition
\ref{def-Mon-2}.
Let $I(X)\subset H^2(X,\Integers)$ be a $Mon^2(X)$-invariant subset 
and let $\Sigma$ be a set. 

\begin{defi}
\label{def-faithful}
A function $f:I(X)\rightarrow \Sigma$ is a 
{\em monodromy-invariant}, if $f(e)=f(g(e))$, for all $g\in Mon^2(X)$. 
The function $f$ is said to be a {\em faithful}  
monodromy-invariant, if the function 
$\bar{f}:I(X)/Mon^2(X)\rightarrow \Sigma$,
induced by $f$, is injective. 
\end{defi}

Given an irreducible holomorphic symplectic manifold $X'$, 
deformation equivalent to $X$, 
denote by $I(X')\subset H^2(X',\Integers)$ the set of all classes
$e'$, such that $(X',e')$ is deformation equivalent to $(X,e)$,
for some $e\in I(X)$, in the sense of Definition 
\ref{def-deformation-equivalent-pairs-with-cohomology-class}.

Assume that $f:I(X)\rightarrow \Sigma$ is 
a monodromy-invariant function. Then $f$ admits a natural extension 
to a function $f:I(X')\rightarrow \Sigma$, for every 
irreducible holomorphic symplectic manifold $X'$ deformation equivalent to 
$X$. The extension is uniquely determined by the following condition.
{\em Given any smooth and proper family 
$\pi:\X\rightarrow T$, of irreducible holomorphic symplectic manifolds
deformation equivalent to $X$, 
and any flat section $e$ of the local system $R^2\pi_*\Integers$,
the function $f(e)$ is locally constant, in the classical topology of the
analytic space $T$.} We denote this extension by $f$ as well.
The following statement relates monodromy invariants to deformation
equivalence.

\begin{lem}
\label{lem-monodromy-invariants-and-deformation-equivalence}
Let $f:I(X)\rightarrow \Sigma$ be a faithful monodromy-invariant function.
Assume given two pairs $(X_i,e_i)$, $i=1,2$,
with $X_i$ deformation equivalent to $X$ and $e_i\in I(X_i)$. 
\begin{enumerate}
\item
\label{lemma-item-monodromy-invariants-and-deformation-equivalence-of-classes}
$f(e_1)=f(e_2)$ if and only if $(X_1,e_1)$ and $(X_2,e_2)$ are 
deformation equivalent, in the sense of 
Definition \ref{def-deformation-equivalent-pairs-with-cohomology-class}.
\item
\label{lemma-item-monodromy-invariants-and-deformation-equivalence-of-lb}
Assume that $f(e_1)=f(e_2)$,  $e_i=c_1(L_i)$, 
for holomorphic line bundles $L_i$ on $X_i$, 
and there exist K\"{a}hler classes $\kappa_i$ on $X_i$, satisfying
$(\kappa_i,e_i)>0$, for $i=1,2$. 
\WithoutTorelli{Assume, furthermore, an affirmative answer to Question
\ref{question-connectedness}
with a marking $\eta$ for $X_1$ and with 
$\lambda=\eta(e_1)$.} 
Then $(X_1,L_1)$ is deformation equivalent to $(X_2,L_2)$,
in the sense of Definition 
\ref{def-deformation-equivalent-pairs-with-line-bundle}.
\end{enumerate}
\end{lem}

\begin{proof}
Part 
\ref{lemma-item-monodromy-invariants-and-deformation-equivalence-of-classes}
is evident. Part
\ref{lemma-item-monodromy-invariants-and-deformation-equivalence-of-lb}
follows from part 
\ref{lemma-item-monodromy-invariants-and-deformation-equivalence-of-classes}
and 
\WithTorelli{Corollary \ref{cor-connectedness}.}
\WithoutTorelli{the assumed connectedness of
${\mathfrak M}^0_{\Lambda,\lambda}$.}
\end{proof}


%
\section{Monodromy-invariants from Mukai's isomorphism}
\label{sec-Mukai}
Let $S$ be a $K3$ surface and $M$ a smooth and projective moduli space of
stable coherent sheaves on $S$.
In section \ref{sec-Mukai-notation} we recalled Mukai's embedding
$\theta^{-1}:H^2(M,\Integers)\rightarrow K(S)$, 
of the second cohomology  of $M$, 
as a sub-lattice of the Mukai lattice. 
In section \ref{sec-a-rank-two-sub-lattice-of-the-Mukai-lattice}
we use this embedding to define a monodromy invariant of a
class in $H^2(M,\Integers)$. The values of this monodromy invariant, 
for monodromy-reflective classes, are 
calculated in sections \ref{sec-isometry-orbits} and 
\ref{sec-invariant-rs}. 
%
\subsection{A rank two sub-lattice of the Mukai lattice}
\label{sec-a-rank-two-sub-lattice-of-the-Mukai-lattice}
Let $\widetilde{\Lambda}$ be the unimodular lattice 
$E_8(-1)^{\oplus 2}\oplus U^{\oplus 4}$, where $U$ is the rank two unimodular 
hyperbolic lattice. 
$\widetilde{\Lambda}$ is isometric 
to the Mukai lattice of a $K3$ surface. 
Let $X$ be an irreducible holomorphic symplectic 
manifold of $K3^{[n]}$-type, $n\geq 2$.
Choose an embedding $\iota:H^2(X,\Integers)\hookrightarrow \widetilde{\Lambda}$
in the canonical $O(\widetilde{\Lambda})$-orbit of $X$ provided
by Theorem \ref{thm-a-natural-orbit-of-embeddings-of-H-2-in-Mukai-lattice}. 
Let $v$ be a generator of the rank $1$ 
sub-lattice of $\widetilde{\Lambda}$ orthogonal to the image of $\iota$. 
Then $(v,v)=2n-2$. Let $e$ be a primitive class in 
$H^2(X,\Integers)$ satisfying $(e,e)=2-2n$. We get the 
sub-lattice 
\[
L \ \ := \ \ {\rm span}_\Integers\{e,v\} \ \ \subset \ \ 
\widetilde{\Lambda},
\]
where we denote by $e$ also the element $\iota(e)$. 
Let 
\begin{equation}
\label{eq-saturation-of-L}
\widetilde{L}
\end{equation}
be the saturation of $L$ in $\widetilde{\Lambda}$.
Note that the pair $(\widetilde{L},e)$ determines the lattice $L$ via the 
equality $L=\Integers e+[e^\perp\cap \widetilde{L}]$.

\begin{defi}
\label{def-isometric-pairs}
Two pairs $(L_i,e_i)$, $i=1,2$, each consisting of a lattice
$L_i$ and a class $e_i\in L_i$, are said to be {\em isometric},
if there exists an isometry $g:L_1\rightarrow L_2$, such that 
$g(e_1)=e_2$.
\end{defi}

\begin{rem}
Let $L_0$ be a lattice. The set of isometry classes of pairs $(L_1,e_1)$, with $L_1$ isometric to $L_0$, 
is in natural bijection with
the orbit set $L_0/O(L_0)$. The bijection sends
the isometry class of $(L_1,e_1)$ to the orbit $O(L_0)g(e_1)$, where $g:L_1\rightarrow L_0$ is some isometry.
The orbit $O(L_0)g(e_1)$ is independent of the choice of $g$.
\end{rem}

Let $U$ be the rank $2$ even unimodular hyperbolic lattice.
Let $U(2)$ be the rank $2$ lattice with Gram-matrix 
$\left(\begin{array}{cc}
0 & -2 \\
-2 & 0
\end{array}
\right)$.
Let $H_{ev}$ be the rank $2$ lattice with Gram-matrix 
$\left(\begin{array}{cc}
2 & 0 \\
0 & -2
\end{array}
\right)$.
Let $I''_n(X)\subset H^2(X,\Integers)$ be the subset of primitive
classes of degree $2-2n$, such that 
$\div(e,\bullet)=n-1$ or $\div(e,\bullet)=2n-2$. 
Let $I_n(\widetilde{L})\subset \widetilde{L}$ be the subset of primitive
classes of degree $2-2n$. Let $\rho$ be the largest positive integer,
such that $(e+v)/\rho$ is an integral class. Define the integer $\sigma$ 
similarly using  $(e-v)$.

\begin{prop}
\label{prop-isometry-class-of-tilde-L-e-is-a-faithful-mon-invariant}
\begin{enumerate}
\item
\label{prop-item-three-possible-rank-2-lattices}
The isometry class of the lattice $\widetilde{L}$ is determined as follows.
\[
\widetilde{L} \ \ \cong \ \ 
\left\{
\begin{array}{ccl}
U & \mbox{if} & \div(e,\bullet)=2n-2,
\\
H_{ev} & \mbox{if} & \div(e,\bullet)=n-1 \ \mbox{and} \ n \ \mbox{is even},
\\
U(2)  & \mbox{if} & \div(e,\bullet)=n-1 \ \mbox{and} \ n \ \mbox{is odd}, \
n\not\equiv 1\ (\mbox{mod} \ 8).
\\
U(2)  & \mbox{if} & \div(e,\bullet)=n-1, 
n\equiv 1\ (\mbox{mod} \ 8) \ \mbox{and} \ \rho\sigma=2n-2.
\\
H_{ev}  & \mbox{if} & \div(e,\bullet)=n-1, 
n\equiv 1\ (\mbox{mod} \ 8) \ \mbox{and} \ \rho\sigma=n-1.
\end{array}
\right.
\]
\item
Consider the function 
\[
f  :  I''_n(X)  \ \ \longrightarrow \ \ 
I_n(U)/O(U) \ \cup \ I_n(U(2))/O(U(2)) \ \cup \ I_n(H_{ev})/O(H_{ev}),
\]
which sends the pair $(X,e)$, $e\in I''_n(X)$, to the
isometry class of the pair $(\widetilde{L},e)$, 
consisting of the primitive sub-lattice 
$\widetilde{L}\subset \widetilde{\Lambda}$, given in equation
(\ref{eq-saturation-of-L}), and the class $e\in I_n(\widetilde{L})$. 
Then $f$ is a faithful monodromy-invariant function
(Definition \ref{def-faithful}). 
\end{enumerate}
\end{prop}

The proposition is proven below in Lemmas
\ref{lem-faithful-Mon-invariant-in-case-divisibility-2n-2} and
\ref{lem-non-unimodular-rank-two-lattice}. 
We provide an explicit and easily computable classification 
of the isometry classes of the pairs $(\widetilde{L},e)$
in Lemma \ref{lemma-isometry-orbits-in-rank-2}. 

Let $L_0$ be a rank $2$ even lattice of signature $(1,1)$.
Let $I_n(L_0)\subset L_0$ be the subset of primitive classes $e$
with $(e,e)=2-2n$. 
Let $I_{L_0,n}(X)\subset I''_n(X)$ be the subset consisting
of classes $e$, such that the lattice $\widetilde{L}$ in
equation (\ref{eq-saturation-of-L}) is isometric to $L_0$.
Consider the function
\begin{equation}
\label{eq-f-from-I-L-0-n}
f:I_{L_0,n}(X)\rightarrow I_n(L_0)/O(L_0),
\end{equation} 
which sends the pair $(X,e)$ to the isometry class of the pair
$(\widetilde{L},e)$. The faithfulness statement in Proposition
\ref{prop-isometry-class-of-tilde-L-e-is-a-faithful-mon-invariant}
follows from the following general statement.

\begin{lem}
\label{lem-faithfulness-of-the-isometry-class-function-of-tilde-L-e}
The function $f$, given in (\ref{eq-f-from-I-L-0-n}),
is a faithful monodromy invariant.
\end{lem}

\begin{proof}
Let $e_1$, $e_2$ be two classes in $I_{L_0,n}(X)$.
Denote by $\widetilde{L}_j$ the primitive rank $2$
sub-lattice of $\widetilde{\Lambda}$ associated to $e_j$
in equation (\ref{eq-saturation-of-L}), via a primitive embedding
$\iota:H^2(X,\Integers)\rightarrow \widetilde{\Lambda}$
in the canonical $O(\widetilde{\Lambda})$-orbit, $j=1,2$.
Denote $\iota(e_j)$ by $e_j$ as well.

Assume that $f(e_1)=f(e_2)$. Then there exists 
an isometry $g:\widetilde{L}_1\rightarrow \widetilde{L}_2$,
such that $g(e_1)=e_2$. 
Let $v\in\widetilde{\Lambda}$ be a generator of 
$\iota[H^2(X,\Integers)]^\perp$. 
Then $v$ is orthogonal to $e_j$.
Hence, $g(v)=v$ or $g(v)=-v$.
If $g(v)=-v$, set $g':=-(R_{e_2}\circ g)$. 
Then $g':\widetilde{L}_1\rightarrow \widetilde{L}_2$ is an isometry
satisfying $g'(e_1)=e_2$ and $g'(v)=v$. 
Hence, we may assume that $g(v)=v$.

There exists an isometry $\gamma\in O_+(\widetilde{\Lambda})$,
such that $\gamma(\widetilde{L}_1)=\widetilde{L}_2$ and 
$\gamma$ restricts to $\widetilde{L}_1$ as $g$, 
by \cite{nikulin}, Theorem 1.14.4. 
Then $\gamma(v)=v$ and so 
$\gamma\circ\iota=\iota\circ \mu$, for some isometry
$\mu\in O_+H^2(X,\Integers)$. The fact that the isometry $\mu$
extends to $\widetilde{\Lambda}$ implies that $\mu$ belongs to $Mon^2(X)$,
by \cite{markman-monodromy-I}, Theorem 1.6 (see also Lemma 4.10
part (3) in \cite{markman-monodromy-I}).
Now $\iota(\mu(e_1))=\gamma(\iota(e_1))=\iota(e_2)$.
So $\mu(e_1)=e_2$.
\end{proof}

%
\subsection{Isometry orbits in three rank two lattices}
\label{sec-isometry-orbits}
Set 
\[
M_U:=\left(
\begin{array}{cc}
0 & -1 \\
-1 & 0
\end{array}
\right), \ \ \ 
M_{H_{ev}}:=\left(
\begin{array}{cc}
2 & 0 \\
0 & -2
\end{array}
\right), \ \ \ 
M_{U(2)}:=\left(
\begin{array}{cc}
0 & -2 \\
-2 & 0
\end{array}
\right). 
\]
Given an integer $m$, let $\F(m)$ be  the set of unordered pairs $\{r,s\}$ of
positive integers, such that $rs=m$ and $\gcd(r,s)=1$.
Set 
\begin{eqnarray*}
\Sigma_n(U) & := & \F(n-1),
\\ 
\Sigma_n(U(2)) & := & 
\F([n-1]/2), 
\ \mbox{if}  \ n \ \mbox{is odd,}
\\
\Sigma_n(H_{ev}) & := & 
\left\{\begin{array}{ccl}
\F(n-1) & \mbox{if} & n\not\equiv 1 \ \mbox{(modulo)} \ 4,
\\
\F([n-1]/4) & \mbox{if} & n\equiv 1 \ \mbox{(modulo)} \ 4.
\end{array}
\right.
\end{eqnarray*} 

\begin{lem}
\label{lemma-isometry-orbits-in-rank-2}
Let $\widetilde{L}$ be $U$, $H_{ev}$, or $U(2)$, and $e\in I_n(\widetilde{L})$,
$n\geq 2$. 
Choose a generator $v$ of the sub-lattice of 
$\widetilde{L}$ orthogonal to $e$.
\begin{enumerate}
\item
Let $\rho$ be the largest positive integer,
such that $(e+v)/\rho$ is an integral class of $\widetilde{L}$. 
Define the integer $\sigma$ similarly using  $(e-v)$. 
Then $\gcd(\rho,\sigma)$ is $1$ or $2$. 
\item
\label{lemma-item-existence-and-uniqueness-of-r-s}
The integers $r:=\rho/\gcd(\rho,\sigma)$ and  $s:=\sigma/\gcd(\rho,\sigma)$ 
have the following properties.
\begin{enumerate}
\item
If $\widetilde{L}=U$, then $rs=n-1$, and the classes 
$\alpha:=\frac{e+v}{2r}$ and $\beta:=\frac{e-v}{2s}$
form a basis of $\widetilde{L}$ with Gram-matrix
$M_U$.
\item
If $\widetilde{L}=U(2)$, then $n$ is odd, $rs=(n-1)/2$, and the classes 
$\alpha:=\frac{e+v}{2r}$ and $\beta:=\frac{e-v}{2s}$
form a basis of $\widetilde{L}$ with Gram-matrix
$M_{U(2)}$.
\item
If $\widetilde{L}=H_{ev}$ and $n$ is even, then
$rs=n-1$ and the classes 
$
\alpha:=\frac{1}{2}\left[\frac{e+v}{r}-\frac{e-v}{s}\right]
\ \mbox{and} \ 
\beta:=\frac{1}{2}\left[\frac{e+v}{r}+\frac{e-v}{s}\right]
$
form a basis of $\widetilde{L}$ with Gram-matrix
$M_{H_{ev}}$.
\item
If $\widetilde{L}=H_{ev}$ and $n$ is odd, then $n\equiv 1$ modulo $4$,   
$rs\nolinebreak=\nolinebreak(n-\nolinebreak 1)/4$, and the classes 
$
\alpha:=\frac{1}{2}\left[\frac{e+v}{2r}-\frac{e-v}{2s}\right]
\ \mbox{and} \ 
\beta:=\frac{1}{2}\left[\frac{e+v}{2r}+\frac{e-v}{2s}\right]
$
form a basis of $\widetilde{L}$ with Gram-matrix
$M_{H_{ev}}$.
\end{enumerate}
\item
\label{lemma-item-the-affect-of-changing-the-sign-of-v}
If we replace $v$ by $-v$, then $(r,s)$ gets replaced by $(s,r)$. 
\item
\label{lemma-item-rs-induces-a-bijection}
Let $rs:I_n(\widetilde{L})\rightarrow \Sigma_n(\widetilde{L})$
be the function, which assigns to a class $e\in I_n(\widetilde{L})$ 
the unordered pair $\{r,s\}$ occurring in the above factorization. 
Then $rs$ factors through a one-to-one 
correspondence
\[
\overline{rs} \ : \ 
I_n(\widetilde{L})/O(\widetilde{L}) \ \ \ \longrightarrow \ \ \ 
\Sigma_n(\widetilde{L}).
\]
\end{enumerate}
\end{lem}

\begin{proof}
Let $\{u_1, u_2\}$ be a basis of $\widetilde{L}$ with Gram-matrix 
$M_{\widetilde{L}}$.
Observe first that $O(\widetilde{L})$ is isomorphic to 
$\Integers/2\Integers\times \Integers/2\Integers$.
Indeed, each of $O(U)$ and $O(U(2))$ is generated by
$-id$ and the isometry, which interchanges $u_1$ and $u_2$.
$O(H_{ev})$ is generated by the two commuting reflections with respect to 
$u_1$ and $u_2$. Write 
\[
e=au_1+bu_2.
\]
\underline{Case $\widetilde{L}=U$.}
We have $n-1=-(e,e)/2=ab$ and $\gcd(a,b)=1$, since $e$ is primitive.
Note also that $a$ and $b$ have the same sign. 
Set $v:=au_1-bu_2$. Then 
$\frac{e+v}{2a}=u_1$ and $\frac{e-v}{2b}=u_2$.
Thus $r=\Abs{a}$ and $s=\Abs{b}$, and part
\ref{lemma-item-existence-and-uniqueness-of-r-s} holds.
Part \ref{lemma-item-the-affect-of-changing-the-sign-of-v} is clear.
Part \ref{lemma-item-rs-induces-a-bijection} follows from
part \ref{lemma-item-the-affect-of-changing-the-sign-of-v} 
and the identification of $O(U)$ above.

\medskip
\noindent
\underline{Case $\widetilde{L}=U(2)$.}
We may identify the free abelian groups underlying $U$ and $U(2)$,
so that the bilinear form on $U(2)$ is $2$ times that of $U$.
The statement of the Lemma follows immediately from the case 
$\widetilde{L}=U$.

\medskip
\noindent
\underline{Case $\widetilde{L}=H_{ev}$.}
We have $2-2n=(e,e)=2(a-b)(a+b)$.  
So $b-a$ and $b+a$ have the same sign, since $n\geq 2$.  
If $n$ is odd, then both $a$ and $b$ are odd, since 
$\gcd(a,b)=1$ and $(a-b)(a+b)$ is even.
If $n$ is even, then $\{a,b\}$ consists of one odd and one even integer.
Furthermore, 
\[
\gcd(b-a,b+a)=\gcd(b-a,2a)=\left\{
\begin{array}{ccl}
1, & \mbox{if} & n \ \mbox{is even,}
\\
2, & \mbox{if} & n \ \mbox{is odd.}
\end{array}
\right.
\]
Choose $v=bu_1+au_2$. We have
\[
u_1=\frac{1}{2}\left[\frac{e+v}{a+b}-\frac{e-v}{b-a}\right], \ \ \ 
u_2=\frac{1}{2}\left[\frac{e+v}{a+b}+\frac{e-v}{b-a}\right].
\]
Hence, $r=\Abs{a+b}$ and $s=\Abs{b-a}$, if $n$ is even, and
$r=\Abs{a+b}/2$ and $s=\Abs{b-a}/2$, if $n$ is odd.
The rest is similar to the case $\widetilde{L}=U$.
\end{proof}

The following table summarizes how the statements of Proposition 
\ref{prop-isometry-class-of-tilde-L-e-is-a-faithful-mon-invariant} 
and Lemma \ref{lemma-isometry-orbits-in-rank-2} determine the lattice 
$\widetilde{L}$ and the pair $\{r,s\}$ in terms of $(e,e)$, $\div(e,\bullet)$,
$n$, and $\{\rho,\sigma\}$.

\smallskip
\noindent
\begin{tabular}{|c|c|c|c|c|c|c|c|}  \hline 
\hspace{1ex}
& $(e,e)$ & $\div(e,\bullet)$ & $n$ & $\rho\sigma$ & $\widetilde{L}$ & 
$\{r,s\}$ & $r\cdot s$ 
\\
\hline
1) &$2-2n$ & $2n-2$ & $\geq 2$ & $4n-4$ & $U$ & 
$\{\frac{\rho}{2},\frac{\sigma}{2}\}$ & $n-1$
\\
\hline
2) & $2-2n$ & $n-1$ & even & $n-1$ & $H_{ev}$ & $\{\rho,\sigma\}$ & $n-1$
\\
\hline
3) & $2-2n$ & $n-1$ & odd & $2n-2$ & $U(2)$ & 
$\{\frac{\rho}{2},\frac{\sigma}{2}\}$ & $(n-1)/2$
\\
\hline
4) & $2-2n$ & $n-1$ & $\equiv 1$ modulo $8$ & $n-1$ & $H_{ev}$ & 
$\{\frac{\rho}{2},\frac{\sigma}{2}\}$ & $(n-1)/4$
\\
\hline
\end{tabular}\\
In line $3$ cases where $n\equiv 1$ modulo $8$ occur as well.
%
\section{Monodromy-invariants of monodromy-reflective classes}
\label{sec-invariant-rs}
\hspace{0ex}\\
Fix $n\geq 2$. Let $X$ be a (K\"{a}hler) irreducible holomorphic symplectic 
manifold of $K3^{[n]}$-type. 
We define in this section the monodromy invariant function $rs$ of
Proposition 
\ref{prop-introduction-Mon-2-orbit-is-determined-by-three-invariants}
and prove that proposition. Part \ref{prop-item-Mon-orbit-in-degree-minus-2} 
of the Proposition was treated in \cite{markman-monodromy-I}, Lemma 8.9. 
We thus consider only part \ref{prop-item-Mon-orbit-in-degree-2-minus-2n}.
We will relate this latter part to Proposition
\ref{prop-isometry-class-of-tilde-L-e-is-a-faithful-mon-invariant}
and prove Proposition
\ref{prop-isometry-class-of-tilde-L-e-is-a-faithful-mon-invariant}.

It will be convenient to use the following normalization.
Fix an isometry $\widetilde{\Lambda}\cong K(S)$, for some
$K3$ surface $S$, 
and use Mukai's notation for classes in the Mukai lattice $K(S)$.
The isometry group $O(\widetilde{\Lambda})$ acts transitively on the
set of primitive classes in $\widetilde{\Lambda}$
of degree $2n-2$. Hence, we may choose the embedding 
$\iota:H^2(X,\Integers)\rightarrow \widetilde{\Lambda}$, 
so that $v=(1,0,1-\nolinebreak n)$ is orthogonal to the image of $\iota$. 
Then $v^\perp=H^2(S,\Integers)\oplus\Integers\delta$,
where $\delta:=(1,0,n-1)$. 
Thus 
\begin{equation}
\label{eq-normalization-of-e-as-x-plus-t-delta}
e \ \ = \ \ x+t\delta, 
\end{equation}
for some integer $t$ and a class $x\in H^2(S,\Integers)$.

%
\subsection{The divisibility case $\div(e,\bullet)=(e,e)$}\hspace{0ex}\\
Let $I_n(X)\subset H^2(X,\Integers)$ be the subset of all primitive
classes $e$, satisfying $(e,e)=2-2n$ and $\div(e,\bullet)=2n-2$.
Recall that $\Sigma_n(U)$ is the set of unordered pairs $\{r,s\}$ of
positive integers, such that $rs=n-1$ and $\gcd(r,s)=1$.

\begin{lem}
\label{lem-faithful-Mon-invariant-in-case-divisibility-2n-2}
If $e$ belongs to $I_n(X)$, 
then $\widetilde{L}$ is isometric to the unimodular 
hyperbolic plane $U$. 
Denote by 
\[
rs \ : \ I_n(X) \ \ \ \longrightarrow \ \ \ \Sigma_n(U)
\]
the composition of the function
$f:I_n(X)\rightarrow I_n(U)/O(U)$,
defined in equation (\ref{eq-f-from-I-L-0-n}), with 
the bijection $\overline{rs}:I_n(U)/O(U)\rightarrow \Sigma_n(U)$
constructed in Lemma \ref{lemma-isometry-orbits-in-rank-2}.
Then the function $rs:I_n(X)\rightarrow \Sigma_n(U)$ 
is surjective and a faithful monodromy-invariant
(Definition \ref{def-faithful}). 
\end{lem}

\begin{proof}
Write $e=x+t\delta$ as in equation
(\ref{eq-normalization-of-e-as-x-plus-t-delta}).
The assumption that $\div(e,\bullet)=2n-2$ implies that $x=(2n-2)\xi$,
for a class $\xi\in H^2(S,\Integers)$. We clearly have the equality
\[
2-2n=(e,e)=(x,x)+t^2(\delta,\delta)=(2-2n)^2(\xi,\xi)+(2-2n)t^2.
\]
Hence, we get the equality 
\[
t^2-1 \ \ \ = \ \ \ (2n-2)(\xi,\xi).
\]
Consequently, $4n-4$ divides $(t-1)(t+1)$. 
Thus $n-1$ divides $\frac{t-1}{2}\frac{t+1}{2}$.
Now $\gcd\left(\frac{t-1}{2},\frac{t+1}{2}\right)=1$.
We get a unique factorization $n-1=rs$, where $s$
divides $(t-1)/2$, $r$ divides $(t+1)/2$, and $\gcd(r,s)=1$.
We may assume that $s$ is odd, possibly after 
replacing the embedding $\iota$ by $-\iota$, which replaces $t$ by $-t$. 


Using the above factorization $n-1=rs$, we get
\begin{eqnarray*}
e+v=2r\alpha, & \mbox{where}, &
\alpha:=\left(\frac{t+1}{2r},s\xi,\frac{(t-1)s}{2}\right)
\\
e-v=2s\beta, & \mbox{where}, &
\beta:=\left(\frac{t-1}{2s},r\xi,\frac{(t+1)r}{2}\right)
\end{eqnarray*}
and the classes $\alpha$ and $\beta$ belongs to $\widetilde{L}$.
The Gram-matrix of $\{\alpha,\beta\}$ is\\
$
\left(\begin{array}{cc}
(\alpha,\alpha) & (\alpha,\beta)
\\
(\alpha,\beta) & (\beta,\beta)
\end{array}
\right)=
\left(\begin{array}{cc}
(e+v,e+v)/4r^2 & (e+v,e-v)/4rs
\\
(e+v,e-v)//4rs & (e-v,e-v)/4s^2
\end{array}
\right)=$
\\
$
\left(\begin{array}{cc}0 & -1\\-1&0
\end{array}\right).
$
\\
We conclude that ${\rm span}\{\alpha,\beta\}$
is a unimodular sub-lattice of $\widetilde{\Lambda}$. Hence,
$\widetilde{L}={\rm span}\{\alpha,\beta\}$ and $\widetilde{L}\cong U$. 

The function $rs$ is shown to be surjective in 
Example \ref{example-any-factorization-rs-is-possible}.
The faithfulness of the monodromy-invariant
$rs$ was proven in Lemma
\ref{lem-faithfulness-of-the-isometry-class-function-of-tilde-L-e}.
\end{proof}

\begin{example}
\label{example-any-factorization-rs-is-possible}
(Compare with section \ref{sec-sequence-of-examples} above).
Choose a factorization $n-1=rs$, with $s\geq r>0$,
and $\gcd(r,s)=1$. 
Let $S$ be a projective $K3$ surface, $v=(r,0,-s)\in K(S)$,
$H$ a $v$-generic polarization, and $X=M_H(r,0,-s)$. 
Let $\iota:H^2(M_H(r,0,-s),\Integers)\hookrightarrow K(S)$ be the 
embedding given in (\ref{eq-iota-for-a-moduli-space}).
Set $e:=\theta(r,0,s)$, where $\theta$ is Mukai's isometry given in equation
(\ref{eq-Mukai-isomorphism}). 
The class $e$ is monodromy-reflective
and $\div(e,\bullet)=2n-2$.
Now $(v+e)/2r=(1,0,0),$ $(e-v)/2s=(0,0,1),$ and $\widetilde{L}\cong U$. 
We get that $rs(e)=\{r,s\}$, by
Lemma \ref{lemma-isometry-orbits-in-rank-2}.
\hide{
The class $e$ arises from a prime exceptional divisor precisely in
the following two cases:
\begin{enumerate}
\item
When $v=(1,0,1-n)$, then $2e=[E]$, for the prime exceptional divisor
$E$ in Example \ref{example-diagonal-of-hilbert-scheme}.
\item
When $v=(2,0,[1-n]/2)$ is primitive (i.e., $n\equiv 3$ modulo $4$), 
then $e=[E]$, for a prime exceptional divisor $E$,
by Lemma \ref{lemma-class-of-exceptional-locus} part 
\ref{lemma-item-L-divisible-by-2}.
\end{enumerate} 
If $s>r>2$, and $k$ is a non-zero integer, then the class
$ke$ is not effective, by Proposition 
\ref{prop-vanishing-in-divisibility-2n-2}.
}
\end{example}

%
\subsection{The divisibility case $\div(e,\bullet)=(e,e)/2$}
\hspace{0ex}\\
Let $n$ be an integer $\geq 2$.
Let $I'_n(X)\subset H^2(X,\Integers)$ be the subset of all 
primitive classes $e$ satisfying $(e,e)=2-2n$, and $\div(e,\bullet)=n-1$.
Set \\
$\Sigma'_n:=
\left\{\begin{array}{ccl}
\Sigma_n(H_{ev})& \mbox{if} & n \ \mbox{is even,}
\\
\Sigma_n(U(2)) & \mbox{if} &  n  \ \mbox{is odd, but} \
n\not\equiv 1 \ \mbox{modulo} \ 8.
\\
\Sigma_n(U(2))\cup \Sigma_n(H_{ev}) & \mbox{if} & n\equiv 1 \  
\mbox{modulo} \ 8.
\end{array}\right.$ \\
In each of the above three cases, let 
$\IC_n$ be the union of the sets $I_n(\widetilde{L})/O(\widetilde{L})$ 
as $\widetilde{L}$ ranges through the one or two lattices appearing.

\begin{lem}
\label{lem-non-unimodular-rank-two-lattice}
Let $e$ be a class in $I'_n(X)$.
\begin{enumerate}
\item
\label{lemma-item-case-n-is-even}
If $n$ is even, then $\widetilde{L}$ is isometric to $H_{ev}$.
\item
\label{lemma-item-case-n-is-odd}
If $n$ is odd, then $\widetilde{L}$ is isometric to $U(2)$ or to
$H_{ev}$. The latter occurs only if $n\equiv 1$ modulo $8$ and $\rho\sigma=n-1$.
\end{enumerate}
In both cases, let 
\[
rs \ : \ I'_n(X) \ \ \ \longrightarrow \ \ \ \Sigma'_n
\]
be the composition 
of the function
$f:I'_n(X)\rightarrow \IC_n$,
defined in equation (\ref{eq-f-from-I-L-0-n}), with 
the injection $\overline{rs}:\IC_n\rightarrow \Sigma'_n$,
constructed in Lemma \ref{lemma-isometry-orbits-in-rank-2}.
Then the function $rs:I'_n(X)\rightarrow \Sigma'_n$ 
is surjective and a faithful monodromy-invariant
(Definition \ref{def-faithful}). 
\end{lem}


\begin{proof} 
Let us first observe that $\widetilde{L}$ cannot be unimodular.
Assume otherwise. Then
$\widetilde{\Lambda}$ decomposes as an orthogonal direct sum
$\widetilde{L}\oplus\widetilde{L}^\perp$. Consequently, 
$v^\perp$ decomposes as the orthogonal direct sum
$\widetilde{L}^\perp\oplus\Integers\{e\}$.
But then $\div(e,\bullet)=2n-2$. 

We keep the normalization $e=x+t\delta$ of equation
(\ref{eq-normalization-of-e-as-x-plus-t-delta}). 
The assumption that $\div(e,\bullet)=n-1$ implies that $x=(n-1)\xi$,
for a class $\xi\in H^2(X,\Integers)$. We have the equality
\begin{equation}
\label{eq-t-square-1-equal-n-1-xi-xi-over-2}
(t^2-1)=\frac{(n-1)(\xi,\xi)}{2}.
\end{equation}
Hence, $n-1$ divides $t^2-1$.

\underline{Case $n$ is even:}
Then $n-1$ is odd. Set
\[
r:=\gcd(t+1,n-1), \ \ \ s:=\gcd(t-1,n-1).
\]
Then both $r$ and $s$ are odd and $\gcd(r,s)$ divides
$\gcd(t-1,t+1)$. We conclude that $\gcd(r,s)=1$ and $rs$ divides $n-1$.
Now $n-1$ divides $(t-1)(t+1)$. Thus, $n-1$ divides $rs$ and so 
$rs=n-1$.  
%
%
Set
\begin{eqnarray*}
\alpha & := & \frac{1}{2}\left[\frac{e+v}{r}-\frac{e-v}{s}\right]
=
\frac{1}{2}\left(\frac{t+1}{r}-\frac{t-1}{s},(s-r)\xi,
(s-r)t-s-r\right)
\\
\beta & := & \frac{1}{2}\left[\frac{e+v}{r}+\frac{e-v}{s}\right]
=
\frac{1}{2}\left(\frac{t+1}{r}+\frac{t-1}{s},(s+r)\xi,
(s+r)t+r-s\right).
\end{eqnarray*}
Note the equality $\frac{t+1}{r}-\frac{t-1}{s}=\frac{(s-r)t+s+r}{rs}$
and the fact that the denominator is odd, while the numerator is even.
Hence, $\alpha$, $\beta$ are integral classes of $\widetilde{\Lambda}$ and 
$\Gram{\alpha}{\beta}=
\left(\begin{array}{cc}2&0\\0&-2\end{array}\right)$.

\begin{claim}
\label{claim-non-unimodular-sub-lattice-is-saturated}
$\widetilde{L}={\rm span}\{\alpha,\beta\}$.
\end{claim}

\begin{proof}
Suppose otherwise. Then $\widetilde{L}$ strictly contains 
$L':={\rm span}\{\alpha,\beta\}$. 
Let $d$ be the index of $L'$ in $\widetilde{L}$. Then 
the determinant of the Gram-matrix of $\widetilde{L}$ is
$d^2$ times the determinant of the Gram-matrix of $L'$. 
The latter determinant is $-4$.
It follows that $\widetilde{L}$ is unimodular, a contradiction.
\end{proof}

\underline{Case $n$ is odd:}
Then $t$ is odd, by equation (\ref{eq-t-square-1-equal-n-1-xi-xi-over-2}).
Set
\[
r :=  \gcd\left(\frac{n-1}{2},\frac{t+1}{2}\right),
\ \ \ 
s := \gcd\left(\frac{n-1}{2},\frac{t-1}{2}\right).
\]
Then $rs$ divides $(n-1)/2$, 
since $\gcd\left(\frac{t+1}{2},\frac{t-1}{2}\right)=1$. 
%
%

\underline{Case $n$ is odd and $(\xi,\xi)/2$ is even:}\\
Then $(n-1)/2$ divides $(t+1)(t-1)/4$, 
by equation (\ref{eq-t-square-1-equal-n-1-xi-xi-over-2}).
Hence, $rs=(n-1)/2$. Set
\begin{eqnarray*}
\alpha & := & \frac{e+v}{2r} = 
\left(\frac{t+1}{2r},s\xi,s(t-1)\right),
\\
\beta & := & \frac{e-v}{2s} = 
\left(\frac{t-1}{2s},r\xi,r(t+1)\right).
\end{eqnarray*}
Then $\alpha$ and $\beta$ are 
integral classes of $\widetilde{\Lambda}$
and $\Gram{\alpha}{\beta}=\left(\begin{array}{cc}0&-2\\-2&0
\end{array}\right)$. 
We conclude the equality $\widetilde{L}={\rm span}\{\alpha,\beta\}$,
by the argument used in Claim 
\ref{claim-non-unimodular-sub-lattice-is-saturated}.

\underline{Case $n$ is odd and $(\xi,\xi)/2$ is odd:}
Let $2^k$ be the largest power of $2$ which divides $t^2-1$.
Then $k\geq 3$. Furthermore, $2^k$ is also the largest power of $2$ which
divides $n-1$, by equation 
(\ref{eq-t-square-1-equal-n-1-xi-xi-over-2}). 
Thus $n\equiv 1$ (modulo $8$). 
The set $\{r,s\}$ consists of one odd and one even integer. 
Say $s$ is odd. Then 
$2^{k-2}$ is the largest power of $2$, which divides $r$.
We conclude that $rs=(n-1)/4$. 
Furthermore, both $(t+1)/2r$ and $(t-1)/2s$ are odd. 
Set
\begin{eqnarray*}
\alpha & := & \frac{1}{2}\left[\frac{e+v}{2r}+\frac{e-v}{2s}\right]=
\frac{1}{2}\left(\frac{t+1}{2r}+\frac{t-1}{2s},(2s+2r)\xi,2s(t-1)+2r(t+1)
\right)
\\
\beta & := & \frac{1}{2}\left[\frac{e+v}{2r}-\frac{e-v}{2s}\right]=
\frac{1}{2}\left(\frac{t+1}{2r}-\frac{t-1}{2s},(2s-2r)\xi,2s(t-1)-2r(t+1)
\right).
\end{eqnarray*}
Then $\alpha$ and $\beta$ are integral classes of $\widetilde{\Lambda}$
and $\Gram{\alpha}{\beta}=\left(\begin{array}{cc}-2&0\\0&2
\end{array}\right)$. 
We conclude the equality $\widetilde{L}={\rm span}\{\alpha,\beta\}$,
by the argument used in Claim 
\ref{claim-non-unimodular-sub-lattice-is-saturated}.

The function $rs$ is shown to be surjective in Examples
\ref{example-rs-is-surjective-divisibility-n-1} and \ref{example-yet-another}.
The faithfulness of the monodromy invariant $rs$ was   
proven in Lemma 
\ref{lem-faithfulness-of-the-isometry-class-function-of-tilde-L-e}.
\end{proof}

\begin{example}
\label{example-rs-is-surjective-divisibility-n-1}
Let $s>r\geq 1$ be positive integers with $\gcd(r,s)=1$.
Set $n:=rs+1$. Let $S$ be a projective $K3$ surface, set $v:=(r,0,-s)$, 
and let $H$ be a $v$-generic polarization of $S$. 
Set $M:=M_H(v)$. Let $A$ be a primitive isotropic class in $H^2(S,\Integers)$.
Set $e:=\theta(r,(n-1)A,s)$.
Then $e$ is monodromy-reflective and $\div(e,\bullet)=n-1$.
If $n$ is even, then $(M,e)$ is an example
of case \ref{lemma-item-case-n-is-even} of Lemma 
\ref{lem-non-unimodular-rank-two-lattice}, with $rs(e)=\{r,s\}$.
If $n$ is odd, then $n-1=rs$ is even and precisely one of $r$ or $s$ is even. 
If $r$ is even, then $\rho=r$ and $\sigma=2s$.
If $s$ is even, then $\rho=2r$ and $\sigma=s$.
$(M,e)$ is an example
of case \ref{lemma-item-case-n-is-odd} of Lemma 
\ref{lem-non-unimodular-rank-two-lattice}, with $\widetilde{L}\cong U(2)$ and
$rs(e)=\{r/2,s\}$,
if $r$ is even, and $rs(e)=\{r,s/2\}$, if $s$ is even. 
\end{example}

\begin{example}
\label{example-yet-another}
We exhibit next examples of the case of Lemma \ref{lem-non-unimodular-rank-two-lattice}, where  $X=S^{[n]}$, 
$n\equiv 1$ modulo $8$, and
$\widetilde{L}\cong H_{ev}$. Set $n=8k+1$, $k$ an integer $\geq 1$.
Choose a factorization $2k=rs$, with $r$ even, $s$ odd, and $\gcd(r,s)=1$. 
There exists an integer $\lambda$, such that 
$\lambda r\equiv -1$ modulo $s$, since $\gcd(s,r)=1$. 
If $\lambda$ is a solution, so is $\lambda+s$. Hence, we may assume that
$\lambda$ is an odd and positive integral solution. 
Set $g:=[r\lambda+1]/s$. Then $g$ is a positive odd integer.

Let $S$ be a $K3$ surface with a primitive class $\xi\in \Pic(S)$ of degree
$(\xi,\xi)=2\lambda g$. 
Set 
\[
v:=(1,0,1-n) \ \ \ \mbox{and} \ \ \ 
e := (2\lambda r+1,(n-1)\xi,[2\lambda r+1](n-1)).
\]
Then $(e,e)=2-2n$, by following two equalities.
\[
(e,e)=(n-1)[2\lambda g(n-1)-2(4\{r^2\lambda^2+r\lambda\}+1)]
\]
and $2\lambda g(n-1)=8r\lambda sg=8r\lambda(r\lambda+1)$. 
The class $e$ is primitive, since 
$\gcd(2\lambda r+1,n-1)=\gcd(2\lambda r+1,4rs)=\gcd(2\lambda r+1,s)=
\gcd(-1,s)=1$.
The classes
$\frac{e+v}{2s}=(g,2r\xi,4\lambda r^2)$ and
$\frac{e-v}{2r}=(\lambda,2s\xi,4gs^2)$ are integral and primitive.
We conclude that $\widetilde{L}\cong H_{ev}$, by Proposition
\ref{prop-isometry-class-of-tilde-L-e-is-a-faithful-mon-invariant} part 
\ref{prop-item-three-possible-rank-2-lattices}, and 
$rs(e)=\{r,s\}$, by Lemma 
\ref{lemma-isometry-orbits-in-rank-2}. 
\end{example}


%
\section{Numerical characterization of exceptional classes via Torelli}
\label{sec-numerical-characterization-via-torelli}
The following table points to an example provided in this paper,
for each possible value of the quadruple
$\{n, \ (e,e), \ \div(e,\bullet), \ rs(e)\}$, for a monodromy
reflective class $e$.

\smallskip
\noindent
\begin{tabular}{|c|c|c|c|c|c|l|}  \hline 
\hspace{1ex}
$(e,e)$ & $\div(e,\bullet)$ & $\widetilde{L}$ & $\{r,s\}$ & $n$ & Reference 
\\
\hline
$-2$ & $1$ & NA & NA & $\geq 2$ & 
Examples \ref{example-case-degree-e-minus-2-div-1}, 
\ref{example-case-degree-e-minus-2-div-1-brill-noether} 
\\
\hline
$-2$ & $2$ & NA & NA & $\geq 6$ and  & 
Example \ref{example-degree-e-minus-2-div-2}
\\
     &     &  & & $\equiv 2$ mod $4$ & 
\\
\hline
$-2$ & $2$ & NA & NA & $=2$       & 
Example \ref{example-diagonal-of-hilbert-scheme}
\\
\hline
$2-2n$ & $2n-2$ & $U$ & $\{1,n-1\}$ & $\geq 2$ & 
Example \ref{example-diagonal-of-hilbert-scheme}
\\
\hline
$2-2n$ & $2n-2$ & $U$ & $\{2,(n-1)/2\}$ & $\geq 7$ and & 
Lemma \ref{lemma-class-of-exceptional-locus} 
part \ref{lemma-item-L-divisible-by-2}
\\
     &     &  & & $\equiv 3$ mod $4$ & 
\\
\hline
$2-2n$ & $2n-2$ & $U$ & $s>r>2$ & $=rs+1$ &  Proposition
\ref{prop-vanishing-in-divisibility-2n-2}
\\
       &       & & $\gcd(r,s)=1$ &  &   
\\
\hline
$2-2n$ & $n-1$ & $H_{ev}$ & $\{1,n-1\}$ & $\geq 4$, even &  
Lemma \ref{lemma-class-of-exceptional-locus} 
part \ref{lemma-item-L-not-divisible-by-2}
\\
\hline
$2-2n$ & $n-1$ & $U(2)$ & $\{1,(n-1)/2\}$ & $\geq 3$, odd &  
Lemma \ref{lemma-class-of-exceptional-locus} 
part \ref{lemma-item-L-not-divisible-by-2}
\\
\hline
$2-2n$ & $n-1$ & $H_{ev}$ & $r\geq 3$, $s\geq 3$ & $=rs+1$ &  
Lemma \ref{lemma-Type-B-divisibility-n-1}
\\
       &       & & $\gcd(r,s)=1$        &  even &   
\\
\hline
$2-2n$ & $n-1$ & $U(2)$ & $r\geq 3$, $s\geq 2$ &  $=2rs+1$ &  
Lemma \ref{lemma-Type-B-divisibility-n-1}
\\
       &       & & $\gcd(r,s)=1$        &  &   
\\ 
\hline
$2-2n$ & $n-1$ & $H_{ev}$ & $r$ even, $s$ odd & $=4rs+1$ &  
Example \ref{example-non-effective-divisibility-n-1-and-n-is-cong-1-mod-8}
\\
       &       & & $\gcd(r,s)=1$        &  &   
\\ 
\hline
\end{tabular}
\hspace{1ex}

\smallskip
\noindent
The congruence constraints on $n$
are necessary. If $(e,e)=-2$ and $\div(e,\bullet)=2$, 
then $n\equiv 2$ (modulo $4$), by 
\cite{markman-monodromy-I}, Lemma 8.9.
If $rs(e)=\{2,(n-1)/2\}$, then $n\equiv 3$  (modulo $4$), 
in order for $\{2,(n-1)/2\}$ to be a pair of relatively prime integers.
If $(e,e)=2-2n$, $\div(e,\bullet)=n-1$, $n>2$, and $rs(e)=\{1,n-1\}$, then 
$n$ must be even, since for odd $n$ the product of $r$ and $s$ 
is equal to $(n-1)/2$ or $(n-1)/4$, 
by Lemmas \ref{lem-non-unimodular-rank-two-lattice} 
and \ref{lemma-isometry-orbits-in-rank-2}. 
This explains also the value of $n$ in the last three rows.

\begin{proof} 
\WithoutTorelli{(Of Theorem 
\ref{thm-main-conjecture-follows-from-torelli}).}
\WithTorelli{(Of Theorem \ref{conj-exceptional-line-bundles}).}
Set $e:=c_1(L)$. 
The pair $(X,e)$ is deformation equivalent, in the sense of Definition
\ref{def-deformation-equivalent-pairs-with-cohomology-class}, 
to a pair $(M,c)$ appearing in the above table of examples, by 
Proposition 
\ref{prop-introduction-Mon-2-orbit-is-determined-by-three-invariants} and 
Lemma \ref{lem-monodromy-invariants-and-deformation-equivalence} part 
\ref{lemma-item-monodromy-invariants-and-deformation-equivalence-of-classes}.
$M$ is projective and \WithoutTorelli{Conjecture}
\WithTorelli{Theorem} \ref{conj-exceptional-line-bundles} holds
for $(M,c)$, by the example referred to in the table.
Suppose that $L$ is numerically effective. Then
Part \ref{conj-item-effective} of 
\WithoutTorelli{Conjecture}
\WithTorelli{Theorem}
\ref{conj-exceptional-line-bundles}
follows for $(X,L)$, by Proposition 
\ref{prop-main-question-on-deformation-equivalence}.

Suppose next that $L$ is not numerically effective. We prove part
\ref{conj-item-vanishing} of 
\WithoutTorelli{Conjecture}
\WithTorelli{Theorem}
\ref{conj-exceptional-line-bundles} by contradiction. 
Assume that part \ref{conj-item-vanishing} fails.
Then there exists a non-zero integer $k$, such that 
$h^0(X_t,L_t^k)>0$, for all $t\in D_e$.
We may assume that the absolute value $\Abs{k}$ is minimal with the above 
property. Now $\Pic(X_t)$ is cyclic, for a generic $t\in D_e$.
Hence, the linear system $\linsys{L_t^k}$ must have a member $E_t$,
which is a prime divisor, by the minimality of $\Abs{k}$.
It follows that $E_t$ is the unique member of the linear system, by
\cite{boucksom}, Proposition 3.13. 
Hence, $h^0(X_t,L_t^k)=1$, away from a closed analytic 
proper subset $Z\subset D_e$.

Set $U:=D_e\setminus Z$ and let $\X_U$ be the restriction of the semi-universal
family from $Def(X)$ to $U$. 
There exists an irreducible divisor $\E\subset \X_U$, which does 
not contain the fiber $X_t$, for any $t\in U$, and which intersects
$X_t$ along a divisor in $\linsys{L_t^k}$, by 
the argument used in the proof of Proposition 
\ref{prop-generic-prime-exceptional}. 
The argument furthermore shows, that there exists a closed analytic proper
subset $Z_1\subset U$, such that the fiber $E_t$ of $\E$
is a prime divisor, over all points $t\in U\setminus Z_1$. 
We do not need the projectivity assumption, as it was used in the proof of 
Proposition \ref{prop-generic-prime-exceptional}
only to establish that the generic dimension of $h^0(X_t,L_t)$ is $1$,
a fact which was already established above. 

We conclude the existence of a pair $(X_1,e_1)$, parametrized by a point in
$U\setminus Z_1$, such that $X_1$ is projective, 
by \cite{huybrechts-norway}, Proposition 21. 
Let $L_2$ be the line bundle on $M$ with $c_1(L_2)=c$.
Then $H^0(M,L_2^d)$ vanishes, for all non-zero integers $d$, since
$L_2$ is not numerically exceptional, and the examples mentioned in the
above table have this property, whenever $c$ is not numerically exceptional.
Hence, $(X_1,e_1)$ and $(M,c)$ 
are not deformation equivalent, by Corollary 
\ref{cor-if-L-1-is-prime-exceptional-L-2-not-Q-effective-then-not-def-equiv}.
On the other hand, $(X_1,e_1)$ is deformation equivalent to $(X,e)$,
and hence to $(M,c)$, a contradiction.
\end{proof}

%

%
\section{Conditions for the existence of slope-stable vector bundles}
\label{sec-conditions-for-existence-ofslope-stable-vector-bundles}
Let $S$ be a projective $K3$ surface with a cyclic Picard group generated by
an ample line bundle $H$.  We assemble in section 
\ref{sec-necessary-conditions}
necessary conditions for the existence 
of locally-free $H$-slope-stable sheaves (Lemmas 
\ref{lem-H-slope-stability-implies-s-geq-r},
\ref{lemma-non-existence-of-H-slope-stable-sheaves-with-slope-half}, and 
\ref{lemma-condition-for-existence-of-H-slope-stable-sheaves-isotropic-u-case}). 

In section \ref{sec-sufficient-conditions} we bound the dimension of the locus
$Exc$ of $H$-stable sheaves, which are not locally free or not
$H$-slope-stable. 
The sheaves $F$ considered all have the following involutive property:
there exists an integer $t$, such that the classes in $K(S)$ 
of $F$ and $F^*\otimes H^t$ are equal. Equivalently, 
$c_1(F)=\frac{t\cdot\rank(F)}{2}h$, for some integer $t$, where $h:=c_1(H)$.
I thank Kota Yoshioka for pointing out that much of the content of  section 
\ref{sec-sufficient-conditions} is essentially proven in sections
2 and 3 of his paper \cite{yoshioka-irreducibility}. Section \ref{sec-sufficient-conditions}
was not replaced by a citation, since the precise statements 
we need are not easily recovered from those of Yoshioka, as he was mainly concerned with 
proving that the locus $Exc$ has codimension $\geq 1$, while we need that in the subset of cases considered\footnote{
For some cases Yoshioka does state that the codimension is $\geq 2$ \cite[Lemma 3.1]{yoshioka-irreducibility},
but under an assumption that excludes some cases which we need.}
$Exc$ has codimension $\geq 2$.

The results of this section are only lightly used  in section 
\ref{sec-examples},
but are essential to the examples in section \ref{sec-non-effective}.

%
\subsection{Necessary conditions}
\label{sec-necessary-conditions}
Set $h:=c_1(H)\in H^2(S,\Integers)$ and 
$d:=\deg(H)/2$.
\begin{lem}
\label{lemma-locally-free-H-stable-of-rank-2-is-slope-stable}
Let $F$ be a locally free $H$-stable sheaf of rank $r$
satisfying $c_1(F)=\frac{tr}{2}h$, for some integer $t$.
Then $F^*$ is $H$-stable, if and only if $F$ is $H$-slope-stable.
In particular, if $r=2$, then $F$ is $H$-slope-stable.
\end{lem}

\begin{proof}
After tensorization by a power of $H$, we may reduce to the case where
either $c_1(F)=0$, or $r=2\rho$ is even and $c_1(F)=\rho h$.
Assume that we are in one of these cases. 
If $c_1(F)=0$, set $L:=\StructureSheaf{S}$.
If $c_1(F)=\rho h$, set $L:=H$. 
In either case, we have the equality $[F]=[F^*\otimes L]$ of
classes in $K(S)$. Furthermore, a sheaf $G$ is $H$-stable if and only if
$G\otimes L$ is $H$-stable.

If $F$ is $H$-slope-stable, then so is $F^*$. 
Hence $F^*$ is $H$-stable as well. 
$F$ is $H$-slope-semi-stable, since it is $H$-stable.
Suppose that $F$ is not $H$-slope-stable.
then there exists 
a saturated subsheaf $F_1\subset F$, of rank (say) $r_1$, 
with $c_1(F_1)=(r_1/r)c_1(F)$ and $0<r_1<r$. 
Set $F_2:=F/F_1$, and $r_2:=r-r_1$. 
$H$-stability of $F$ yields the inequality 
$
\frac{\chi(F_2)}{r_2}>\frac{\chi(F)}{r}.
$
We get the injective homomorphism $F_2^*\rightarrow F^*$, and the
inequalities:
\[
\frac{\chi(F_2^*\otimes L)}{r_2}=\frac{\chi(F_2^{**})}{r_2}
\geq \frac{\chi(F_2)}{r_2}>\frac{\chi(F)}{r}=\frac{\chi(F^*\otimes L)}{r}.
\]
Hence, $F^*\otimes L$ is $H$-unstable. 
Consequently, $F^*$ is $H$-unstable. 

If $r=2$, then $F^*\otimes L$ is isomorphic to $F$
and thus $F^*$ is $H$-stable.
\end{proof}

\begin{lem}
\label{lem-H-slope-stability-implies-s-geq-r}
Let $F$ be a locally free $H$-slope-stable sheaf of class $v=(r,0,-s)$. 
Then $v=(1,0,1)$ or $s\geq r\geq 2$.
\end{lem}

\begin{proof}
If $\rank(F)=1$, then $F$ is isomorphic to $\StructureSheaf{S}$, since
$F$ is locally free, and so $s=-1$.
Assume that $v\neq (1,0,1)$. 
Then $H^0(F)$ vanishes, by the $H$-stability of $F$.
Similarly, 
$
H^2(F)^*\cong H^0(F^*)= (0),
$
by the $H$-slope-stability of $F^*$. 
Thus, $r-s=\chi(F)=-\dim H^1(F)\leq 0$.
\end{proof}

Lemma \ref{lem-H-slope-stability-implies-s-geq-r} states a
necessary cohomological condition for the existence of a locally free 
$H$-slope-stable sheaf of class $v$ with $c_1(v)=0$ (slope $0$). 
The condition states that $\chi(v)\leq 0$,
unless $v$ is the class $u=(1,0,1)$ of the trivial line bundle. 
If $\chi(v)\leq 0$, then the locus of sheaves with non-zero global sections
is expected to have positive co-dimension. 
The condition $\chi(v)\leq 0$ translates to $(u,v)\geq 0$.
The following Lemma
states a similar cohomological condition, for a class $v$ with a non-zero slope.
The role of the trivial line bundle is replaced next by a
simple and rigid sheaf $E$ of the same slope as $v$.

\begin{lem}
\label{lemma-non-existence-of-H-slope-stable-sheaves-with-slope-half}
Let $F$ be a locally free $H$-slope-stable sheaf of class $v=(2r,rh,-b)$,
where $r>0$, $\gcd(r,b)=1$, $(h,h)=2d$, and $d$ is an odd integer.
Set $u:=(2,h,(d+1)/2)$. If $(v,v)=-2$, then $v=u$.
Otherwise, $(v,u)\geq 0$ and $(v,u)$ is even. 
Furthermore, $(v,u)=0$, if and only if $v=(2,h,(d-1)/2)$.
\end{lem}

\begin{proof}
$M_H(u)$ consists of a single isomorphism class. 
Let $E$ be an $H$-stable sheaf of class $u$. Then $E$ is necessarily $H$-slope-stable and locally free \cite{mukai-hodge}. 
$M_H(v)$ is non-empty, by assumption.
Let $2n$ be its dimension. Then $2n-2=(v,v)=2dr^2+4rb$, and
$b=\frac{n-1-dr^2}{2r}$. If $n=0$, then $r=1$, 
since $r$ divides $(v,v)/2$.
We conclude that $v=u$, if $(v,v)=-2$.

Assume that $v\neq u$. Then $n\geq 1$. We get the inequality
\[
\frac{\chi(u)}{2}=\frac{5+d}{4}>
\frac{(4+d)r^2+1-n}{4r^2}=\frac{\chi(v)}{2r}.
\]
Thus, $\Hom(E,F)=0$. Similarly, 
\[
\Ext^2(E,F)^*\cong \Hom(F,E)\cong \Hom(E^*,F^*)\cong
\Hom(E^*\otimes H,F^*\otimes H)=0,
\]
since $E^*\otimes H\cong E$ and $F^*\otimes H$ is an $H$-slope-stable sheaf
of class $v$.
Thus, $(v,u)=-\chi(v,u)\geq 0$. Furthermore, 
$(v,u)=2\left[r(d-1)/2+b\right]$.
If $(v,u)=0$, then $r$ divides $b$. Hence, $r=1$, since $\gcd(r,b)=1$.
If $r=1$ and $(v,u)=0$, then $b=(1-d)/2$, as claimed.
\end{proof}

\begin{lem}
\label{lemma-condition-for-existence-of-H-slope-stable-sheaves-isotropic-u-case}
Let $F$ be a locally free $H$-slope-stable sheaf of class $v=(2r,rh,-b)$,
where $r>0$, $\gcd(r,b)=1$, $(h,h)=2d$, and $d$ is an even integer.
Set $u:=(2,h,d/2)$. If $(v,v)=0$, then $v=u$. Otherwise, 
$(u,v)$ is a positive even number.
\end{lem}

This lemma has a cohomological interpretation as well.
$M_H(u)$ is two dimensional and it parametrizes locally-free
$H$-slope-stable sheaves \cite{mukai-hodge}.
Let $B\subset M_H(u)\times M_H(v)$ be the correspondence
consisting of pairs $(E,F)$, with non-vanishing $\Hom(E,F)$.
The lemma states that if $v\neq u$, then the expected co-dimension 
$(u,v)+1$ of $B$ is larger than $2$,
and so $B$ is not expected to surject onto $M_H(v)$.

\begin{proof}
Set $n:=(1/2)\dim_\ComplexNumbers M_H(v)=1+(v,v)/2$.
If $n=1$, then $(v,v)=0$, and so $v=ku$, for some positive integer $k$.
If $k>1$, then the moduli space $M_H(ku)$ is 
the $k$-th symmetric product of $M_H(u)$ and it consists entirely of 
$H$-unstable but $H$-semistable sheaves. We are assuming however 
the existence of an $H$-slope-stable sheaf $F$ of class $v$.
Hence, $k=1$ and $v=u$.

Assume that $n>1$. We get
\[
\frac{\chi(v)}{2r}=\frac{(4+d)r^2-n+1}{4r^2}<
\frac{4+d}{4}=\frac{\chi(u)}{2}.
\]
The normalized Hilbert polynomial $p$ of a sheaf $G$ of positive rank is 
the Hilbert polynomial divided by the rank $p(n):=\chi(G\otimes H^n)/\rank(G)$.
The first two leading terms in the normalized Hilbert polynomials of $u$ and $v$ are equal, 
and the constant terms are related by the above inequality.
Hence, $\Hom(E,F)=0$, for every $H$-slope-stable sheaf $E$
of class $u$. Such a sheaf $E$ is necessarily locally free, and
so $E^*\otimes H$ is $H$-slope-stable of class $u$.
We get also the vanishing of $\Ext^2(E,F)$, by the argument
used in the proof of Lemma 
\ref{lemma-non-existence-of-H-slope-stable-sheaves-with-slope-half}.
We conclude the inequality $(u,v)=-\chi(E^*\otimes F)\geq 0$.
Furthermore, $(u,v)=dr+2b$ is even. If $(u,v)=0$, then $(v,v)=0$,
and so $v=u$. The lemma follows.
\end{proof}
%
\subsection{Sufficient conditions}
\label{sec-sufficient-conditions}
%
\subsubsection{The case $c_1(v)=0$.}
Let $r$, $s$ be integers satisfying $s>r>2$ and $\gcd(r,s)=1$.
Set $n:=rs+1$ and $v:=(r,0,-s)$. 
A sheaf $F$ of class $v$ is $H$-stable, if and only if it is $H$-semi-stable.
Hence, $M_H(v)$ is smooth and projective of dimension $2n$.
Let $Exc\subset M$ be the locus of $H$-stable sheaves
of class $v$ that are not locally free or not
$H$-slope-stable. $Exc$ is clearly a closed subset of $M_H(v)$.

\begin{lem}
\label{lem-codimension-of-Exc}
$Exc$ has codimension at least $2$ in $M$. 
\end{lem}

%
%
\begin{proof}
We will use the following notation, in order for large parts of 
the proof to generalize to a proof of Lemma 
\ref{lemma-Exc-has-codimension-at-least-2-again}.
Let $u:=(1,0,1)$ be the class of $\StructureSheaf{S}$. 
Set $\epsilon:=\rank(u)=1$.

\underline{Step 1:} [Jun Li's morphism to the Uhlenbeck-Yau compactification].
Let $Y_H(w)$ be the moduli space of $H$-slope-stable 
locally free sheaves of class $w\in K(S)$.
Let $v_1$, \dots, $v_k$ be distinct classes in $K(S)$, 
with $v_i=(r_i,0,-s_i)$, $r_i>0$, $(v_i,v_i)\geq -2$. 
Let $d_1$, \dots, $d_k$ be positive integers satisfying
\[
r=\sum_{i=1}^kd_ir_i, \ \ \mbox{and} \ \ 
t(\vec{v},\vec{d}):=s-\left(\sum_{i=1}^kd_is_i\right)\geq 0.
\]
Denote the $d$-th symmetric product of $Y_H(v_i)$ by $Y_H(v_i)^{(d)}$. 
Set 
\[
Y(\vec{v},\vec{d}) \ \ := \ \ \prod_{i=1}^k Y_H(v_i)^{(d_i)}\times 
S^{(t(\vec{v},\vec{d}))}.
\]
Note that for $Y_H(v_i)$ to be non-empty, $v_i=(r_i,0,-s_i)$ should satisfy
\begin{equation}
\label{eq-non-emptyness-condition-on-r-s}
r_i=-s_i=1, \ \ \ \mbox{or} \ \ \ s_i\geq r_i\geq 2,
\end{equation}
by Lemma \ref{lem-H-slope-stability-implies-s-geq-r}.
If $r_i=-s_i=1$ then $v_i=u$.  

Let $M^{\mu{ss}}_H(v)$ be the moduli space of
$S$-equivalence classes of $H$-slope-semi-stable sheaves of class $v$
(\cite{huybrechts-lehn-book}, section 8.2).
Then $M^{\mu{ss}}_H(v)$ is a projective scheme. Set theoretically, 
$M^{\mu{ss}}_H(v)$ is the disjoint union of all such varieties
$Y(\vec{v},\vec{d})$.
There exists a projective morphism 
\[
\bar{\phi} \ : \ M_H(v) \ \ \ \longrightarrow \ \ \ M^{\mu{ss}}_H(v)
\]
\cite{jun-li}. 
Each irreducible component of each 
fiber of the morphism $\bar{\phi}$ is unirational,
as it is dominated by an iterated construction of 
open subsets in extension bundles and bundles of punctual Quot-schemes 
(\cite{huybrechts-lehn-book}, Theorem 8.2.11).
The morphism $\bar{\phi}$ is thus generically finite, since $M_H(v)$
is holomorphic symplectic. 

It suffices to prove the inequality
\[
\dim{Y}(\vec{v},\vec{d}) \ \ \ \leq \ \ \ \dim M_H(v)-4,
\]
for all strata $Y(\vec{v},\vec{d})\subset M^{\mu ss}_H(v)$, such that 
$Y(\vec{v},\vec{d})\neq Y_H(v)$.
It would then follow that $Y_H(v)$ is non-empty, 
and the image of $\bar{\phi}$ is contained in the closure 
$\overline{Y}_H(v)$ of $Y_H(v)$ in $M_H^{\mu ss}(v)$.
The fiber of $\bar{\phi}$ over a point of $Y_H(v)$ consists of a single point.
Let $\widetilde{Y}_H(v)$ be the normalization of $\overline{Y}_H(v)$.
The morphism $\bar{\phi}$ would then factor through a surjective birational 
morphism 
\[
\phi:M_H(v)\rightarrow \widetilde{Y}_H(v),
\]
since $M_H(v)$ is smooth and irreducible, 
and $Exc=M_H(v)-\nolinebreak Y_H(v)$ would be the exceptional locus of $\phi$. 
It would also follow that the singular locus of $\overline{Y}_H(v)$
has co-dimension $\geq 4$ in $\overline{Y}_H(v)$. 
It would then follow that $Exc$ has co-dimension $\geq 2$ in $M_H(v)$, by
Proposition \ref{prop-dissident-locus}.

\underline{Step 2:} [Upper bounds for $\dim{Y}(\vec{v},\vec{d})$].
Fix a stratum $Y(\vec{v},\vec{d})$. Set $t:=t(\vec{v},\vec{d})$,
and $v':=(r,0,t-s)$. Then $v'=\sum_{i=1}^kd_iv_i$.
Set
\[
c(\vec{v},\vec{d}) \ \ := \ \ \dim M_H(v)-\dim Y(\vec{v},\vec{d}).
\]
We compute:
\begin{eqnarray}
\nonumber
c(\vec{v},\vec{d})&=&
2+(v,v)-\sum_{i=1}^k d_i[(v_i,v_i)+2]-2t
\\
\nonumber
&=&
2+2t(\epsilon r-1)+(v',v')-\sum_{i=1}^k d_i[(v_i,v_i)+2]
\\
\label{eq-c-vec-r-vec-d}
&=& 2+2t(\epsilon r-1)+\sum_{i=1}^k\sum_{j=1}^kd_id_j(r_is_j+r_js_i)-
2\sum_{i=1}^kd_i[r_is_i+1].
\end{eqnarray}

\underline{Case 1:} Suppose that $v_i\neq u$, for all $i$.
Then $s_i\geq r_i\geq 2$, for $1\leq i\leq k$.
Write $c(\vec{v},\vec{d})$ in
the form
\[
2+2t(\epsilon r-1)+2\sum_{i=1}^{k-1}d_i\sum_{j=i+1}^kd_j(r_is_j+r_js_i)
+2\sum_{i=1}^kd_i[(d_i-1)r_is_i-1].
\]
\underline{Case 1.1:} Assume that $k=1$. Then 
\[
c(\vec{v},\vec{d})=2+2t(\epsilon r-1)+2d_1[(d_1-1)r_1s_1-1].
\]
\underline{Case 1.1.1:} 
If $d_1=1$, then $c(\vec{v},\vec{d})=2t(\epsilon r-1)\geq 4t$.
If $t=0$, we are in the open subset where $F$ is locally free and 
$H$-slope-stable. If $t>0$, we see that 
indeed $c(\vec{v},\vec{d})\geq 4$.

\noindent
\underline{Case 1.1.2:} 
If $d_1>1$, then $2d_1[(d_1-1)r_1s_1-1]$ is a positive even number,
so $c(\vec{v},\vec{d})\geq 4+2t(\epsilon r-1)\geq 4.$

\noindent
\underline{Case 1.2:} 
Assume that $k>1$. Then 
\begin{eqnarray*}
c(\vec{v},\vec{d}) & = & 2+2t(\epsilon r-1)+2(A+B), \ \ \ \mbox{where}
\\
A&=&\sum_{i=1}^{k-1}d_i\left\{
\left[\sum_{j=i+1}^kd_j(r_is_j+r_js_i)\right]+
\left[(d_i-1)r_is_i-1\right]
\right\},
\\
B&=& d_k[(d_k-1)r_ks_k-1].
\end{eqnarray*}
We are assuming that $s_i\geq r_i\geq 2$. Hence, 
$(r_is_j+r_js_i)\geq 2r_ir_j\geq 8$.
Hence, $A\geq 7$. Now $B>0$, if $d_k>1$, and $B=-1$, if $d_k=1$.
The desired inequality $c(\vec{v},\vec{d})\geq 4$ follows.

\noindent
\underline{Case 2:} Assume that $v_1=u$.
Note that $r_1s_1+1=0$ and $r_1s_j+r_js_1=s_j-r_j=(u,v_j)$.
Equation (\ref{eq-c-vec-r-vec-d}) becomes
\begin{eqnarray*}
c(\vec{v},\vec{d})&=&A+B+C, \ \ \ \mbox{where}
\\
A& = & 2+2t(\epsilon r-1)-2d_1^2
\\
B& = & 2d_1\sum_{j=2}^kd_j(s_j-r_j)
=2d_1(u,v'-d_1v_1)
=2d_1(2d_1+s-r-\epsilon t).
\\
C & = & 
\sum_{i=2}^k\sum_{j=2}^kd_id_j(r_is_j+r_js_i)-2\sum_{i=2}^kd_i[r_is_i+1],
\end{eqnarray*}
Note the equality
\[
A+B=2+2t(\epsilon r-\epsilon d_1-1)+2d_1(s-r+d_1).
\]

\noindent
\underline{Case 2.1:}
Assume that $k=1$. Then $r=d_1$, $t=(r+s)/\epsilon$, and 
\\
$c(\vec{v},\vec{d})=A=
2+2[(r+s)\left(r-\frac{1}{\epsilon}\right)-r^2]\geq
2+2[(s-1)(r-1)-1]\geq 2r(r-1)\geq 12$.

\noindent
\underline{Case 2.2:} Assume that $k\geq 2$. 
Then $1\leq \epsilon d_1\leq \epsilon r-d_2r_2\leq \epsilon r-2$.
So 
\begin{eqnarray*}
A+B & \geq & 2+2t+2d_1(s-r+1)\geq 2+4d_1\geq 6.
\\
C/2 & = & \sum_{i=2}^{k-1}\sum_{j=i+1}^kd_id_j(r_is_j+r_js_i)+
\sum_{i=2}^k[d_i^2r_is_i-d_i(r_is_i+1)]
\\
 & = & \sum_{i=2}^{k-1}d_i\left(
\left[\sum_{j=i+1}^kd_j(r_is_j+r_js_i)\right]+[(d_i-1)r_is_i]-1
\right)
\\ & & 
+
d_k[(d_k-1)r_ks_k-1].
\end{eqnarray*}
If $k=2$ and $d_2=1$, then $C=-2$. Otherwise, $C\geq 0$. 
We conclude that $c(\vec{v},\vec{d})\geq 4$.
This completes the proof of Lemma \ref{lem-codimension-of-Exc}.
\end{proof}

%
%
\subsubsection{The case with slope equal one half.}
Let $r$ be a positive odd integer, 
$\sigma$ a positive integer, 
and set $n:=r\sigma+1$. Assume that $r\geq 3$, $\sigma\geq 3$, 
and  $\gcd(r,\sigma)=1$. 
Let $S$ be a $K3$ surface with a cyclic Picard group generated by 
an ample line bundle $H$. Set $d:=\deg(H)/2$. 
Choose $(S,H)$, so that $\sigma$ and $d$ have the same parity.
If $d$ is odd, assume  that $\sigma>r$, possibly after interchanging 
$r$ and $\sigma$.
Set $h:=c_1(H)$ and $v := (2r,rh,-b)$, 
where $b:=[\sigma-rd]/2$.
Note that $\gcd(r,b)=\gcd(r,\sigma)=1$. 
Hence, $v$ is a primitive class in $K(S)$, $(v,v)=2n-2$,
and the moduli space $M_H(v)$ is smooth and projective of type 
$K3^{[n]}$. 
Let $Exc\subset M_H(v)$ be the locus parametrizing sheaves $F$
that are not locally free or not $H$-slope-stable. 

\begin{lem}
\label{lemma-Exc-has-codimension-at-least-2-again}
$Exc$ is an algebraic subset of co-dimension\footnote{Note
that the assumption $\sigma>r$, adopted above when $d$ is odd, is necessary,
since otherwise $Exc=M_H(v)$, by Lemma
\ref{lemma-non-existence-of-H-slope-stable-sheaves-with-slope-half}.}
$\geq 2$ in $M_H(v)$.
\end{lem}

\begin{proof}
{\bf Proof in the case $d$ is odd:}
When $d$ is odd, then $\sigma$ is odd.
Set $u:=(2,h,(d+1)/2)$ and $s:=r+(v,u)$. 
Then $s=\sigma$ is odd. Thus, $s>r$, by assumption, and $s-r$ is even,
so $s\geq r+2$.

Given an $H$-slope-stable locally free sheaf $F_i$ of class $v_i=(2r_i,r_ih,-b_i)$,
set $s_i:=r_i+(v_i,u)$. 
If $v_i=u$, then $s_i=-1$. 
If $v_i\neq u$, then $s_i\geq r_i$, by Lemma 
\ref{lemma-non-existence-of-H-slope-stable-sheaves-with-slope-half}.
Furthermore, $s_i=r_i$, if and only if $v_i=(2,h,(d-1)/2)$.
If $v_i\neq u$ and $s_i\neq r_i$, then 
$s_i\geq r_i+2$, since $s_i-r_i$ is even, by Lemma 
\ref{lemma-non-existence-of-H-slope-stable-sheaves-with-slope-half}.

With the above notation of $s$ and $s_i$, the proof is almost identical to
that of Lemma \ref{lem-codimension-of-Exc}.
Following are the necessary changes.
Replace the class $(1,0,1)$ by the class $u$ defined above.
Then $\epsilon=\rank(u)=2$. Set $\lambda:=(0,0,-1)$.
Then $(u,\lambda)=2$ and $(\lambda,\lambda)=0$.
With the above definition of $s_i$, we have
$v_i=r_iu+\left(\frac{s_i+r_i}{2}\right)\lambda$. 
Hence,
\[
(v_i,v_j)  \ \ \ = \ \ \ r_is_j+r_js_i, 
\]
as in the proof of Lemma \ref{lem-codimension-of-Exc}.

Equation (\ref{eq-non-emptyness-condition-on-r-s}) is replaced by
\[
r_i=-s_i=1, \ \ \ \mbox{or} \ \ \ s_i=r_i=1, \ \ \ \mbox{or} \ \ \ 
s_i\geq r_i+2\geq 3,
\]
by Lemma \ref{lemma-non-existence-of-H-slope-stable-sheaves-with-slope-half}.
Equation (\ref{eq-c-vec-r-vec-d}), 
for the co-dimension $c(\vec{v},\vec{d})$ of $Y(\vec{v},\vec{d})$,
remains valid.
The argument for case 1.1.1 remains valid.
In case 1.1.2 the term $2d_1[(d_1-1)r_1s_1-1]$ vanishes, if $d_1=2$
and $r_1=s_1=1$. However, in that case $(u,v')=(u,2v_1)=0$
and  $2rt=(u,v-v')=(u,v)=s-r>0$. So $t>0$ and
$c(\vec{v},\vec{d})=2+2t(2r-1)\geq 4$.

In case 1.2 we are no longer assuming that $r_i\geq 2$.
However, since all $v_i$ are different from $u$, and $k>1$,
then at least one $v_i$, say $v_1$, is different from
$(2,h,(d+1)/2)$. Then $s_1\geq r_1+2\geq 3$.
Thus $r_1s_2+r_2s_1\geq 4$ and $A\geq 3$. 
The rest of the argument in case 1.2 is identical.

In case 2, the equations for $c(\vec{v},\vec{d})$, $A$, $B$, and $C$,
remain valid. The argument in case 2.1 remains valid.
The 
inequality $\epsilon r -d_2r_2\leq \epsilon r-2$,
in the first line of case 2.2, need not hold. Nevertheless, 
$\epsilon r -\epsilon d_1\geq \epsilon d_2\geq 2$. So
\[
A+B\geq 2 + 2d_1(s-r+1)\geq 2+6d_1\geq 8.
\]
The rest of the argument remains valid.

{\bf Proof\footnote{The cases where $\sigma<r$ were proven earlier in \cite[Lemma 3.1]{yoshioka-irreducibility}.} 
in the case $d$ is even.}
When $d$ is even then $\sigma$ is even. 
Set $u:=(2,h,d/2)$. By assumption, $r\geq 3$, $s\geq 3$,
and $\gcd(r,s)=1$. 

Given an $H$-slope-stable sheaf $F_i$ of class $v_i=(2r_i,r_ih,-b_i)$, 
with $r_i>0$, set $\sigma_i:=(u,v_i)=2b_i+r_id$. 
If $v_i=u$, then $\sigma_i=0$. If $v_i\neq u$, then
$\sigma_i$ is a positive even integer, by Lemma
\ref{lemma-condition-for-existence-of-H-slope-stable-sheaves-isotropic-u-case}.
Note also that $(u,v)=\sigma$.

The proof is again almost identical to that of Lemma 
\ref{lem-codimension-of-Exc}.
Following are the necessary changes.
Replace the class $(1,0,1)$ by the class $u$ defined above.
Then $\epsilon=\rank(u)=2$. Set $\lambda:=(0,0,-1)$.
Then $(u,\lambda)=2$ and $(\lambda,\lambda)=0$.
With the above definition of $\sigma_i$, we have
$v_i=r_iu+\left(\frac{\sigma_i}{2}\right)\lambda$. 
Hence,
\[
(v_i,v_j)  \ \ \ = \ \ \ r_i\sigma_j+r_j\sigma_i, 
\]
and we replace $s_i$ by $\sigma_i$ 
in the proof of Lemma \ref{lem-codimension-of-Exc}.
Then Equation (\ref{eq-c-vec-r-vec-d}), 
for the co-dimension $c(\vec{v},\vec{d})$ of $Y(\vec{v},\vec{d})$,
remains valid.

Equation (\ref{eq-non-emptyness-condition-on-r-s}) is replaced by
\[
v_i=u, \ \ \ \mbox{or} \ \ \ \sigma_i \ \ \ \mbox{is a positive even integer,} 
\]
by Lemma 
\ref{lemma-condition-for-existence-of-H-slope-stable-sheaves-isotropic-u-case}.
The argument for case 1.1 remains valid.

In case 1.2 we are no longer assuming that $r_i\geq 2$.
However, since all the $v_i$ are different from $u$, then $\sigma_i\geq 2$, 
for all $i$. Thus, $r_1\sigma_2+r_2\sigma_1\geq 4$ and $A\geq 3$.
The rest of the argument is identical.

In case 2, $v_1=u$, $\sigma_1=0$, and 
$(r_1\sigma_j+r_j\sigma_1)=\sigma_j=(u,v_j)$.
Then $c(\vec{v},\vec{d})=A+B+C$, where
\begin{eqnarray*}
A&=& 2+2t(\epsilon r-1),
\\
B&=& 2d_1\sum_{j=2}^k d_j\sigma_j=2d_1(u,v'-d_1v_1)=
2d_1(\sigma-\epsilon t),
\end{eqnarray*}
and $C$ remains the same. Then 
\[
A+B \ \ \ = \ \ \ 2 + 2t(\epsilon r - \epsilon d_1 -1) +2d_1\sigma.
\]

In case 2.1, $r=d_1$, $v'=ru$,
$\epsilon t =(v-v',u)=\sigma$, and 
$c(\vec{v},\vec{d})=A=2+\sigma(\epsilon r -1)$.
Hence $c(\vec{v},\vec{d})\geq 17$. 

In case 2.2 we assume that $k\geq 2$ and so $\epsilon(r-d_1)\geq 2$. 
So $A+B\geq 2 +2d_1\sigma\geq 8$. The rest of the argument is the same.
This completes the proof of Lemma 
\ref{lemma-Exc-has-codimension-at-least-2-again}.
\end{proof}

%
\section{Examples of prime exceptional divisors}
\label{sec-examples}

%
Let $e$, $E$ and $X$ be as in Theorem \ref{thm-2}.
Set $n:=\dim_{\ComplexNumbers}(X)/2$.
The pair $(X,E)$ has the following elementary invariants:
\begin{enumerate}
\item
\label{item-degree}
$(e,e)=-2$, or $(e,e)=2-2n$.
\item
The divisibility $\div(e,\bullet)$ of the class $(e,\bullet)$ in 
$H^2(X,\Integers)^*$ is equal to $(e,e)$ or $(e,e)/2$.
\item
\label{item-invariant-k}
Write $[E]=ke$, where $e$ is a primitive class in $H^2(X,\Integers)$.
Then $k=1$, or $k=2$.
\end{enumerate}

Set $[E]^\vee:=\frac{-2([E],\bullet)}{([E],[E])}$.
We have $[E]^\vee=e^\vee$ or $[E]^\vee=2e^\vee$,
where 
$e^\vee$ is a primitive class in 
$H_2(X,\Integers)$, and the coefficient 
is determined by Lemma \ref{lemma-divisibility} in terms of
the invariant $\div(e,\bullet)$ and the coefficient
$k$ in (\ref{item-invariant-k}) above. 
In particular, if $\div(e,\bullet)=(e,e)/2$, then $[E]=e$, by 
Lemma \ref{lemma-divisibility}.

\begin{thm}
\label{thm-a-uniruled-divisor-is-exceptional}
Let $X$ be a smooth projective holomorphically symplectic variety, and $E$
a prime divisor on $X$.
\begin{enumerate}
\item
\label{thm-item-uniruled-implies-exceptional}
(\cite{druel}, Theorem 1.3)
Assume that through a generic point of $E$ passes a rational curve of class 
$\ell\in H_2(X,\Integers)$, such that $[E]\cdot \ell<0$. 
Then $E$ is an exceptional divisor.
\item
\label{thm-item-ell-is-E-vee}
Let $E$ and $\ell$ be as in part \ref{thm-item-uniruled-implies-exceptional}
and $\pi:X'\rightarrow Y$  the birational contraction of $E$ introduced in
Proposition \ref{prop-druel}. Then
$Y$ has $A_i$ singularities away from its dissident locus,\footnote{
See the paragraph preceding Proposition \ref{prop-dissident-locus} 
for the definition of the dissident locus.} 
and $i=1$ or $i=2$. 
Furthermore, 
$\ell=\left\{\begin{array}{ccc}
[E]^\vee & \mbox{if} & i=1,
\\
\frac{1}{2}[E]^\vee & \mbox{if} & i=2.
\end{array}\right.$
\end{enumerate}
\end{thm}

\begin{proof}
We need only prove part \ref{thm-item-ell-is-E-vee}.
$Y$ has $A_1$ or $A_2$ singularities, by Corollary 
\ref{cor-1}. Let $E'$ be the strict transform of $E$ in $X'$. The 
generic fiber of the restriction of $\pi$  to $E'$ is a rational curve, or a pair of rational curves 
joined at a node. The exceptional locus of the birational transformation from $X$ to $X'$
does not dominate $\pi(E')$, by the proof of Proposition 1.4 in
\cite{druel}. 
The morphism $\pi$ thus restricts to a rational morphism from $E$ to $\pi(E')$,
whose generic fiber is isomorphic to the generic fiber of $E'$ over $\pi(E')$.
The class $\ell$ must be the class of an irreducible component of 
the generic fiber of the restriction of $\pi$ to $E$, 
by the uniqueness of the family of rational curves, which dominates $E$
(\cite{druel}, Proposition 4.5). The equality $\ell=\frac{1}{i}[E]^\vee$ 
follows from part \ref{item-integrality} of Corollary \ref{cor-1}.
\end{proof}

We will say that {\em the prime exceptional divisor is of type} $A_i$, if 
the variety $Y$ in part \ref{thm-item-ell-is-E-vee}
of Theorem \ref{thm-a-uniruled-divisor-is-exceptional} has $A_i$ 
singularities away from its dissident locus. 
All prime exceptional divisor studied in this paper are of type $A_1$.

%
\subsection{Brill-Noether exceptional divisors}
\label{sec-Brill-Neother-exceptional-divisors}
Let $S$ be a $K3$ surface, $F_0$ a simple and rigid coherent sheaf, i.e.,
a sheaf satisfying $\End(F_0,F_0)\cong\ComplexNumbers$, and
$\Ext^1(F_0,F_0)=0$. Then the class $v_0$ of $F_0$ is a primitive
class in $K(S)$ with $(v_0,v_0)=-2$. Examples of exceptional divisors
$E$ of degree $-2$ in moduli spaces of sheaves on $S$ seem to arise 
as Brill-Noether loci as follows. Let $v\in K(S)$ be a class 
satisfying $(v_0,v)=0$, and such that there exists 
a smooth and compact moduli space $M(v)$ of stable sheaves of class $v$.
The locus $M(v)^1$, of points representing sheaves $F$ with non-vanishing
$\Ext^1(F,F_0)$, is often an exceptional divisor of degree $-2$. 
The examples considered in this section are all of this type.

\begin{example} 
\label{example-case-degree-e-minus-2-div-1}
The case 
$n\geq 2$, $(e,e)=-2$, and $\div(e,\bullet)= 1$.
\hide{
We provide two examples. The first example is simpler, 
but Druel's birational map $X\rightarrow Y$, which contracts $E$, is not
a regular morphism. In the second example it is.

a) 
}
Let $S$ be a $K3$ surface containing a smooth irreducible 
rational curve $\Sigma$. Let $E\subset S^{[n]}$, $n\geq 2$, be the divisor
consisting of length $n$ subschemes intersecting $\Sigma$ along
a non-empty subscheme. The class $[E]$ is identified with $[\Sigma]$,
under the embedding of $H^2(S,\Integers)$ as an orthogonal direct summand
in the decomposition (\ref{eq-orthogonal-direct-sum}) 
of $H^2(S^{[n]},\Integers)$.
Thus $([E],[E])=([\Sigma],[\Sigma])=-2$ and $\div([E],\bullet)=1$. 
$E$ is of type $A_1$. 
Let $F_0$ be the direct image of $\StructureSheaf{\Sigma}(-1)$ and
$v_0\in K(S)$ the class of $F_0$. Then $v_0$ is orthogonal to the 
class of the ideal sheaf $I_Z$ of a length $n$ subscheme $Z$ of $S$, and 
$E$ is the Brill-Noether locus, where $\Ext^1(I_Z,F_0)$ does not vanish.
\end{example}

Let $S$ be a projective $K3$ surface, with a cyclic Picard group generated by
an ample line bundle $H$ of degree $d\geq 2$. 
In the remainder of this section the simple and rigid sheaf $F_0$ will be 
$\StructureSheaf{S}$. Then $\Ext^1(F,\StructureSheaf{S})\cong H^1(F)^*$,
by Serre's Duality. 
We will need the following results.

\begin{lem}
\label{lemma-stability-of-co-kernel-for-primitive-c-1}
(\cite{markman-reflections}, Lemma 3.7 part 3)
Let $F$ be an $H$-stable sheaf on $S$ of rank $r$ and determinant $H$,
and $U\subset H^0(F)$ a subspace of dimension $r'\leq r$. Then the evaluation
homomorphism $U\otimes\StructureSheaf{S}\rightarrow F$ is injective and its
co-kernel is an $H$-stable sheaf.
\end{lem}

Consider the Mukai vector $v:=(r,H,s)$ and assume that $r\geq 0$ and 
$r+s\geq 0$.
Set $v_0:=(1,0,1)$. Let $M_H(v)^t$ be the Brill-Noether locus 
of $H$-stable sheaves $F$ with $h^1(F)\geq t$. 

\begin{thm} 
\label{thm-brill-noether}
(\cite{markman-reflections}, Corollary 3.19, \cite{yoshioka-brill-noether})
\begin{enumerate}
\item
$M_H(v)^t$ is empty, if and only if $M_H(v+tv_0)$ is.
\item
There exists a smooth surjective projective morphism
\[
f_t \ : \ \left[M_H(v)^t\setminus M_H(v)^{t+1}\right] \ \ \ 
\longrightarrow \ \ \ 
\left[M_H(v+tv_0)\setminus M_H(v+tv_0)^1\right].
\]
\item
\label{thm-fibers-are-grassmannians}
The fiber of $f_t$, over a point representing a sheaf $E$, is naturally 
isomorphic to the Grassmannian $G(t,H^0(E))$. 
Furthermore, $H^0(E)$ is $r+s+2t$-dimensional, and the dimension 
$t(r+s+t)$ of
the fiber is equal to the co-dimension of $M_H(v)^t$ in $M_H(v)$.
\item
\label{thm-item-brill-neother-divisor-has-class-minus-v0}
(\cite{markman-part-two}, Lemma 4.11)
If $s=-r$, then $M(v)^1$ is a prime divisor of class $\theta(-v_0)$.
The class $\ell\in H^2(M_H(v),\Integers)^*$ of a 
$\PP^1$-fiber of $f_1$ is $(\theta(-v_0),\bullet)$.
\end{enumerate}
\end{thm}

The embedding $G(t,H^0(E))\hookrightarrow M_H(v)$ in part
\ref{thm-fibers-are-grassmannians} sends a $t$-dimensional subspace 
$U\subset H^0(E)$ to the co-kernel of the evaluation
homomorphism $U\otimes \StructureSheaf{S}\rightarrow E$.
The co-kernel is stable, by Lemma 
\ref{lemma-stability-of-co-kernel-for-primitive-c-1}.

\begin{example} 
\label{example-case-degree-e-minus-2-div-1-brill-noether}
The case $n\geq 2$, $(e,e)=-2$, and $\div(e,\bullet)= 1$ 
was considered in Example \ref{example-case-degree-e-minus-2-div-1}.
Additional examples of such prime exceptional divisors 
are provided in part \ref{thm-item-brill-neother-divisor-has-class-minus-v0}
of Theorem \ref{thm-brill-noether}. $M_H(v)^1$ is exceptional, 
since it is prime of degree $-2$, by part 
\ref{thm-item-brill-neother-divisor-has-class-minus-v0}
of Theorem \ref{thm-brill-noether}. 
Examples of prime exceptional 
Brill-Noether divisors, for more general simple and rigid sheaves, 
can be found in the work of Yoshioka \cite{yoshioka-brill-noether}.
\end{example}

\begin{example}
\label{example-degree-e-minus-2-div-2}
The case $(e,e)=-2$, $\div(e,\bullet)=2$, and $[E]=e$. 
\\
Assume that $n$ is congruent to $2$ modulo $4$ and $n\geq 6$. 
Let $S$ be a projective $K3$ surface with a cyclic Picard group generated by
an ample line bundle $H$ of degree $(H,H)=\frac{n-2}{2}$. 
Then $h^i(H^2)=0$, for $i>0$, and $h^0(H^2)=n$. 
Set $X:=M_H(1,H^2,-1)\cong S^{[n]}$. 
Let $E:=M_H(1,H^2,-1)^1$ be the Brill-Noether divisor
in $M_H(1,H^2,-1)$ of sheaves $F$ with $h^1(S,F)>0$.

We recall the explicit definition of $E$.
Let $\pi_i$, $i=1,2$, be the projection from $S\times M_H(1,H^2,-1)$
onto the $i$-th factor. Let $\Z\subset S\times M_H(1,H^2,-1)$
be the universal subscheme, and $I_\Z$ its ideal sheaf.
Then $\F:=I_\Z\otimes \pi_1^*H^2$ is a universal sheaf over 
$S\times \M_H(1,H^2,-1)$.
We have the short exact sequence
\[
0\rightarrow \F \rightarrow \pi_1^*H^2\rightarrow 
\StructureSheaf{\Z}\otimes \pi_1^*H^2\rightarrow 0,
\]
and the homomorphism of rank $n$ vector bundles 
\[
g \ : \ H^0(H^2)\otimes\StructureSheaf{M_H(v)}
\cong \pi_{2_*}(\pi_1^*H^2) \ \ \ \longrightarrow \ \ \ 
\pi_{2_*}\left(\StructureSheaf{\Z}\otimes\pi_1^*H^2\right).
\]
The homomorphism $g$ is injective, since a generic length $n$ 
subscheme of $S$ induces $n$ independent conditions on a
linear system $\linsys{L}$, provided the line bundle $L$ on $S$ satisfies 
$h^0(L)\geq n$.
The Brill-Noether divisor is the zero divisor of $\Wedge{n}g$.
$E$ is an effective divisor of class $\theta(-v_0)$, where
$v_0:=(1,0,1)$ is the class in $K(S)$ of the trivial line bundle,
and $\theta$ is the Mukai isomorphism given in (\ref{eq-Mukai-isomorphism}).

\end{example}

\begin{lem}
\label{lemma-brill-noether-divisor-is-exceptional}
$E$ is a prime exceptional divisor
of class $e:=\theta(-v_0)$.
In particular, $(e,e)=-2$ and $\div(e,\bullet)=2$.
\end{lem}

The rest of section \ref{sec-Brill-Neother-exceptional-divisors} 
is devoted to the 
proof of Lemma \ref{lemma-brill-noether-divisor-is-exceptional}.

\begin{lem}
\label{lemma-stability-criteria}
Let $F$ be an $H$-slope-stable sheaf of class $(2,H^2,0)$.
\begin{enumerate}
\item
\label{lemma-item-Q-is-torsion-free}
For every non-zero section $s\in H^0(F)$, the  evaluation
homomorphism $s:\StructureSheaf{S}\rightarrow F$ has a rank $1$ torsion free
co-kernel sheaf.
\item
\label{lemma-item-G-is-H-slope-stable}
If $\epsilon\in \Ext^1(\StructureSheaf{S},F)$ is a non-zero class and
\[
0\rightarrow \StructureSheaf{S}\rightarrow G_\epsilon\rightarrow F\rightarrow 0
\]
the corresponding extension, then the sheaf $G_\epsilon$ is $H$-slope-stable.
\end{enumerate}
\end{lem}

\begin{proof}
\ref{lemma-item-Q-is-torsion-free})
Denote the co-kernel of $s$ by $Q_s$.
If $T$ is a subsheaf of $Q_s$ with zero-dimensional support, then 
$\Ext^1(T,\StructureSheaf{S})=0$. Thus, the inverse image of $T$ in $F$
would contain a subsheaf isomorphic to $T$. But $F$ is torsion free. 
Hence, the dimension of the support of any subsheaf of $Q_s$ is at least $1$.

If $T$ is a subsheaf of $Q_s$ of one-dimensional support, then its inverse 
image in
$F$ is a rank one subsheaf $F'$ of $F$ with effective determinant line bundle.
Hence, $\det(F')\cong H^k$, for some positive integer $k$. 
This contradicts the slope-stability of $F$. Hence, $Q_s$ is torsion-free.

\ref{lemma-item-G-is-H-slope-stable})
Assume that $G_\epsilon$ is $H$-slope-unstable, and let $G'\subset G_\epsilon$
be an $H$-slope-stable  subsheaf of maximal slope of rank $r\leq 2$.
If $G'$ maps to zero in $F$, then $G'$ is a subsheaf of 
$\StructureSheaf{S}$, and can not destabilize $G_\epsilon$. 
For the same reason, the slope of the image $\bar{G}$ of $G'$  
satisfies $\mu(\bar{G})\geq \mu(G')$. 
Thus $\rank(\bar{G})\neq 1$, since otherwise $\bar{G}$ would destabilize $F$.
Hence, $\rank(G')=2$, $G'$ maps isomorphically onto $\bar{G}$, 
and $\det(G')\cong H^k$, 
for $2/3<k\leq 2$. It follows that $k=2$.

Set $Q:=F/\bar{G}$. 
We get the short exact sequence
\[
0\rightarrow \bar{G}\RightArrowOf{\iota} F \rightarrow Q\rightarrow 0.
\]
$\Ext^1(Q,\StructureSheaf{S})$ vanishes, since $Q$ has zero-dimensional 
support. 
Hence, $\iota^*:\Ext^1(F,\StructureSheaf{S})\rightarrow 
\Ext^1(\bar{G},\StructureSheaf{S})$
is injective. 
On the other hand, the pullback $\iota^*(\epsilon)$ 
vanishes in $\Ext^1(\bar{G},\StructureSheaf{S})$.
This contradicts the assumption that $\epsilon$ is a non-zero class.
\end{proof}

The moduli space $M_H^{ss}(2,H^2,0)$, of $H$-semi-stable sheaves of class 
$(2,H^2,0)$,
is known to be an irreducible normal projective variety of dimension
$2n-2$. Furthermore, the singular locus is equal to the strictly semi-stable 
locus and it 
has co-dimension 2, if $n=6$, and $4$, if $n>6$ (\cite{kaledin-lehn-sorger}, 
Theorem 4.4
and Theorem 5.3). 
A generic $H$-stable sheaf of class $(2,H^2,0)$ is 
$H$-slope-stable. This is equivalent to the corresponding statement
for $M_H(2,0,1-(n/2))$, and follows from the following lemma.

\begin{lem}
\label{lemma-existence-of-slope-stable-vb-of-rank-2}
Let $s$ be an integer $\geq 2$. Then the set of
$H$-stable locally free sheaves of class $(2,0,-s)$ is 
Zariski dense in $M_H^{ss}(2,0,-s)$. Furthermore, 
any $H$-stable locally free sheaf of class $(2,0,-s)$ is 
$H$-slope-stable. 
\end{lem}

The proof of the density statement is similar to that of Lemma 
\ref{lem-codimension-of-Exc} and is omitted. 
The case $s=2$ is proven in \cite{ogrady-10}, Proposition 3.0.5.
The second statement is a special case of 
Lemma \ref{lemma-locally-free-H-stable-of-rank-2-is-slope-stable}.

\begin{lem}
Let $U\subset M_H^{ss}(2,H^2,0)$ be the subset parametrizing 
$H$-slope-stable sheaves $F$ with $h^1(F)=0$. Then $U$ is a 
Zariski-dense open subset. 
\end{lem}

\begin{proof}
Let $M^{\mu s}\subset M_H^{ss}(2,H^2,0)$ be the Zariski open subset of 
$H$-slope-stable sheaves. Note that $M^{\mu s}$ is a dense subset, by Lemma 
\ref{lemma-existence-of-slope-stable-vb-of-rank-2}.
Let $t$ be the minimum
of the set $\{h^1(F) \ : \  [F]\in M^{\mu s}\}$. 
It suffices to prove that $t=0$. Assume that $t>0$.
Let $U'\subset M^{\mu s}$ be the Zariski-open subset 
of sheaves $F$ with $h^1(F)=t$. 
Let $p:\PP\rightarrow U'$ be the projective bundle with fiber 
$\PP{H}^1(F)^*$ over $F$. $\PP$ is a 
Zariski open subset of the moduli space of coherent systems
constructed by Le Potier in \cite{le-potier-coherent}.
A point in $\PP$ parametrizes an equivalence class 
of a pair $(F,\ell)$, consisting of an $H$-slope-stable sheaf $F$ of 
class $(2,H^2,0)$ and a one-dimensional subspace 
$\ell\subset \Ext^1(F,\StructureSheaf{S})$. 
We have $\dim(\PP)=\dim(U')+t-1=2n+t-3$.

There exists a natural morphism 
\[
f : \PP \ \ \ \longrightarrow \ \ \ M_H(3,H^2,1),
\]
sending a pair $(F,\ell)$ to the isomorphism class of the 
sheaf $G_\ell$ in the natural extension
$
0\rightarrow \ell^*\otimes\StructureSheaf{S}\rightarrow G_\ell\rightarrow F
\rightarrow 0.
$
$G_\ell$ is $H$-slope-stable, by Lemma \ref{lemma-stability-criteria} part 
\ref{lemma-item-G-is-H-slope-stable}.

Now $h^0(G_\ell)=h^0(F)+1=t+3$. Furthermore, the data $(F,\ell)$
is equivalent to the data $(G_\ell,\ell)$, where $\ell$ is a one-dimensional 
subspace
of $H^0(G_\ell)$. 
Hence, the fiber of $f$, over the isomorphism class
of $G_\ell$, has dimension at most $t+2$. 
The dimension of $M_H(3,H^2,1)$ is $2n-8$.
Thus, $\dim(\PP)\leq 2n+t-6.$
This contradicts the above computation of the dimension of $\PP$.
%
%
\end{proof}

Let $G(1,U)$ be the moduli space of equivalence classes of pairs $(F,\lambda)$,
where $F$ is an $H$-slope-stable sheaf of class $(2,H^2,0)$ with $h^1(F)=0$,
and $\lambda\subset H^0(F)$ is a one-dimensional subspace. $G(1,U)$
is a Zariski open subset of the moduli space of coherent systems constructed
by Le Potier in \cite{le-potier-coherent}. The forgetful morphism
$G(1,U)\rightarrow U$ is a $\PP^1$-bundle. Let
\[
\psi \ : \ G(1,U) \ \  \ \longrightarrow \ \ \ M_H(1,H^2,-1)
\]
be the morphism, sending a pair $(F,\lambda)$ to the quotient 
$F/[\lambda\otimes\StructureSheaf{S}]$.
The morphism $\psi$ is well-defined, by Lemma \ref{lemma-stability-criteria}
part \ref{lemma-item-Q-is-torsion-free}.

\begin{lem}
\label{lemma-class-of-psi-of-P-1}
\begin{enumerate}
\item
\label{lemma-item-psi-is-an-isomorphism-onto-open-subset-of-E}
The divisor $E$ is smooth along the image of $\psi$
and $\psi$ maps $G(1,U)$ isomorphically onto a Zariski open subset of $E$.
\item
\label{lemma-item-class-of-rational-curve-is-pairing-with-minus-v-0}
Let $F$ be an $H$-slope-stable sheaf of class $(2,H^2,0)$ with $h^1(F)=0$.
Then $\psi(\PP{H}^0(F))$ is a rational curve of class 
$(\theta(-v_0),\bullet)$ in $H_2(M_H(1,H^2,-1),\Integers)$.
\end{enumerate}
\end{lem}

\begin{proof}
\ref{lemma-item-psi-is-an-isomorphism-onto-open-subset-of-E})
The proof is similar to that of parts 6 and 7 of Proposition 3.18 in
\cite{markman-reflections}. Let us first prove that 
the morphism $\psi$ is injective. 
Let $Q$ be a sheaf represented by the point $\psi(F,\lambda)$.
We know, by construction, that $H^i(F)=0$, for $i>0$, and $h^0(F)=\chi(F)=2$.
Hence, $h^0(Q)=1$, $h^1(Q)=1$, and $h^2(Q)=0$. It follows that 
$\dim\Ext^1(Q,\StructureSheaf{S})=1$, $F$ is isomorphic to the 
unique non-trivial extension of $Q$ by $\StructureSheaf{S}$,
and $\lambda$ is the kernel of the homomorphism 
$H^0(F)\rightarrow H^0(Q)$. Hence, $\psi$ is injective.

The image of $\psi$ is Zariski open in $E$, since it is characterized by 
$\dim \Ext^1(Q,\StructureSheaf{S})=1$, and by the $H$-slope-stability of the
unique non-trivial extension. $G(1,U)$ is clearly smooth. It suffices to 
construct the inverse of $\psi$ as a morphism. This is done as in the
proof of Proposition 3.18 in \cite{markman-reflections}.

\ref{lemma-item-class-of-rational-curve-is-pairing-with-minus-v-0})
Let $w\in K(S)$ be a class orthogonal to $(1,H^2,-1)$.
The equality
\[
\int_{\PP{H}^0(F)}\psi^*(\theta(w)) \ \ \ = \ \ \ -(v_0,w)
\]
follows by an argument identical to the proof of Lemma 4.11 in
\cite{markman-part-two}. It follows that $\psi(\PP{H}^0(F))$ 
has class $(\theta(-v_0),\bullet)$ in $H_2(M_H(1,H^2,-1),\Integers)$.
\end{proof}

\begin{lem}
The closure $E'$ of the image of $\psi$ is a prime exceptional divisor
of class $\theta(-v_0)$.
\end{lem}

\begin{proof}
$G(1,U)$ is a $\PP^1$-bundle over $U$. Hence $G(1,U)$ is irreducible
of dimension $2n-1$. 
The image of $\psi$ is irreducible of dimension $2n-1$, as $\psi$ is 
injective. 
Hence, $E'$ is irreducible. The canonical line bundle of
$G(1,U)$ restricts to the fiber $\PP{H}^0(F)$ as the canonical
line bundle of the fiber, since $U$ is holomorphic-symplectic. 
The normal of $\psi(G(1,U))$ in $M_H(1,H^2,-1)$ is isomorphic to 
the canonical line bundle of $\psi(G(1,U))$, by
Lemma \ref{lemma-class-of-psi-of-P-1}
part \ref{lemma-item-psi-is-an-isomorphism-onto-open-subset-of-E}.
Hence, $E'\cdot \psi[\PP{H}^0(F)]=-2$, and 
$E'$ is exceptional, by Theorem \ref{thm-a-uniruled-divisor-is-exceptional}
part \ref{thm-item-uniruled-implies-exceptional}.
$E'$ is of type $A_1$, by Lemma \ref{lemma-class-of-psi-of-P-1}
part \ref{lemma-item-psi-is-an-isomorphism-onto-open-subset-of-E}.
The class of $E'$ is $\theta(-v_0)$, 
by Lemma \ref{lemma-class-of-psi-of-P-1} and
Theorem \ref{thm-a-uniruled-divisor-is-exceptional} part 
\ref{thm-item-ell-is-E-vee}.
\end{proof}

\begin{proof} {\bf (of Lemma \ref{lemma-brill-noether-divisor-is-exceptional})}
$E$ is an effective divisor of class $\theta(-v_0)$, by definition of $E$. 
This is also the class of the reduced and irreducible divisor $E'$ 
supporting a component of $E$. 
Hence, $E$ is reduced and irreducible. 
We have the equality 
$([E],\theta(x))=(\theta(-v_0),\theta(x))=-(v_0,x)$, 
which is divisible by $2$, for all $x\in (1,H^2,-1)^\perp$,
since $(1,H^2,-1)-v_0=2(0,H,-1)$. 
Hence, $\div([E],\bullet)=2$. 
\end{proof}

%
\subsection{Exceptional divisors of non-locally-free sheaves}
In this section we will consider examples of prime
exceptional divisors that arise as the exceptional locus for the
morphism from the Gieseker-Simpson moduli space of $H$-stable sheaves
to the Uhlenbeck-Yau compactification of the moduli space of 
$H$-slope-stable locally free sheaves. Such divisors 
on a $2n$-dimensional moduli space seem to have 
class $e$ or $2e$, where $e$ is a primitive class of degree $(e,e)=2-2n$. 

\begin{example}
\label{example-diagonal-of-hilbert-scheme}
\cite{beauville}
The case $n\geq 2$, $(e,e)=2-2n$, $\div(e,\bullet)=2n-2$, 
$rs(e)=\{1,n-1\}$, and $[E]=2e$. 
\\
Let $S$ be a $K3$ surface, $X:=S^{[n]}$, 
and $E\subset X$ the big diagonal. Then
$[E]=2e$, for a primitive class $e\in H^2(S^{[n]},\Integers)$, and 
$(e,e)=2-2n$. Hence $[E]^\vee=e^\vee$, by Corollary
\ref{cor-1}.
$E$ is the exceptional locus of the Hilbert-Chow morphism 
$S^{[n]}\rightarrow S^{(n)}$ onto the $n$-th symmetric product.
The symmetric product 
$S^{(n)}$ has $A_1$ singularities away from its dissident locus.
The monodromy-invariant $rs(e)$ is equal to $\{1,n-1\}$,
by Example \ref{example-any-factorization-rs-is-possible}.
\end{example}

The following result will be needed in the next example.

\begin{lem}
\label{lemma-class-of-a-contractible-rational-curve}
Let $S$ be a $K3$ surface, $\LB$ a line bundle on $S$, 
$v=(r,\LB,s)$ a class in $K(S)$, satisfying $(v,v)\geq 2$,
and $r\geq 2$. Let $H$ be a $v$-generic polarization.
Assume given an $H$-slope-stable sheaf $G$ of class $(r,\LB,s+1)$
and a point $P\in S$, such that $G$ is locally free at $P$.
For each $2$-dimensional quotient $Q$ of the
fiber $G_P$, there exists a natural embedding
\[
\kappa \ : \ \PP{Q} \ \ \ \longrightarrow \ \ \ M_H(v),
\]
whose image $C:=\kappa(\PP{Q})$ is a smooth rational curve satisfying
\[
\int_C\theta(x) \ \ \ = \ \ \ (w,x) \ \ \ = \ \ \ -\rank(x),
\]
for all $x\in v^\perp$, 
where $w$ is the following rational class in $v^\perp$.
\begin{equation}
\label{eq-class-wich-pair-with-lambda-to-minus-rank-lambda}
w \ \ := \ \  \frac{r}{(v,v)}v+(0,0,1)
\ \ = \ \ \frac{1}{(v,v)}\left(r^2,r\LB,sr+(v,v)\right).
\end{equation}
\end{lem}

\begin{proof}
Consider the short exact sequence of the tautological quotient bundle 
$q_{\PP{Q}}$ over $\PP{Q}$
\[
0\rightarrow \StructureSheaf{\PP{Q}}(-1)\rightarrow
Q\otimes\StructureSheaf{\PP{Q}}\rightarrow q_{\PP{Q}}
\rightarrow 0.
\]
Let 
\[
\iota \ : \ \PP{Q} \ \ \ \rightarrow \ \ \ S \times \PP{Q}
\]
be the morphism given by $\iota(\ell)=(P,\ell)$.
Let $\pi_i$ be the projection from $S\times \PP{Q}$ onto the $i$-th factor.
Over $S\times \PP{Q}$ we get the short exact sequence
\[
0\rightarrow \F \rightarrow \pi_1^*G\RightArrowOf{j}\iota_*(q_{\PP{Q}})
\rightarrow 0,
\]
where $j$ is the natural homomorphism, and $\F$ its kernel.
Given a point $\ell\in \PP{Q}$, we denote by $\tilde{\ell}\subset G_P$ 
the corresponding hyperplane. The sheaf $\F_\ell$, $\ell\in \PP{Q}$,
is the subsheaf of $G$, with local sections whose values at $P$ belong to 
$\tilde{\ell}$.
$\F_\ell$ is $H$-slope-stable, since $G$ is. $\F$ is thus a family of 
$H$-stable sheaves, flat over $\PP{Q}$.
Let $\kappa:\PP{Q}\rightarrow M_H(v)$ 
be the classifying morphism associated to $\F$. 
The morphism $\kappa$ is clearly injective. An elementary calculation verifies
that the differential $d_\ell\kappa$ is injective.
\hide{
Indeed, the differential $d_\ell\kappa$ is the composition
\[
\Hom(\ell,Q/\ell)\subset \Hom(\tilde{\ell},Q/\ell)
\subset \Hom(F_\ell,(Q/\ell)_P)\rightarrow \Ext^1(F_\ell,F_\ell),
\]
where $(Q/\ell)_P$ is the sky-scraper sheaf supported at $P$ 
with fiber $Q/\ell$. 
The rightmost homomorphism is the connecting homomorphism in the
long exact sequence obtained by applying the functor $\Hom(F_\ell,\bullet)$
to the short exact sequence 
$
0\rightarrow F_\ell\rightarrow G\rightarrow (Q/\ell)_P\rightarrow 0.
$
Each of the homomorphisms in the above composition is clearly injective.
}

Let us compose the Mukai isomorphism $\theta$, given in 
(\ref{eq-Mukai-isomorphism}), with pull-back by $\kappa$
\[
v^\perp \LongRightArrowOf{\theta} H^2(M_H(v),\Integers) 
\LongRightArrowOf{\kappa^*} H^2(\PP{Q},\Integers).
\]
The composition is given by
\[
\kappa^*(\theta(x)) \ \ = \ \ 
c_1\left\{\pi_{2_!}\left[\pi_1^!(x^\vee)\otimes\F\right]\right\}.
\]
Let $[\F]$ be the class of $\F$  in $K(S\times\PP{Q})$.
Then $[\F]=\pi_1^![G]-\iota_![q_{\PP{Q}}]$. 
We have the equalities
\[
c_1\left\{\pi_{2_!}\left(\pi_1^!(x^\vee\otimes G)\right)\right\}
\ \ = \ \ 0,
\]
\begin{eqnarray*}
c_1\left\{\pi_{2_!}\left(\pi_1^!(x^\vee)\otimes \iota_!(q_{\PP{Q}})
\right)\right\}
&=&
c_1\left\{\pi_{2_!}\left(\rank(x)\cdot \iota_!(q_{\PP{Q}})\right)\right\}
\\
&=&
\rank(x)c_1(q_{\PP{Q}}).
\end{eqnarray*}
We conclude that the following equality holds, for all $x\in v^\perp$. 
\[
\int_{\PP{Q}}\kappa^*(\theta(x))=-\rank(x).
\]

A direct calculation verifies that 
the class $w$, given in 
(\ref{eq-class-wich-pair-with-lambda-to-minus-rank-lambda}),
is orthogonal to $v$ and satisfies $(w,x)=-\rank(x)$, for all $x\in v^\perp$. 
\end{proof}

\begin{example} 
Let $S$ be a $K3$ surface with a cyclic Picard group 
generated by an ample line bundle $H$. Let  
$b$ be an odd integer, such that there exists a line bundle 
$\LB\in\Pic(S)$ of degree $2n-4b-2$, where $n>2$.
If $c_1(\LB)$ is divisible by $2$, assume that $n>3$.
Let $v\in K(S)$ be the class $(2,\LB,-b)$ in Mukai's notation. Then 
$(v,v)=2n-2$ and 
the moduli space
$M:=M_H(v)$ is smooth, projective, and $2n$-dimensional. 

Let $E\subset M$ be the closure of the locus of points representing
$H$-stable sheaves $F$, which are not locally free, but such that 
$F^{**}/F$ has length one. 
Let $Y$ be the normalization of the Uhlenbeck-Yau compactification 
of the moduli space of locally free $H$-slope stable sheaves of class $v$. 
Then $Y$ is a projective variety and there exists a 
morphism $\phi:M\rightarrow Y$ whose exceptional locus
contains $E$ \cite{jun-li}. 
\end{example}

\begin{lem}
\label{lemma-class-of-exceptional-locus} 
$E$ is a prime exceptional divisor of type $A_1$. The class 
$[E]\in H^2(M,\Integers)$ is the primitive class 
$e:=\theta(2,\LB,n-b-1)$. In particular, $(e,e)=2-2n$. 
\begin{enumerate}
\item
\label{lemma-item-L-not-divisible-by-2}
If the class $c_1(\LB)$ is not divisible by $2$, then 
$\div(e,\bullet)=n-1$, 
\[
rs(e)=\left\{
\begin{array}{ccl}
\{1,n-1\}, & \mbox{if}& n \ \mbox{is even},
\\
\{1,(n-1)/2\}, & \mbox{if}& n \ \mbox{is odd}.
\end{array}\right.
\]
\item
\label{lemma-item-L-divisible-by-2}
If the class $c_1(\LB)$ is divisible by $2$, then $n\equiv 3$ (modulo $4$), 
$\div(e,\bullet)=2n-2$, $rs(e)=\{2,(n-1)/2\}$.
\end{enumerate}
\end{lem}

\begin{proof} 
When the class $c_1(\LB)$ is divisible by $2$ then $n\equiv 3$
(modulo $4$), since $\deg(\LB)=2n-4b-2$ is divisible by $8$.
In that case $n\geq 7$, by assumption.

If $F^{**}/F$ has length one and $F$ is
$H$-stable, then the reflexive hull $F^{**}$ is necessarily 
$H$-semi-stable\footnote{The $H$-semi-stability is proven by an easy
modification of the proof of \cite{markman-constraints}, 
Proposition 4.10, Part 1. The assumption that the rank is $2$ is needed.
} 
of class $u:=(2,\LB,1-b)\in K(S)$. 
$M_H(u)$ is irreducible of dimension $2n-4$ and its generic
point represents a locally free $H$-slope-stable sheaf.
This is clear if $c_1(\LB)$ is not divisible by $2$.
If $c_1(\LB)$ is divisible by $2$, this follows from Lemma 
\ref{lemma-existence-of-slope-stable-vb-of-rank-2} and 
the assumption that $n\geq 7$. 

Let $U\subset M_H(u)$ be the
locus of $H$-stable locally free sheaves.
Choose a twisted universal sheaf $\G$ over $S\times U$. 
Then $E$ contains a Zariski dense open subset isomorphic to 
the projectivization of $\G$. $E$ is irreducible, since 
 the moduli space $M_H(u)$ is irreducible. 
We conclude that $E$ is a prime exceptional divisor, since it is contracted
by the morphism to the Uhlenbeck-Yau compactification 
(also by Theorem \ref{thm-a-uniruled-divisor-is-exceptional}).
Furthermore, $E$ is of type $A_1$. 
 
We calculate
next the class $[E]^\vee\in H_2(M,\Integers)$, given in equation
(\ref{eq-E-vee}). 
Recall that $[E]^\vee$ is the class of the fiber of $E\rightarrow Y$,
by Corollary \ref{cor-1}.
Fix an $H$-slope-stable locally free sheaf $G$ of class 
$u\in K(S)$. 
Fix a point $P\in S$  and denote by $G_P$ the fiber of $G$ at $P$. 
Let $\PP{G}_P$ be the projectivization of the fiber 
and denote by
\[
\kappa \ : \ \PP{G}_P \ \ \ \longrightarrow \ \ \ M_H(v)
\]
the morphism given in Lemma 
\ref{lemma-class-of-a-contractible-rational-curve}.
Then $\kappa(\PP{G}_P)$ is a fiber of $E\rightarrow Y$, and we get
the equality 
$
[\kappa(\PP{G}_P)] = [E]^\vee
$
of classes in $H_2(M,\Integers)$.
%
We conclude that the following equalities hold, for all $x\in v^\perp$. 
\begin{equation}
\label{eq-integral-equal-minus-rank-and-expressed-in-term-of-w}
\int_{[E]^\vee}\theta(x)=
-\rank(x)=\frac{1}{n-1}(w,x),
\end{equation}
where $w=(2,\LB,n-b-1)$, by Lemma
\ref{lemma-class-of-a-contractible-rational-curve}.

The class $[E]$ is the unique class in $H^2(M,\Integers)$
satisfying the equality
\[
\frac{-2([E],\theta(x))}{([E],[E])} \ \ = \ \ \int_{[E]^\vee}\theta(x),
\]
for all $x\in v^\perp$, by Corollary \ref{cor-1}.
Now $(w,w)=2-2n$ and so the equality
$[E]=\nolinebreak\theta(w)$ follows from 
equation (\ref{eq-integral-equal-minus-rank-and-expressed-in-term-of-w}). 
Hence, $([E],[E])=(\theta(w),\theta(w))=2-2n$.
The class $[E]$ is primitive, since $w$ is.

Let us calculate $\div(e,\bullet)$. 
The class $x=(\rho,\LB',\sigma)$ belongs to $v^\perp$, if and only if
$b\rho=2\sigma-(\LB,\LB')$. 
Hence, $(e,\theta(x))=(w,x)=(1-n)\rho$. 
If $c_1(\LB)$ is divisible by $2$, then every integral class $x\in v^\perp$ 
has even rank $\rho$, and so $\div(e,\bullet)=2n-2$.
If $c_1(\LB)$ is not divisible by $2$, 
choose a line bundle $\LB'$,
such that $(\LB,\LB')$ is odd and set 
$\sigma:=[(\LB,\LB')+b]/2$. 
Then $(1,\LB',\sigma)$ belongs to $v^\perp$.
Hence, $\div(e,\bullet)=n-1$.

The pair $(\widetilde{L},e)$, given in equation (\ref{eq-saturation-of-L}),
may be chosen to consist of the
saturation $\widetilde{L}$ in $K(S)$ of the lattice spanned by the classes
$v$ and $w=\theta^{-1}(e)$, by Theorem 
\ref{thm-item-orbit-of-inverse-of-Mukai-isom-is-natural}.
The largest integer dividing $w-v=(0,0,n-1)$ is $\sigma:=n-1$.
Now $w+v=(4,2\LB,n-1-2b)$. The largest integer $\rho$
dividing $w+v$ is $4$, if $c_1(\LB)$ is divisible by $2$.
Otherwise, $\rho=1$, if $n$ is even, and $\rho=2$, if $n$ is odd.
The invariant $rs(e)$ is then calculated via the table after Lemma 
 \ref{lemma-isometry-orbits-in-rank-2}.
%
\end{proof}

%
\section{Examples of non-effective monodromy-reflective classes}
\label{sec-non-effective}

We provide examples of monodromy-reflective classes, which are not $\RationalNumbers$-effective. 
Observation \ref{observation-not-Q-effective} guides us to lift these reflections to birational self-maps. 
Let us first prove the observation.
\begin{proof} (of Observation \ref{observation-not-Q-effective})
There exists a Zariski open subset $U\subset X$, 
such that $X\setminus U$ has codimension $\geq 2$, 
$\iota$ restricts to a regular involution of $U$, and
the composition
$
H^2(X,\Integers)\cong H^2(U,\Integers)\LongRightArrowOf{\iota^*}
H^2(U,\Integers)\cong H^2(X,\Integers)
$
is an isometry, by \cite{ogrady-weight-two}, Proposition 1.6.2.  
The isometry $\iota^*$ is assumed to be the reflection $R_{e}$. 
Hence, $\iota^*L\cong L^{-1}$ and $L$ is not $\RationalNumbers$-effective.
\end{proof}

Let $S$ be a projective $K3$ surface with a cyclic Picard group generated by
an ample line bundle $H$. Set $h:=c_1(H)\in H^2(S,\Integers)$.
Set $d:=\deg(H)/2$.

%
\subsection{Non-effective classes of divisibility $\div(e,\bullet)=2n-2$}
\hspace{1ex}\\

Let $r$, $s$ be integers satisfying $s>r>2$ and $\gcd(r,s)=1$.
Set $n:=rs+1$. 
Set $v:=(r,0,-s)$, $e:=\theta(r,0,s)$, and $M:=M_H(v)$.
$M$ is smooth and projective of dimension $2n$.
Let $\LB\in\Pic(M)$ be the line bundle with class $e$. 
Let $Exc\subset M$ be the locus of $H$-stable sheaves
of class $v$, which are not locally free or not
$H$-slope-stable. $Exc$ is a closed subset of 
co-dimension $\geq 2$ in $M$, by 
Lemma \ref{lem-codimension-of-Exc}.
Let $M^0$ be  the complement $M\setminus Exc$ of $Exc$
and $\eta:M^0\rightarrow M$ the inclusion.
The restriction homomorphism
$\eta^*:H^2(M,\Integers)\rightarrow H^2(M^0,\Integers)$ is an isomorphism.
Let $\phi:M^0\rightarrow M^0$ be the involution
sending a point representing the sheaf $F$, to the point representing 
$F^*$. Set $\psi:=(\eta^*)^{-1}\circ \phi^*\circ \eta^*$.

\begin{prop}
\label{prop-vanishing-in-divisibility-2n-2}
\begin{enumerate}
\item
\label{lemma-item-clear}
The class $e$ is monodromy-reflective, $(e,e)=2-2n$, 
$\div(e,\bullet)=2n-2$, and $rs(e)=\{r,s\}$.
\item
\label{lemma-item-R-e-is-minus-dualization}
Let $R_e:H^2(M,\Integers)\rightarrow H^2(M,\Integers)$ be the reflection by 
$e$. Then $R_e(\theta(\lambda))=-\theta(\lambda^\vee)$, for all
$\lambda\in v^\perp$.
\item
\label{lemma-item-psi-is-R-e}
$\psi = R_e.$
\item
\label{lemma-item-vanishing-in-divisibility-2n-2}
$H^0(M,\LB^k)$ vanishes, for all non-zero integral powers $k$.
\end{enumerate}
\end{prop}

\begin{proof}
Part \ref{lemma-item-clear} was proven in Example
\ref{example-any-factorization-rs-is-possible}. 
Set $\tilde{e}:=(r,0,s)\in v^\perp$. Part
\ref{lemma-item-R-e-is-minus-dualization} follows from  
the fact that $\theta:v^\perp\rightarrow H^2(M,\Integers)$ is an isometry,
and the equality $R_{\tilde{e}}(\lambda)=-\lambda^\vee$, for all
$\lambda\in v^\perp$. 
Part \ref{lemma-item-vanishing-in-divisibility-2n-2} follows from
part \ref{lemma-item-psi-is-R-e}, 
via Observation \ref{observation-not-Q-effective}.
We proceed to prove part \ref{lemma-item-psi-is-R-e}.
We need to prove the equality 
$\phi^*(\eta^*(y))=\eta^*R_e(y)$, 
for all $y\in H^2(M,\Integers)$.

Let $\pi_i$ be the projection from $S\times M^0$ onto the $i$-th factor, 
$i=1,2$. Let $\F$ be a universal sheaf over $S\times M$, $\G$
its restriction to $S\times M^0$, and $[\G]$ its class
in $K(S\times M^0)$.
The morphism $\phi:M^0\rightarrow M^0$ satisfies
\[
(id\times \phi)^!\G \cong (\G\otimes \pi_2^*A)^*,
\]
for some line bundle $A\in \Pic(M^0)$. 
We have the commutative diagram
\[
\begin{array}{ccc}
K(S\times M)&\LongRightArrowOf{(id\times \eta)^!}& K(S\times M^0)
\\
\pi_{2_!} \ \downarrow \ \hspace{2ex} & & 
\hspace{2ex} \ \downarrow \ \pi_{2_!} 
\\
K(M) & \LongRightArrowOf{\eta^!} & K(M^0),
\end{array}
\]
by the K\"{u}nneth Theorem \cite{atiyah-book}.
Hence, 
\[
\eta^*\theta(x)=
c_1\left[\pi_{2_!}\left(\pi_1^!(x^\vee)\otimes[\G]\right)\right],
\] 
for all $x\in v^\perp\subset K(S)$. This explains the first
equality below.
\begin{eqnarray*}
\phi^*(\eta^*\theta(x))&=&
c_1\left\{\phi^!\pi_{2_!}\left(\pi_1^!(x^\vee)\otimes[\G]\right)\right\}
\\
&=&
c_1\left\{\pi_{2_!}
\left(\pi_1^!(x^\vee)\otimes(id\times \phi)^![\G]\right)\right\}
\\
&=&
c_1\left\{\pi_{2_!}\left(\pi_1^!(x^\vee)\otimes([\G]\otimes \pi_2^![A])^\vee
\right)\right\}
\\
&=&
-c_1\left\{\pi_{2_!}\left(\pi_1^!(x)\otimes([\G]\otimes \pi_2^![A])
\right)\right\}
\\
&=&
\eta^*\theta(-x^\vee) \ \ = \ \ \eta^*(R_e(\theta(x))).
\end{eqnarray*}
The fourth equality follows from Grothendieck-Verdier duality,
the fifth is due to the fact that $\theta$ is independent of 
the choice of a universal sheaf, and the sixth follows from part
\ref{lemma-item-R-e-is-minus-dualization}.
%
\end{proof}
%
\subsection{Non-effective classes of divisibility $\div(e,\bullet)=n-1$}
\hspace{1ex}\\
Let $r$ be a positive odd integer, 
$\sigma$ a positive integer, 
and set $n:=r\sigma+1$. Assume that $r\geq 3$, $\sigma\geq 3$, 
and  $\gcd(r,\sigma)=1$. 
Let $S$ be a $K3$ surface with a cyclic Picard group generated by 
an ample line bundle $H$. Set $d:=\deg(H)/2$. 
Choose $(S,H)$, so that $\sigma$ and $d$ have the same parity.
If $d$ is odd, assume  that $\sigma>r$, possibly after interchanging 
$r$ and $\sigma$.
Set $h:=c_1(H)$ and
\[
v \ \ \ := \ \ \ (2r,rh,-b),
\]
where $b:=[\sigma-rd]/2$.
Note that $\gcd(r,b)=\gcd(r,\sigma)=1$. 
Hence, $v$ is a primitive class in $K(S)$, $(v,v)=2n-2$,
and the moduli space $M_H(v)$ is smooth and projective of type 
$K3^{[n]}$. 
Let $Exc\subset M_H(v)$ be the locus parametrizing sheaves $F$,
which are not locally free or not $H$-slope-stable. 
$Exc$ is an algebraic subset of co-dimension
$\geq 2$ in $M_H(v)$, by Lemma 
\ref{lemma-Exc-has-codimension-at-least-2-again}.

Let $M^0$ be the complement $M\setminus Exc$ and $\eta:M^0\rightarrow M$
the inclusion. Let $\phi:M^0\rightarrow M^0$
be the involution sending a point $[F]$, representing 
the sheaf $F$, to the point representing 
$F^*\otimes H$. The homomorphism 
$\eta^*:H^2(M,\Integers)\rightarrow H^2(M^0,\Integers)$ is an isomorphism,
by Lemma \ref{lemma-Exc-has-codimension-at-least-2-again}.
Set $\psi:=(\eta^*)^{-1}\circ \phi^*\circ \eta^*$. 

Set $e:=(2r,rh,\sigma-b)$. 
Note that $\gcd(r,\sigma-b)=\gcd(r,2\sigma-2b)=\gcd(r,\sigma+rd)=
\gcd(r,\sigma)=1$.
Hence, $e$ is a primitive class in $v^\perp$ of degree $(e,e)=2-2n$.

\begin{lem}
\label{lemma-Type-B-divisibility-n-1}
\hspace{1ex}
\begin{enumerate}
\item
\label{lemma-item-e-is-monodromy-reflective-of-divisibility-n-1}
The class $\theta(e)$ is monodromy-reflective and 
$\div(\theta(e),\bullet)=n-1$.  
\item
\label{lemma-item-s-versus-sigma}
$\widetilde{L}$ and $rs(e) \ \ = \ \ \left\{
\begin{array}{cl}
H_{ev} \ \mbox{and} \ \{r,\sigma\} & \mbox{if} \ \sigma \ \mbox{is odd}
\ (n \ \mbox{even}),
\\
U(2) \ \mbox{and} \ \{r,\sigma/2\} & \mbox{if} \ \sigma \ \mbox{is even}
\ (n \ \mbox{odd}).
\end{array}
\right.
$
\item
\label{lemma-item-R-e-dualize-and-then-tensorize-with-H}
Let $R_e:v^\perp\rightarrow v^\perp$ be the reflection by $e$.
Then 
$
R_e(\lambda) =  -[\lambda^\vee]\otimes\nolinebreak H,
$
for all $\lambda\in v^\perp$. 
\item
\label{lemma-item-psi-equal-R-e-once-again}
$\psi=R_e$.
\item
\label{lemma-item-e-of-type-B-again}
Let $\LB$ be the line bundle on $M_H(v)$ with $c_1(\LB)=\theta(e)$. Then
$H^0(M_H(v),\LB^k)$ vanishes, for all non-zero integral powers $k$.
\end{enumerate}
\end{lem}

\begin{proof}
\ref{lemma-item-e-is-monodromy-reflective-of-divisibility-n-1})
Let $\lambda:=(x,c,y)\in K(S)$. Then $\lambda$ belongs to $v^\perp$,
if and only if 
$
r(h,c)-2ry+bx=0.
$ 
In particular, $x$ is divisible by $r$, since $\gcd(r,b)=1$. Now
$(\lambda,e)=(\lambda,v)-x\sigma=-x\sigma$.
Thus $(e,\lambda)$ is divisible by $r\sigma=n-1$. 
There exists a class $c\in H^2(S,\Integers)$, satisfying
$(c,h)=-b$, since the class $h$ is primitive and $H^2(S,\Integers)$
is unimodular. Then the class $\lambda:=(r,c,0)$ belongs to $v^\perp$, 
and $(e,\lambda)=-r\sigma=1-n$.
Hence, $\div(\theta(e),\bullet)=n-1$.
The class $\theta(e)$ is monodromy-reflective, by Proposition
\ref{prop-reflection-by-a-numerically-prime-exceptional-is-in-Mon}.

\ref{lemma-item-s-versus-sigma})
If $\sigma$ is odd, then $n=r\sigma+1$ is even, and
$\widetilde{L}\cong H_{ev}$, by Lemma 
\ref{lem-non-unimodular-rank-two-lattice}.
Now $(e-v)/\sigma=(0,0,1)$ is primitive. Hence, $rs(e)=\{r,\sigma\}$,
by Lemma \ref{lemma-isometry-orbits-in-rank-2}.
If $\sigma$ is even, then $n$ is odd and $d$ is even. 
The classes  $\alpha:=(e-v)/\sigma=(0,0,1)$ 
and $\beta:=(e+v)/2r=(2,h,d/2)$ are integral isotropic classes
and $(\alpha,\beta)=-2$. Hence, $\{\alpha,\beta\}$ 
spans the primitive sublattice 
$\widetilde{L}\cong U(2)$. Consequently, $rs(e)=\{r,\sigma/2\}$,
by Lemma \ref{lemma-isometry-orbits-in-rank-2}.

Part \ref{lemma-item-R-e-dualize-and-then-tensorize-with-H} is verified
by a straightforward calculation. 
Part \ref{lemma-item-psi-equal-R-e-once-again} follows from 
part \ref{lemma-item-R-e-dualize-and-then-tensorize-with-H} by
the same argument used in the proof of Proposition
\ref{prop-vanishing-in-divisibility-2n-2}.
Part \ref{lemma-item-e-of-type-B-again} follows from part 
\ref{lemma-item-psi-equal-R-e-once-again}, by Observation 
\ref{observation-not-Q-effective}.
\end{proof}

\begin{example}
\label{example-non-effective-divisibility-n-1-and-n-is-cong-1-mod-8}
Let $r$ and $s$ be positive integers satisfying $s>r$, one of $r$ or $s$
is even, and $\gcd(r,s)=1$. 
Set $n=4rs+1$. Note that $n\equiv 1$ (modulo $8$).
Let $S$ be a $K3$ surface with a cyclic Picard  group generated by an ample
line bundle $H$. Set $h:=c_1(H)$ and
$d:=(h,h)/2$. Assume that $d$ is odd.
Then $s+rd$ is odd, since $r$ and $s$ consist of one odd and one even integer,
by assumption. 
Set
$v:=(4r,2rh,-s+rd)$ and $e:=(4r,2rh,s+rd)$. 
Then the classes $v$ and $e$ are primitive, 
$(e,v)=0$,  $(v,v)=2n-2$ and $(e,e)=2-2n$.
We have $(e+v)/2r=(4,2h,d)$ and $(e-v)/2s=(0,0,1)$.
We claim that $\div(e,\bullet)=n-1$.
Indeed, if $\lambda=(x,c,y)$ belongs to $v^\perp$,
then $2r$ divides $x$. Hence, $(e,\lambda)=(v,\lambda)-2sx=-2sx$ is divisible
by $4rs=n-1$. Furthermore, let $c$ be a class in $H^2(S,\Integers)$ satisfying $(c,h)=rd-s$,
and set $\lambda=(2r,c,0)$. Then $(\lambda,e)=1-n$.
Hence, $\div(e,\bullet)=n-1$.
Hence, $\widetilde{L}=H_{ev}$, by Proposition
\ref{prop-isometry-class-of-tilde-L-e-is-a-faithful-mon-invariant}. 
Thus $rs(e)=\{r,s\}$, by Lemma \ref{lemma-isometry-orbits-in-rank-2}.
Let $Exc\subset M_H(v)$ be the locus parametrizing sheaves which are 
not locally-free or not $H$-slope-stable. $Exc$ is an algebraic subset 
of co-dimension $\geq 2$ in $M_H(v)$, 
by the proof of the odd $d$ case of  
Lemma \ref{lemma-Exc-has-codimension-at-least-2-again}
(the symbol $r$ in that proof should remain $\rank(v)/2$ 
and so should be set equal to twice the symbol $r$ above,
and similarly the symbol $s$ in that proof should be set equal to
twice the symbol $s$ above).
Let $\LB$ be the line bundle over $M_H(v)$ with 
$c_1(\LB)=\theta(e)$. 
We conclude that $H^0(M_H(v),\LB^k)$ vanishes, for all non-zero integers $k$,
by the same argument used to prove Lemma \ref{lemma-Type-B-divisibility-n-1}.
\end{example}

\bigskip

{\bf Acknowledgements:} This work was influenced by S. Druel's talk 
at the workshop ``Holomorphically symplectic varieties and moduli spaces'',
which took place at the Universite des Sciences et Technologies de Lille,
in June 2009.
I thank  S. Druel for his clear exposition of his 
interesting work on the divisorial Zariski decomposition \cite{druel},
and for correcting inaccuracies in an earlier draft of this paper.
I would like to express my gratitude to the organizers, 
Dimitri Markushevich and Valery Gritsenko, for the invitation to
participate in the workshop, their hospitality, 
and for their wonderful work organizing
this stimulating workshop.
The paper is influenced by the work of Brendan Hassett
and Yuri Tschinkel \cite{hassett-tschinkel-conj,hassett-tschinkel}.
I thank them for communicating to me a 
draft of their recent preprint on the subject
\cite{hassett-tschinkel-monodromy}. 
Section \ref{sec-deformation-equivalence}
elaborates on ideas found already in \cite{hassett-tschinkel-monodromy}. 
I thank Brendan also for reading an early draft of the 
proof of Theorem \ref{thm-2}. 
I thank Misha Verbitsky for answering my questions about his preprint
\cite{verbitsky}. I thank Valery Gritsenko for pointing out reference
\cite[Corollary 3.4]{GHS-K3}.
I thank Kota Yoshioka for pointing out the relationship between sections 2 and 3 of his paper 
\cite{yoshioka-irreducibility}, and section \ref{sec-sufficient-conditions} above.



\begin{thebibliography}{B-N-R}
{\scriptsize
\bibitem[At]{atiyah-book} Atiyah, M. F.:
{\em $K$-theory.\/} 
Lecture notes by D. W. Anderson. 
W. A. Benjamin, Inc., New York-Amsterdam 1967. 

\bibitem[BCHM]{BCHM}
 Birkar, C.; Cascini, P.; Hacon, C.; McKernan, J.:
{\em Existence of minimal models for varieties of log general type.\/}
J. Amer. Math. Soc. 23 (2010), no. 2, 405--468.

\bibitem[Be1]{beauville}
Beauville, A.: {\em Varietes K\"ahleriennes dont la premiere classe de Chern 
est nulle.}  J. Diff. Geom. 18, p. 755--782 (1983).

\bibitem[Be2]{beuville-symplectic-singularities}
Beauville, A.: {\em Symplectic singularities.\/} 
Invent. Math.  139  (2000),  no. 3, 541--549.

\bibitem[BHPV]{BHPV} Barth, W., Hulek, K., Peters, C., and Van de Ven, A.:
{\em Compact Complex Surfaces.\/} Second edition, 
Springer-Verlag, 2004. 

\bibitem[Bou]{boucksom} Boucksom, S.:
{\em Divisorial Zariski decompositions on compact complex manifolds.\/}  
Ann. Sci. École Norm. Sup. (4)  37  (2004),  no. 1, 45--76. 


\bibitem[D]{druel} Druel, S.: 
{\em Quelques remarques sur la decomposition de Zariski divisorielle sur
   les varietes dont la premire classe de Chern est nulle.\/}
Math. Z. 267 (2011), no. 1-2, 413--423.



\bibitem[GHS]{GHS-K3} Gritsenko, V. A., Hulek, K., Sankaran, G. K.: 
{\em The Kodaira dimension of the moduli of K3 surfaces.\/} 
Invent. Math. 169 (2007), no. 3, 519--567.

\bibitem[HL]{huybrechts-lehn-book}
Huybrechts, D, Lehn, M.: 
{\em The geometry of moduli spaces of sheaves.\/} 
Aspects of Mathematics, E31. Friedr. Vieweg \& Sohn, Braunschweig, 1997.

\bibitem[HT1]{hassett-tschinkel-conj} Hassett, B., Tschinkel, Y.:
{\em Rational curves on holomorphic symplectic fourfolds.\/}
Geom. Funct. Anal. 11,  no. 6, 1201--1228, (2001). 

\bibitem[HT2]{hassett-tschinkel} Hassett, B., Tschinkel, Y.:
{\em Moving and ample cones of holomorphic symplectic fourfolds.\/}
Geom. Funct. Anal. 19, no. 4, 1065--1080, (2009).

\bibitem[HT3]{hassett-tschinkel-monodromy} Hassett, B., Tschinkel, Y.:
{\em Monodromy and rational curves on holomorphic symplectic fourfolds.\/}
Preprint 2009. 

\bibitem[Hu1]{huybrects-basic-results}
Huybrechts, D.: 
{\em Compact Hyperk\"{a}hler Manifolds: Basic results.\/}
Invent. Math. 135 (1999), no. 1, 63-113 and
Erratum: Invent. Math. 152 (2003), no. 1, 209--212. 

\bibitem[Hu2]{huybrechts-kahler-cone}
Huybrechts, D.: 
{\em  The K\"{a}hler cone of a compact hyperk\"{a}hler manifold.\/}
Math. Ann. 326 (2003), no. 3, 499--513.


\bibitem[Hu3]{huybrechts-norway}
Huybrechts, D.: {\em Compact hyperk\"{a}hler manifolds.\/}
 Calabi-Yau manifolds and related geometries (Nordfjordeid, 2001),  
161--225, Universitext, Springer, Berlin, 2003. 

\bibitem[Hu4]{huybrechts-bourbaki}
Huybrechts, D.: {\em A global Torelli theorem for hyperk\"{a}hler manifolds (after Verbitsky).\/}
S\'{e}minaire Bourbaki, 63 ann\'{e}e, 2010--2011, no. 1040.

\bibitem[Ka]{kawamata} Kawamata, Y.: 
{\em Unobstructed deformations. II,\/} J. Algebraic Geom. 4 (1995) 277--279.
See also Erratum: J. Algebraic Geom. 6 (1997), no. 4, 803--804.

\bibitem[KLS]{kaledin-lehn-sorger} Kaledin, D.; Lehn, M.; Sorger, Ch.: 
{\em Singular symplectic moduli spaces.\/}  Invent. Math.  164  (2006),  no. 3, 591--614.

\bibitem[KM]{kollar-mori} Koll\'{a}r, J., Mori, S.:
{\em Classification of three-dimensional flips.\/}  
J. Amer. Math. Soc.  5  (1992),  no. 3, 533--703.



\bibitem[Le]{le-potier-coherent} Le Potier, J.:
{\em Syst\'{e}mes coh\'{e}rents et structures de niveau.\/}
Ast\'{e}risque 214, 1993

\bibitem[Li]{jun-li} Li, J.:
{\em Algebraic geometric interpretation of Donaldson's
polynomial invariants of algebraic surfaces.\/}
J. Diff. Geom. 37 (1993), 417--466.

\bibitem[Ma1]{markman-reflections} Markman, E.:
{\em Brill-Noether duality for moduli spaces of sheaves on K3 surfaces.\/}
J. of Alg. Geom. {\bf 10} (2001), no. 4, 623--694.

\bibitem[Ma2]{markman-monodromy-I} Markman, E.:
{\em On the monodromy of moduli spaces of sheaves on 
K3 surfaces.\/}
J. Algebraic Geom. {\bf 17}  (2008), 29--99. 

\bibitem[Ma3]{markman-part-two} Markman, E.: 
{\em On the monodromy of moduli spaces of sheaves on K3 surfaces II.\/} 
Preprint, math.AG/0305043 v4.

\bibitem[Ma4]{markman-integral-generators} Markman, E.:
{\em Integral generators for the cohomology ring of moduli spaces of 
sheaves over Poisson surfaces.\/} 
Adv. in Math. 208 (2007) 622--646.

\bibitem[Ma5]{markman-constraints} Markman, E.:
{\em Integral constraints on 
the monodromy group of the hyperk\"{a}hler 
resolution of a symmetric product of a $K3$ 
surface.\/}  Internat. J. of Math. 21 (2010), no. 2, 169--223.

\bibitem[Ma6]{markman-galois} Markman, E.:
{\em Modular Galois covers associated to symplectic resolutions of 
singularities.\/}  J. Reine Angew. Math. 644 (2010), 189--220.

\bibitem[Ma7]{markman-torelli} Markman, E.:
{\em  A survey of Torelli and monodromy results for hyperkahler manifolds.\/}
In  ``Complex and Differential Geometry'', W. Ebeling et. al. (eds.),
Springers Proceedings in Math. 8, (2011), pp 257--323.
Refereed. 

\bibitem[Mat]{matsumura} Matsumura, H.: 
{\em Commutative Algebra.\/} W. A. Benjamin Co., New York (1970)


\bibitem[Mu]{mukai-hodge} Mukai, S.: 
{\em On the moduli space of bundles on K3 surfaces I},
Vector bundles on algebraic varieties,
Proc. Bombay Conference, 1984, Tata Institute of Fundamental Research Studies,
no. 11, Oxford University Press, 1987, pp. 341--413.

\bibitem[Na1]{namikawa} Namikawa, Y.:
{\em Deformation theory of singular symplectic $n$-folds.\/} 
Math. Ann.  319  (2001),  no. 3, 597--623. 

\bibitem[Na2]{namikawa-galois} Namikawa, Y.: 
{\em Poisson deformations of affine symplectic varieties II.\/}
Kyoto J. Math. 50 (2010), no. 4, 727--752.

\bibitem[Ni]{nikulin} Nikulin, V. V.: 
{\em Integral symmetric bilinear forms and some of their applications.\/}
Math. USSR Izvestija, Vol. 14 (1980), No. 1

\bibitem[O]{oguiso} Oguiso, K.:
{\em K3 surfaces via almost-primes. \/}
Math. Res. Lett. 9 (2002), no. 1, 47--63. 

\bibitem[OG1]{ogrady-weight-two} O'Grady, K.: 
{\em The weight-two Hodge structure of moduli spaces of sheaves on a K3 
surface.\/} J. Algebraic Geom. 6 (1997), no. 4, 599--644.

\bibitem[OG2]{ogrady-10} O'Grady, K.:
{\em Desingularized moduli spaces of sheaves on a $K3$\/.}  
J. Reine Angew. Math.  512  (1999), 49--117. 


\bibitem[R1]{ziv-ran} Ran, Z.:
{\em Hodge theory and the Hilbert scheme.\/}
J. Differential Geom. 37 (1993), no. 1, 191--198.


\bibitem[Ver]{verbitsky} Verbitsky, M.:
{\em A global Torelli theorem for hyperkahler manifolds.\/} 
Preprint arXiv:0908.4121 v7.

\bibitem[Voi]{voisin-book-vol1} Voisin, C.:
{\em Hodge Theorey and Complex Algebraic Geometry I.\/}
Cambridge studies in advanced mathematics 76, Cambridge Univ. Press (2002).

\bibitem[W]{wierzba} Wierzba, J.:
{\em Contractions of symplectic varieties.\/}  
J. Algebraic Geom.  12  (2003),  no. 3, 507--534. 


\bibitem[Y1]{yoshioka-brill-noether} Yoshioka, K.:
{\em Some examples of Mukai's reflections on $K3$ surfaces.\/} 
 J. Reine Angew. Math.  515  (1999), 97--123.

\bibitem[Y2]{yoshioka-abelian-surface} Yoshioka, K.:
{\em 
Moduli spaces of stable sheaves on abelian surfaces. \/
}
Math. Ann. 321 (2001), no. 4, 817--884.

\bibitem[Y3]{yoshioka-irreducibility} Yoshioka, K.:
{\em Irreducibility of moduli spaces of vector bundles on $K3$ surfaces.\/} 
Electronic preprint arXiv:math/9907001.

}
\end{thebibliography}
\end{document}